\documentclass{article}
\pdfoutput=1
\usepackage[numbers,sort]{natbib}
\usepackage[utf8]{inputenc}
\usepackage{fullpage}
\usepackage{amsmath,amsbsy,amsgen,amscd,amssymb,amsthm,amsfonts,stmaryrd}
\usepackage{mathtools}
\usepackage{url}
\DeclareMathOperator{\prox}{prox}
\DeclareMathOperator{\dist}{dist}
\DeclareMathOperator{\Gap}{Gap}
\DeclareMathOperator{\dom}{dom}
\DeclareMathOperator{\diag}{diag}
\DeclareMathOperator{\nnz}{nnz}
\usepackage[colorlinks,citecolor=blue]{hyperref}
\usepackage{algorithm,algorithmic}
\usepackage{cleveref}
\usepackage{caption}

\newtheorem{assumption}{Assumption}
\newtheorem{example}{Example}
\newtheorem{theorem}{Theorem}
\newtheorem{lemma}[theorem]{Lemma}
\newtheorem{corollary}[theorem]{Corollary}

\theoremstyle{definition}
\newtheorem{remark}[theorem]{Remark}

\numberwithin{equation}{section}
\numberwithin{theorem}{section}

\newcommand{\R}{\mathbb{R}}
\newcommand{\Tau}{\mathrm{T}}
\newcommand{\Beta}{\mathrm{B}}
\newcommand{\E}{\mathbb{E}}
\newcommand{\xout}{x_{\text{out}}}
\newcommand{\yout}{y_{\text{out}}}
\usepackage{xcolor}

\title{On the Complexity of a Practical Primal-Dual Coordinate Method}
\author{Ahmet Alacaoglu \\ UW-Madison \\ \url{alacaoglu@wisc.edu} \and Volkan Cevher \\ EPFL \\ \url{volkan.cevher@epfl.ch} \and Stephen J. Wright \\ UW-Madison \\ \url{swright@cs.wisc.edu}}
\begin{document}

\maketitle
\begin{abstract}
We prove complexity bounds for the \emph{primal-dual algorithm with random extrapolation and coordinate descent (PURE-CD)}, which has been shown to obtain good practical performance for solving convex-concave min-max problems with bilinear coupling. Our complexity bounds either match or improve the best-known results in the literature for both dense and sparse (strongly)-convex-(strongly)-concave problems.
\end{abstract}
\section{Introduction}
We consider the convex-concave min-max problem with bilinear coupling based on the function $L\colon\R^d \times \R^n \to \R\cup\{-\infty, +\infty\}$ defined as follows:
\begin{equation}\label{eq:prob}
     \min_{x\in\mathbb{R}^d}\max_{y\in\mathbb{R}^n} \left\{L(x, y):= \sum_{i=1}^n \left[ \langle A_ix, y^{(i)} \rangle - h_i^\ast(y^{(i)}) \right] + g(x) = \langle Ax,y\rangle - h^*(y) + g(x) \right\},
\end{equation}
where $h_i^*\colon\R \to \R\cup\{+\infty\}$, $i=1,2,\dotsc,n$ and $g\colon\R^d \to \R\cup\{+\infty\}$ are convex and extended-valued, $y=(y^{(1)},y^{(2)}, \dotsc, y^{(n)})^\top$, $h^\ast(y) = \sum_{i=1}^n h_i^\ast(y^{(i)})$, $A_i$ is a row vector of length $d$, and $A$ is the $n \times d$ matrix whose rows are $A_i$, $i=1,2,\dotsc,n$.
We assume the existence of a primal-dual solution $(x_\star, y_\star)$ satisfying the saddle point property
\begin{equation} \label{eq:saddle}
L(x,y_\star) \ge L(x_\star,y_\star) \ge L(x_\star,y), \quad \mbox{for all $x \in \R^d$, $y \in \R^n$.}
\end{equation}
Under this assumption we can exchange $\min$ and $\max$ in~\eqref{eq:prob} and write the problem as
\begin{equation} \label{eq:prob2}
    \min_{x\in\R^d} \, F(x), \quad \mbox{where } \;\; F(x):= \max_{y \in \R^n} \, \left\{L(x,y) = \sum_{i=1}^n h_i(A_ix)+g(x) = h(Ax)+g(x)\right\},
\end{equation}
where $h_i$ is the convex conjugate of $h_i^\ast$, and
$h\colon\R^n \to \R\cup\{+\infty\}$ has the separable form $h(z) = \sum_{i=1}^n h_i(z^{(i)})$.
Problems of this form arise in machine learning, especially in
empirical risk minimization (ERM, see~\eqref{eq: case1_prob}), matrix games, and also in linearly constrained optimization (see~\eqref{eq: case2_prob}) as well as imaging.  Therefore, many methods have been developed for solving this problem under various assumptions on $h$, $g$, and $A$, such as (strong)-convexity of $g$, $h$, or $h^\ast$; or sparsity of $A$~\cite{chambolle2011first,chambolle2018stochastic,allen2017katyusha,zhang2015stochastic,alacaoglu2020random,song2021variance,tan2020accelerated,song2021coordinate,xiao2014proximal,shalev2013stochastic,shalev2014accelerated}.

The various methods that have been proposed for \eqref{eq:prob} have favorable complexity guarantees in certain special cases.  
The plethora of methods and results makes it difficult for both theoreticians and practitioners to choose the method best suited to particular instances of \eqref{eq:prob}.  
In this paper, 
we focus on improving the theory for an existing method, the PURE-CD algorithm described in \cite{alacaoglu2020random}.
We show that this method achieves or improves best-known complexity results for interesting special cases of \eqref{eq:prob}. 
The state-of-the-art results are currently dispersed around different methods.

To facilitate our discussion on complexity, we define the function $G: \R^d \times \R^n \times \R^d \times \R^n \to \R\cup\{-\infty, +\infty\}$ as follows:
\begin{equation}\label{eq: g_from_l}
G(x', y', x,y) := L(x', y) - L(x, y') = g(x') + \langle Ax', y \rangle - h^\ast(y) - g(x) -\langle Ax, y'\rangle + h^\ast(y').
\end{equation}
For a given compact set $\mathcal{Z}\subset \R^d \times \R^n$, we define the {\em gap function} as follows:
\[
\Gap (x',y'):= \max_{(x,y) \in \mathcal{Z}} \, G(x',y',x,y).
\]
The set $\mathcal{Z}$ is introduced to handle cases in which the domains for $x$ and $y$ are unbounded~\cite{chambolle2011first,chambolle2018stochastic,nesterov2007dual,nesterov2009primal}. 
(See~\cite[Lemma 4]{nesterov2009primal} for further details about $\mathcal{Z}$.)
In general, our goal in this paper is to find the total complexity of an algorithm to output a pair $(\xout, \yout)$ such that
\begin{equation}\label{eq: opt_cond_def}
\mathbb{E}\Gap(\xout, \yout) = \mathbb{E}\max_{(x, y)\in\mathcal{Z}}G(\xout, \yout, x, y) = \mathbb{E}\max_{(x, y)\in\mathcal{Z}} \, \left[ L(\xout, y) - L(x, \yout) \right] \leq \varepsilon.
\end{equation}
Because of the presence of $\mathcal{Z}$, we refer to \eqref{eq: opt_cond_def} as the  expected \emph{restricted} primal-dual gap function, or expected duality gap, for short.
For special cases of~\eqref{eq:prob} such as linearly constrained optimization or ERM, we also consider such optimality measures as objective suboptimality and feasibility.

\subsection{Context and Contributions} \label{sec:cc}

There are several reasons for our focus on PURE-CD. 
Its  algorithmic structure is simple and it adapts efficiently to the case in which $A$ is sparse; unlike the methods in~\cite{zhang2015stochastic,tan2020accelerated,song2021coordinate} it does not require any special implementation techniques to exploit sparsity (such as lazy updates). 
Its appealing empirical performance is documented in~\cite{alacaoglu2020random,song2021coordinate}.
It has favorable theoretical properties such as almost-sure convergence in the general convex-concave setting and adaptive linear convergence under metric subregularity; these were shown in the paper~\cite{alacaoglu2020random}, which introduced the method.
Such properties do not hold for many competing algorithms~\cite{song2021coordinate,song2021variance,tan2020accelerated,allen2017katyusha,zhang2015stochastic,shalev2013stochastic}.
Unfortunately, the complexity bounds proven in~\cite{alacaoglu2020random} for PURE-CD do not improve over deterministic algorithms.

Before we proceed, a little terminology is needed. We say that a convex function $g$ is ``$\mu$-strongly convex" or ``$\mu$-s.c." for some $\mu \ge 0$ when $g(x) \geq g(y) + \langle q_y, x-y \rangle + \frac{\mu}{2} \| x-y\|^2$ for all $x$ and all $q_y \in \partial g(y)$.
We say that \eqref{eq:prob} is ``convex-concave" in the basic case in which $g$ and $h^\ast$ are convex.
We describe \eqref{eq:prob} as ``convex-strongly concave" when, additionally, $h^\ast$ is $\mu$-strongly convex with $\mu>0$. The terms "strongly convex-concave" and ``strongly convex-strongly concave" are defined accordingly. 
The latter three terms are collective referred to as the ``strongly convex cases."
We also sometimes use terminology ``one-sided strong convexity" for strongly convex-concave and convex-strongly concave cases.

The contributions of this paper to the theory of PURE-CD and the problem \eqref{eq:prob} can be summarized as follows (see Sec~\ref{sec: sect_overview} for the details):
\begin{itemize}
  \setlength\itemsep{1mm}
\item For linearly constrained problems with dense $A$, we improve the best known complexity on objective value and feasibility from $\min\left\{O\left(nd\|A\|\varepsilon^{-1} \right), O\left(nd+ \sqrt{nd(n+d)}\|A\|_F\varepsilon^{-1} \right)\right\}$ (for PDHG~\cite{chambolle2011first} and variance reduced extragradient~\cite{alacaoglu2021stochastic}) to $O\left( nd+ d\sum_{i=1}^n \|A_i\|\varepsilon^{-1} \right)$.
\item For general convex-concave problems with dense $A$, we improve the best-known complexity on expected duality gap from $\min\left\{O\left(nd\|A\|\varepsilon^{-1} \right), O\left(nd+ \sqrt{nd(n+d)}\|A\|_F\varepsilon^{-1} \right)\right\}$ (for PDHG~\cite{chambolle2011first} and variance reduced extragradient~\cite{alacaoglu2021stochastic}) to $\tilde{O}\left( nd+ nd \max_i \|A_i\|\varepsilon^{-1} \right)$. 
This improvement is strict when rows of $A$ are normalized. 
We make use of several techniques from~\cite{song2021variance} that proved the same complexity for the weaker measure $\max_{(x,y)} \mathbb{E}G(\xout, \yout, x,y)$ (see \Cref{sec: prelim} for the relations between these measures).
\item For empirical risk minimization problems with $L$-Lipschitz convex losses, non-strongly convex regularizers, and normalized feature vectors with with dense $A$, we improve the best-known complexity from $O\left( nd\log\varepsilon^{-1} + d\sqrt{n}L \varepsilon^{-1} \right)$ (see the accelerated variance reduction and primal-dual coordinate techniques of~\cite{allen2017katyusha,zhang2015stochastic,shalev2013stochastic,tan2020accelerated}) to $O\left( nd + d\sqrt{n}L \varepsilon^{-1} \right)$.
 To our knowledge, this is the first time that the logarithmic factor is removed for this problem, despite it being a well-studied case.
(See \Cref{app: upper_lower_bd} for details.)
\item For strongly convex-strongly concave problems with sparse $A$, we match the best known complexity of
$O\left( \left(\nnz(A)+ \nnz(A)\max_i\|A_i\|\sqrt{\mu_1\mu_2}\right)\log\varepsilon^{-1} \right)$ (see~\cite{tan2020accelerated}).
\item For strongly convex-concave or convex-strongly concave problems with sparse $A$, we improve the best-known complexity $\min\left\{ O\left( \nnz(A)\|A\|\mu^{-1}\varepsilon^{-1/2} \right), O\left( nd+nd\max_i\|A_i\|\mu^{-1}\varepsilon^{-1/2} \right) \right\}$ (for PDHG~\cite{chambolle2011first}, SPDHG~\cite{chambolle2018stochastic}, and VRPDA~\cite{song2021coordinate}) to $O\left( \nnz(A)+\nnz(A)\max_i\|A_i\|\mu^{-1}\varepsilon^{-1/2} \right)$, in the common regime in which $\max_i \|A_i\| \geq \mu$. This complexity has also been shown for the strongly convex-concave case in the recent work~\cite{song2021coordinate} .
\end{itemize}

In addition to the different methods achieving the best known rates in these cases (prior to this work), we note that some of the methods that are best in one case may have no guarantees for the other cases.  
Consider the following.
\begin{itemize}
\item The technique used in the analysis of VRPDA~\cite{song2021variance} led to improved
  complexity in the convex-concave case with dense $A$ for the quantity  $\max_{x, y} \mathbb{E} G(\xout, \yout, x, y)$ (see \Cref{sec: prelim}), and this method also has good complexity for the strongly convex-concave case with dense $A$.  
  On the other hand, this method does not have guarantees for the expected primal-dual gap $ \mathbb{E}\max_{x, y} G(\xout, \yout, x, y)$ or for the convex-strongly concave
  or strongly convex-strongly concave cases. Moreover, the output sequence for VRPDA does not have almost sure convergence guarantees in the convex-concave case, and it does not adapt to sparsity in $A$. 
\item SPDHG~\cite{chambolle2018stochastic} has good complexity with dense $A$ and in the strongly convex
  cases,  but it does not adapt to sparsity of $A$.
\item Although the method of \cite{tan2020accelerated} (and also~\cite{zhang2015stochastic} with more restrictions) can adapt to sparsity of $A$ for strongly convex-strongly concave problems, it does not have a guarantee for the general convex-concave case or linearly constrained problems, see Table~\ref{tb: 1}.
  \item The recent work~\cite{song2021coordinate} built on~\cite{song2021variance} to propose a new method that exploits sparsity and has improved guarantees in the convex-concave and strongly convex-concave cases.
  However, the guarantees for convex-concave case are for the quantity $\max_{x, y} \mathbb{E} G(\xout, \yout, x, y)$ (see \Cref{sec: prelim}) rather than for the expected duality gap. Additionally, the objective and feasibility guarantees are on the expectation of the output $\mathbb{E}[\xout]$ rather than the output iterates $\xout$ themselves. 
  This algorithm is not analyzed for the convex-strongly concave or strongly convex-strongly concave cases.
\item Variance reduced variational inequality (VRVI) methods
  \cite{carmon2019variance,alacaoglu2021stochastic} have competitive complexity bounds when $n
  \approx d$, but they are not appealing otherwise, since the per-iteration cost depends on $n+d$,  degrading the complexity, as we explain in~\Cref{sec: sect_overview}.
\end{itemize}

\subsection{Preliminaries}\label{sec: prelim}

We introduce a blanket assumption to ensure that $L$ defined in \eqref{eq:prob} is convex-concave and that strong duality holds.  
Additional assumptions are introduced as needed in the analysis.
\begin{assumption}\label{asmp: asmp1}
Let $g\colon \R^d \to (-\infty, +\infty]$, $h^\ast\colon\mathbb{R}^n
  \to (-\infty, +\infty] $ be proper convex lower semicontinuous and
    that a saddle point $(x_\star,y_\star)$ for \eqref{eq:prob} exists
    (but is not necessarily unique).
\end{assumption}
This assumption is standard in the literature for analyzing both stochastic and deterministic algorithms~\cite{chambolle2011first,alacaoglu2019convergence,chambolle2018stochastic,fercoq2019coordinate,song2021variance,carmon2019variance}. 
For the relationship between the existence of a saddle point, primal and dual solutions, and strong duality, we refer to~\cite[Cor. 19.19, 19.20]{bauschke2011convex},~\cite[Lem. 36.2]{rockafellar1970convex}.

Some works show convergence of the ``max of expectation" gap-like measure $\max_{(x,y)\in\mathcal{Z}} \mathbb{E} G(x', y', x,y)$.
However, results for the ``expectation of max"  or the ``expected gap" $\mathbb{E}\max_{(x,y)\in\mathcal{Z}} G(x', y', x,y) = \mathbb{E} \Gap (\xout,\yout)$ are more interesting, as the following example shows.
\begin{example}\label{ex: np3}
Consider the saddle-point problem defined by $L(x,y) =xy$, with $x, y \in \R$, which has the unique solution $(x_\star,y_\star)=(0,0)$. 
We have 
\begin{equation*}
\max_{(x, y) \in \mathcal{Z}}\mathbb{E} G(x_k, y_k, x, y) = \max_{(x, y)\in\mathcal{Z}}\mathbb{E} [x_k y - x y_k],
\end{equation*}
for any set $\mathcal{Z}$. Consider an algorithm that generates iterates $(x_k,y_k)$ with $x_k=y_k = k$
with probability (w.p.) $1/2$ and $x_k=y_k = -k$ w.p. $1/2$, for $k=1,2,\dotsc$.  
Then we have for any fixed $x$ and  $y$ that $\mathbb{E} [x_k y - x y_k]=\tfrac12 [k(y-x)+k(x-y)]=0$.
This seems to indicate acceptable behavior of the algorithm, although the iterates are diverging.  
On the other hand, the optimality measure $\mathbb{E} \max_{(x, y)\in\mathcal{Z}} G(x_k, y_k, x, y)$ is nonzero, as desired.
\end{example}

We show in sequel (\Cref{sec: lincons_erm}) that guarantees on expected gap can be transformed into guarantees in objective suboptimality and/or feasibility for specific problems.

\begin{remark}
In all the comparison tables, a dash ($-$) indicates that the result for the particular setting is not proven in the previous work.
\end{remark}

The complexity results that we discuss in the remainder of the paper depend on different norms of the matrix $A \in \R^{n \times d}$. 
The relationship between these norms is critical to determining when and in what circumstances the complexity of one method is superior to another.  
We summarize the equivalences between these norms here.
(Note that $\|A \|$ denotes the spectral norm $\|A\|_2$.)
\begin{subequations} \label{eq:Anorms}
  \begin{align}
    \|A \| &\le \|A\|_F \le \sqrt{r}\|A\|, \;\; \mbox{where $r$ is the rank of $A$,} \\
    \frac{1}{\sqrt{n}} \|A\|_F & \le \max_i \, \|A_i\| \le \|A\|_F, \\
    \frac{1}{\sqrt{n}} \|A\| & \le \max_i \, \|A_i\| \le \|A\|, \\
    \|A\|_F & \le \sum_{i=1}^n \|A_i\| \le \sqrt{n} \|A\|_F, \\
    \|A\| & \le \sum_{i=1}^n \|A_i\| \le n \|A\|.    
  \end{align}
\end{subequations}

\subsection{Overview of Results}\label{sec: sect_overview}
In the tables described here, the notation $O(\cdot)$ suppresses the term $\max_{x, y} \| x\|^2 + \| y\|^2$ for the measures $\mathbb{E}\max_{(x,y)} G(x', y', x,y)$ and $\max_{(x,y)} \mathbb{E}G(x', y', x,y)$. 
It suppresses $\| x_\star-x_0\|^2 + \|y_\star-y_0 \|^2$ for the other measures. 
(Since these quantities are common across the complexity bounds for all algorithms, they only complicate the comparisons.)

\paragraph{Convex-concave $L$ and dense $A$ (Table~\ref{tb: 1} and Section~\ref{sec: cvx_ccv_dense}).}

We first consider the the third column of~\Cref{tb: 1}, showing complexities on the expected gap.
Compared to variance reduced variational inequality (VRVI) methods~\cite{carmon2019variance}, coordinate methods are preferable when $n \gg d$, because in this case it is more likely that  $\sqrt{d}\max_i \|A_i\|\leq \|A\|_F$.  
When $d \gg n$, we can apply the coordinate methods in the dual, switching the roles of $n$ and $d$ and still obtain a better complexity than VRVI. (Note that since VRVI incurs a cost that depends on $n+d$ at each iteration and is invariant under switching.)  
The complexity of VRVI  methods is better when $A$  is dense and square ($n=d$), since $\|A\|_F \leq \sqrt{n}\max_i\|A_i\|$.
(VRVI might also be competitive when $A$ is {\em approximately} square.)

If the norms of the rows $A_i$ are normalized, then coordinate methods such as PURE-CD or VRPDA strictly improve variance reduced methods for the expected gap, since we then have $\|A_i\|= \frac{1}{\sqrt{n}}\|A\|_F$ for all $i\in[n]$.  
Finally, for the first two columns in~\Cref{tb: 1} (see the corresponding problems in~\eqref{eq: case1_prob},~\eqref{eq: case2_prob}), coordinate methods
have strictly better complexity thanks to importance sampling, since $\sum_{i=1}^n \|A_i\|\leq \sqrt{n}\|A\|_F$, by Cauchy-Schwarz and since $\sum_{i=1}^n \|A_i\| \leq n \max_i\|A_i\|$.
  
  We refer to \Cref{sec: prelim} for the comparison between the measures in the last two columns of \Cref{tb: 1}. We include the last column only to be able to compare with all the results in the literature, many of which apply only to this weaker ``max of expectation" measure.
  
\begin{table}[h]
\captionsetup{font=small}
\centering
\begin{tabular}{ l | l | l | l | l }
  & $\mathbb{E}F(\xout) - F(x_\star)^\ast$ & \begin{tabular}{@{}l@{}}$\mathbb{E}|g(\xout) - g(x_\star)|$ \\ $\mathbb{E}\dist(A\xout, C)^{\triangle}$\end{tabular} & $\mathbb{E} \Gap(\xout, \yout)$ & $\max\limits_{(x,y)} \mathbb{E} G(\xout, \yout, x,y)$ \\ 
 \hline
 PDHG~\cite{chambolle2011first}  & $O\left( \frac{d \sqrt{n} L_f \|A\|}{\varepsilon} \right)$ & $O\left( \frac{nd \|A\|}{\varepsilon} \right)$ & $O\left( \frac{nd \|A\|}{\varepsilon} \right)$ & $O\left( \frac{nd \|A\|}{\varepsilon} \right)$ \\[2mm]
 VRPDA~\cite{song2021variance} & $-$ & $-$ & $-$ & $\tilde{O}\left( \frac{nd \max_i \|A_i\|}{\varepsilon} \right)$ \\  [2mm]
 CLVR~\cite{song2021coordinate} & $-$ & $-^\dagger$ & $-$ & ${O}\left( \frac{nd \max_i \|A_i\|}{\varepsilon} \right)^\ddagger$ \\  [2mm]
SPDC~\cite{zhang2015stochastic} & $\tilde{O}\left( nd +\frac{d\sqrt{n}L_f \max_i \|A_i\|}{\varepsilon} \right)$ & $-$ & $-$ & $-$ \\[2mm]
Katyusha~\cite{allen2017katyusha} & $\tilde{O}\left(nd+ \frac{d\sqrt{n} L_f \max_i \|A_i\|}{\varepsilon} \right)$ & $-$ & $-$ & $-$ \\[2mm]
SDAPD~\cite{tan2020accelerated} & $\tilde{O}\left( nd+\frac{d\sqrt{n} L_f \max_i \|A_i\|}{\varepsilon} \right)$ & $-$ & $-$ & $-$ \\[2mm]
VRVI \cite{carmon2019variance,alacaoglu2021stochastic} & $-$ & $O\left( \frac{\sqrt{nd(n+d)} \|A\|_F}{\varepsilon}  \right)$ & $O\left( \frac{\sqrt{nd(n+d)} \|A\|_F}{\varepsilon} \right)$    & $O\left( \frac{\sqrt{nd(n+d)}\|A\|_F}{\varepsilon}  \right)$  \\[2mm]
 PURE-CD & $O\left( nd+ \frac{d L_f \sum_{i=1}^n \|A_i\|}{\sqrt{n}\varepsilon}  \right)$ & $O\left( \frac{d \sum_{i=1}^n \|A_i\|}{ \varepsilon}  \right)$ & $\tilde{O}\left( \frac{nd\max_i \|A_i\|}{\varepsilon}  \right)$    & ${O}\left( \frac{nd\max_i \|A_i\|}{\varepsilon}  \right)$ 
\end{tabular}
\caption{Complexity bounds for various measures of $\varepsilon$-optimality, for the convex-concave case of \eqref{eq:prob} with dense $A$ (\Cref{alg: purecd_dense}). Each entry shows the number of operations required to reduce the quantities indicated in each column label below $\varepsilon$. 
$^\ast$The first column is for solving $\min_x F(x) = g(x) + \frac{1}{n}\sum_{i=1}^n f_i(A_i x)$ and assumes that $f_i$ are $L_f$-Lipschitz 
and uses the conversion described in \Cref{app: upper_lower_bd}.  
$^\triangle$The second column is when $h(\cdot) = \delta_C(\cdot)$ for a set $C$.
$^\dagger$The result derived in~\cite[Cor.~1]{song2021coordinate} for CLVR applied to probpems with linear constraints are on the {\em expectation} of the output iterate $\mathbb{E}[\xout]$ (not normally available) rather than on $\xout$ itself, so we omit this result.
$^\ddagger$This result is given for the specific case of linear constrained problems.}
\label{tb: 1}
\end{table}

\paragraph{Convex-concave $L$ and sparse $A$ (Table~\ref{tb: 2} and Section~\ref{subsec: random_iter}).}
In the case of sparse $A$, we compare the results for PURE-CD obtained in this paper with the results for CLVR \cite{song2021coordinate} and the standard results for PDHG.
The results of~\cite{carmon2020coordinate} may also be applicable in this case. 
For matrix games with $\ell_2$ domains in both primal and dual,  \cite{carmon2020coordinate}  derives a complexity bound for expected gap of 
\begin{equation} \label{eq:carmon}
O\left( \nnz(A) + \varepsilon^{-1}\sqrt{\nnz(A)}\max\left\{\sqrt{\sum_i \|A_{i:}\|_1^2}, \sqrt{\sum_j \|A_{:j}\|_1^2}\right\} \right),
\end{equation}
which may be a remarkable improvement over the deterministic complexity in certain regimes.
We do not include this algorithm in our comparisons for two reasons. \emph{(i)}  The algorithm targets constrained problems with non-separable, bounded domains whereas we focus on the case of separable proximal functions, with potentially unbounded domains. The non-separability requires new data structures which result in a complicated algorithm whose practicality is unclear.
No implementation is done in~\cite{carmon2020coordinate}.
\emph{(ii)} Algorithmically,~\cite{carmon2020coordinate} builds on~\cite{carmon2019variance}, where the extension of the arguments to the proximal case requires either additional assumptions on the functions or additional boundedness assumptions on the domains. 
These assumptions hold on the matrix game applications of~\cite{carmon2020coordinate,carmon2019variance} but they may not hold on the problems we are most interested in here.
Still, when the relevant assumptions hold,  the complexity results in the algorithms of \cite{carmon2020coordinate}  may improve over our results and other best-known complexities (as discussed in that paper).

As shown in Table~\ref{tb: 2}, our analysis techniques do not show {\em improved} complexities for PURE-CD over PDHG, for  $\mathbb{E}\Gap(\xout, \yout)$ with sparse $A$ (\Cref{alg: purecd_sparse}). 
We show however that PURE-CD can {\em match} the complexity of the deterministic algorithm PDHG, which is the best known for \eqref{eq:prob}, and which is proved in \cite{alacaoglu2020random} when step sizes are selected appropriately.
We can also obtain for PURE-CD the same complexity bound as CLVR from~\cite{song2021coordinate}, which is on the weaker measure on $\max_{x, y} \mathbb{E} G(\xout, \yout, x, y)$.
We leave it as an open question whether improved complexity can be obtained for PURE-CD on the expected gap (see also our discussions in \Cref{subsec: random_iter} and \Cref{sec: conclusions}).

\begin{table}[h!]
\centering
\captionsetup{font=small}
\begin{tabular}{ l | l | l }
   & $\mathbb{E} \Gap(\xout, \yout)$ & $\max\limits_{(x,y)} \mathbb{E} G(\xout, \yout, x,y)$ \\ 
 \hline
 PDHG~\cite{chambolle2011first}  & $O\left( \text{nnz}(A) \|A\| \varepsilon^{-1} \right)$ & $O\left( \text{nnz}(A)  \|A\| \varepsilon^{-1} \right)$ \\[2mm]
 CLVR~\cite{song2021coordinate} & $-$ & $O\left( \text{nnz}(A) \max_i \|A_i\| \varepsilon^{-1} \right)$ \\[2mm]
 PURE-CD & $O\left( \text{nnz}(A) \|A\| \varepsilon^{-1} \right)^\dagger$ & $O\left( \text{nnz}(A) \max_i \|A_i\| \varepsilon^{-1} \right)$
\end{tabular}
\caption{Complexity bounds for various measures of $\varepsilon$-optimality, for the convex-concave case of \eqref{eq:prob} with sparse $A$ (\Cref{alg: purecd_sparse}). Each entry shows the number of operations required to reduce the quantities indicated in each column label below $\epsilon$. The first column shows that the complexity of PURE-CD matches that of PDHG for the expected-gap measure. 
In the second column, we assume that the output point $(\xout,\yout)$ is randomly selected iterate, rather than the ergodic average; see~\Cref{subsec: random_iter}.
$^\dagger$The proof of this result can be found in \cite{alacaoglu2020random} when  appropriate choices are made for the step sizes  and when $\|A\|\geq g(x_0) + h^\ast(y_0) - g(x_\star) - h^\ast(y_\star)$}.
\label{tb: 2}
\end{table}

\paragraph{Strongly convex cases and sparse $A$ (\Cref{tb: 3} and Sections~\ref{sec: str_cvx_str_ccv},~\ref{sec: str_cvx_ccv}).} 
When $h^*$ is strongly convex (second column of \Cref{tb: 3}), we  prove complexity results only for a measure involving only the dual variable, namely, $\mathbb{E}\| y_k - y_\star \|^2$.
For $A$ dense, it is easy to derive the same complexities for the expected gap using arguments from~\cite{chambolle2016ergodic}, but since sparsity introduces additional complications in the argument, we state results only for this measure.
This measure is commonly used in the primal-dual optimization literature when there is strong convexity in the dual problem~\cite{chambolle2018stochastic,chambolle2011first}.
We believe that by using regularization-based approach of~\cite{zhang2015stochastic,allen2017katyusha}, we can derive guarantees for the primal objective suboptimality, but we refrain from doing so here, to keep the algorithm and analysis \emph{direct}. 
Similar comments apply to the other strongly convex cases, shown in the first and third columns of the table.
To make comparisons between algorithms more straightforward, we focus on the (usual) case in which $\max_i\|A_i\|\geq \mu$, so that $\max \left\{ \max_i \|A_i \|/\mu,1 \right\} = \max_i \|A_i \|/\mu$.

For the strongly convex-strongly concave case, our complexity improves over VRVI methods~\cite{carmon2019variance,alacaoglu2021stochastic} when $\sqrt{\nnz(A)}\max_i\|A_i\|\leq \sqrt{n+d}\|A\|_F$.
This inequality holds when row norms are normalized, but otherwise one can find examples where one complexity is better than the other.
When $n\gg d$, this requirement would usually be satisfied since it is roughly equivalent to  $\sqrt{\frac{\nnz(A)}{n}}\max_i \|A_i\|\leq \|A\|_F$.
For sparse $A$, this condition will usually be satisfied easily.

\begin{table}[h!]
\captionsetup{font=small}
\begin{tabular}{ l | l | l | l  }
  & \begin{tabular}{@{}l@{}} $g$ is $\mu$-s.c., \\ $\mathbb{E}\| x_\star - x_k\|^2$\end{tabular} & \begin{tabular}{@{}l@{}} $h^\ast$ is $\mu$-s.c.,\\
  $\mathbb{E}\| y_\star - y_k\|^2$\end{tabular} & \begin{tabular}{@{}l@{}} $g, h^\ast$ are $\mu_1, \mu_2$-s.c.,\\ $\mathbb{E}\left[ \| x_\star - x_k \|^2 + \| y_\star - y_k\|^2 \right]$ \end{tabular}  \\ 
 \hline
 PDHG~\cite{chambolle2011first} & $O\left( \frac{\text{nnz}(A) \|A\|}{\mu \sqrt{\varepsilon}}  \right)$ & $O\left( \frac{\text{nnz}(A) \|A\|}{\mu \sqrt{\varepsilon}}  \right)$ & $O\left( \frac{\text{nnz}(A) \|A\|}{\sqrt{\mu_1\mu_2}} \log\varepsilon^{-1}  \right)$  \\[1mm]
  SPDHG~\cite{chambolle2018stochastic} & $O\left( \frac{nd \max_i \|A_i\|}{\mu \sqrt{\varepsilon}}  \right)$ &$O\left( \frac{nd \max_i \|A_i\|}{\mu \sqrt{\varepsilon}}  \right)$ &$O\left( \frac{nd\max_i \|A_i\|}{\sqrt{\mu_1\mu_2}} \log\varepsilon^{-1}  \right)$ \\[1mm]
 VRPDA~\cite{song2021variance} & $O\left( \frac{nd \max_i \|A_i\|}{\mu \sqrt{\varepsilon}}  \right)$ &$-$ &$-$ \\[1mm]
 CLVR~\cite{song2021coordinate} & $O\left( \frac{\nnz(A) \max_i \|A_i\|}{\mu \sqrt{\varepsilon}}  \right)$ &$-$ &$-$ \\[1mm]
 SPDAD~\cite{tan2020accelerated} & $-$ & $-$ & $O\left( \frac{\text{nnz}(A) \max_i \|A_i\|}{\sqrt{\mu_1\mu_2}} \log\varepsilon^{-1}  \right)$\\[1mm]
  VRVI~\cite{carmon2019variance,alacaoglu2021stochastic} & $-$ & $-$ & $O\left( \frac{\sqrt{\nnz(A)(n+d)}\|A\|_F} {\sqrt{\mu_1\mu_2}} \log\varepsilon^{-1}  \right)$\\[1mm]
  PURE-CD & $O\left( \frac{\text{nnz}(A) \max\left\{\frac{\max_i \|A_i\|}{\mu}, 1 \right\}}{\sqrt{\varepsilon}}  \right)$ & $O\left( \frac{\text{nnz}(A) \max\left\{\frac{\max_i \|A_i\|}{\mu}, 1 \right\}}{\sqrt{\varepsilon}} \right)$ & $O\left( \frac{\text{nnz}(A) \max_i \|A_i\|}{\sqrt{\mu_1\mu_2}} \log\varepsilon^{-1}  \right)$
\end{tabular}
\caption{Complexity bounds for various measures of $\epsilon$-optimality, for the strongly convex cases of \eqref{eq:prob} with sparse $A$ (\Cref{alg: purecd_sparse}). Each entry shows the number of operations required to reduce the quantities indicated in each column label below $\epsilon$. \cite{tan2020accelerated} provided a lazy update scheme for SPDAD only in the case tabulated in the last column. 
}
\label{tb: 3}
\end{table}

\subsection{Notation and Terminology}\label{sec: notation}
We use the notation $\dom{g}=\{ x\in\mathbb{R}^d\colon g(x) < +\infty \}$; $A_i$ for $i$-th row of data matrix $A \in \R^{n \times d}$; $\|A\|$ for spectral norm; $\nnz(A)$ for number of nonzero elements in matrix $A$; $[n]=\{ 1,\dots,n\}$; the projection operator $P_{C}(x) = \arg\min_{u\in C} \| u-x\|^2$ for a nonempty, convex, closed set $C$; $\dist^2(x, C) = \| x-P_C(x)\|^2$; $\delta_C(x) = 0$ if $x\in C$ and $\delta_C(x) = +\infty$ if $x\not \in C$. For $Q\succ 0$, we denote the weighted inner product and norm as $\langle x, y \rangle_Q = \langle Qx, y \rangle$ and $\| x \|_Q = \langle x, Qx \rangle^{1/2}$. 
The $i$-th coordinate of a vector $y_k$ is denoted by $y_k^{(i)}$.

When we use diagonal step sizes and parameters, we frequently use the following notations for important diagonal matrices: $\Sigma=\diag(\sigma^{(1)}, \dots, \sigma^{(n)})$, $\Tau = \diag(\tau^{(1)}, \dots, \tau^{(d)})$, $\Theta = \diag(\theta^{(1)}, \dots, \theta^{(d)})$, and $\Pi = \diag(\pi^{(1)}, \dots, \pi^{(d)})$.

For a key conditional expectation, we use the notation $\mathbb{E}_k [\cdot] = \mathbb{E}[\cdot | i_0, \dots, i_{k-1}]$.
When dealing with sparse $A$, we use the definitions 
\[
J(i)=\{ j\in[d]\colon A_{i, j} \neq 0 \}\text{~~and~~}I(j)=\{ i\in[n]\colon A_{i, j} \neq 0 \}.
\]

For a diagonal weighting matrix $\Tau$ and convex function $g$,  we define $\prox_{\Tau, g}(x) = \arg\min_{u} g(u) + \frac{1}{2}\| u-x\|^2_{\Tau^{-1}}$.
We will frequently use the prox-inequality, which states that for a $\mu$-strongly convex function $g$ with $\mu \geq 0$, we have
\begin{equation}\label{eq: prox_ineq}
x^+ = \prox_{\Tau, g}(x) \iff \langle x^+ - x, \Tau^{-1}(u - x^+) \rangle + g(u) - g(x^+) \geq \frac{\mu}{2} \| u-x^+\|^2, ~~\forall u,
\end{equation}
which follows by convexity of $g$ and definition of $x^+$.

We make frequent use of the following notation for a distance measure from initial point $(x_0,y_0)$ to the solution set:
\begin{equation} \label{eq:Dstar}
D_\star = \min_{(x_\star, y_\star) \in \mathcal{Z}_\star}\left(\| x_\star-x_0 \|^2 + \| y_\star-y_0 \|^2\right),
\end{equation}
where $\mathcal{Z}_\star$ is the set of primal-dual solutions, which is nonempty by~\Cref{asmp: asmp1}. Also denote 
\begin{equation} \label{eq:DZ}
D_{\mathcal{Z}} = \max_{(x, y) \in \mathcal{Z}} \left(\| x-x_0\|^2 + \|y-y_0\|^2\right)
\end{equation}
for any compact set $\mathcal{Z}$. It is known from~\cite{alacaoglu2020random} that iterates of PURE-CD are almost surely bounded and that they converge almost surely to a solution.

\section{Algorithm}
The PURE-CD algorithm proposed in~\cite{alacaoglu2020random} is a variant of primal-dual hybrid gradient (PDHG), which is a deterministic algorithm for~\eqref{eq:prob}.
PDHG can be seen as a gradient descent-ascent method with extrapolation. 
One of the most straightforward algorithms for \eqref{eq:prob} is gradient descent-ascent (GDA), whose iterates have the form
\begin{equation}\tag{GDA}
\begin{aligned}
    \bar x_{k+1} & = \prox_{\tau, g}(\bar x_k - \tau A^\top \bar y_{k})\\
    \bar y_{k+1} & = \prox_{\sigma, h^\ast}(\bar y_k + \sigma A \bar x_{k+1}),
\end{aligned}
\end{equation}
for positive step sizes $\tau$ and $\sigma$.
By contrast, one iteration of PDHG has the form
\begin{equation}\label{eq: pdhg}\tag{PDHG}
\begin{aligned}
    \bar x_{k+1} &= \prox_{\tau, g}(\bar x_k - \tau A^\top (2\bar y_{k}-\bar y_{k-1})) \\
    \bar y_{k+1} &= \prox_{\sigma, h^\ast}(\bar y_k + \sigma A \bar x_{k+1}),
\end{aligned}
\end{equation}
where the extrapolated iterate $(2\bar y_k-\bar y_{k-1})$ in place of the of $\bar y_k$ used by GDA.
An equivalent form of PDHG is 
\begin{subequations}\label{eq: tripd}
\begin{align}
\label{eq: tripd.1}
    \bar x_{k+1} &= \prox_{\tau, g}(\hat x_k - \tau A^\top \bar y_{k}) \\
    \label{eq: tripd.2}
    \bar y_{k+1} &= \prox_{\sigma, h^\ast}(\bar y_k + \sigma A  \bar x_{k+1}) \\
    \label{eq: tripd.3}
    \hat x_{k+1} &= \bar x_{k+1} - \tau A^\top (\bar y_{k+1} -  \bar y_k).
\end{align}
\end{subequations}

PURE-CD builds on~\eqref{eq: tripd}. In the case of dense $A$ (see \Cref{alg: purecd_dense}), it computes $\bar{x}_{k+1}$ as in \eqref{eq: tripd.1}, but then updates just a single, randomly chosen component $i_k$ of $y$ to obtain $y_{k+1}$. It then takes an extrapolation step like \eqref{eq: tripd.3}, based on the difference between $y_k$ and $y_{k+1}$ (which differ in only a single component).

When $A$ is sparse and $g$ is separable (that is, $g(x)=\sum_{j=1}^d g_j(x^{(j)})$, a sparse-friendly version of \Cref{alg: purecd_dense} can be designed to exploit the sparsity of $A$ and require only $O(|J(i_k)|)$ operations on iteration $k$. Such an algorithm is shown in \Cref{alg: purecd_sparse} (proposed in ~\cite{alacaoglu2020random}), for the case in which $i_k$ is chosen from a uniform distribution on $[n]$ at each iteration. 
We note that~\Cref{alg: purecd_dense} and~\Cref{alg: purecd_sparse} are equivalent with dense $A$ (when $p^{(i)}=1/n$ for all $i \in [n]$), but they are different with sparse $A$. 
Specifically, for all $j\not\in J(i_k)$,~\Cref{alg: purecd_dense} sets $x_{k+1}^{(j)} = \bar x_{k+1}^{(j)}$, which requires knowledge of  all components of $\bar x_{k+1}$. 
On the other hand,~\Cref{alg: purecd_sparse} sets $x_{k+1}^{(j)}=x_k^{(j)}$ for all $j\not\in J(i_k)$. 
In fact, this is the main reason for lower per-iteration cost in~\Cref{alg: purecd_sparse}.

The efficiency of \Cref{alg: purecd_sparse} depends critically on the fact that  at iteration $k$,  $x_k$ and $x_{k+1}$ differ only in the components $j \in J(i_k)$. 
Thus, we compute only the components $j \in J(i_k)$ of the vector $\bar{x}_{k+1}$ (which are subsequently used in computation of the corresponding components of $x_{k+1}$)
and simply copy across the remaining components $j \notin J(i_k)$ from $x_k$ to $x_{k+1}$.
The computation of $A_{i_k} \bar x_{k+1}$ on line~\ref{st: alg_sparse_yk} of \Cref{alg: purecd_sparse} also requires only knowledge of $\bar x_{k+1}^{(j)}$ for $j\in J(i_k)$. 
Moreover, the vector $A^\top y_k$ can be maintained using a classical technique borrowed from  coordinate descent (see for example~\cite{nesterov2012efficiency,wright2015coordinate,chambolle2018stochastic}), namely:
\begin{align*}
    A^\top y_{k+1} &= A^\top y_{k} + A^\top (y_{k+1} - y_k) = A^\top y_{k} + A_{i_k}^\top (y_{k+1}^{(i_k)} - y_k^{(i_k)}) \\
    \iff (A^\top y_{k+1})^{(j)} &=  \begin{cases}
    (A^\top y_{k})^{(j)} + A_{i_k, j} (y_{k+1}^{(i_k)} - y_k^{(i_k)}), &  j\in J(i_k),\\
     (A^\top y_{k})^{(j)}, & j\not\in J(i_k).
     \end{cases}
\end{align*}
The cost of this procedure is  $O(|J(i_k)|)$ operations.

\begin{algorithm}
\caption{PURE-CD with dense data~\cite{alacaoglu2020random}}\label{alg: purecd_dense}
\begin{algorithmic}[1]
\STATE Initialize $x_0\in \dom g, y_0\in \dom h^\ast$
\FOR {$k\geq 0$} 
\STATE $\bar x_{k+1} = \prox_{\Tau_k, g}(x_k - \Tau_k A^\top y_k)$\label{eq: wh8}
\STATE $\text{Pick } i_{k} \in [n] \text{ with } \Pr(i_k = i) = p^{(i)}\text{ and } P =\diag(p^{(1)}, \dots, p^{(n)})$
\STATE $y_{k+1}^{(i_{k})} = \prox_{\sigma_k^{(i_{k})}, h^\ast_{i_{k}}}(y_k^{(i_{k})} + \sigma_k^{(i_{k})} A_{i_k} \bar x_{k+1})$
\STATE $y_{k+1}^{(i)} = y_k^{(i)}, \forall i \neq i_{k}$
\STATE $x_{k+1} = \bar x_{k+1} - \Tau_k \Theta_k  A^\top P^{-1} (y_{k+1} - y_k)$\label{st: step_extrap_dense}
\ENDFOR
\end{algorithmic}
\end{algorithm}

\begin{algorithm}
\caption{PURE-CD with sparse data~\cite{alacaoglu2020random}}\label{alg: purecd_sparse}
\begin{algorithmic}[1]
\STATE Initialize $x_0\in \dom g, y_0\in \dom h^\ast$ and recall $J(i) = \{ j\in[d]\colon A_{i, j} \neq 0 \}$, $I(j) = \{ i\in[n]\colon A_{i, j} \neq 0 \}$
\FOR {$k\geq 0$} 
\STATE $\text{Pick } i_{k} \in [n] \text{ with } \Pr(i_k = i) = \frac{1}{n}$
\STATE $\bar x_{k+1}^{(j)} = \prox_{\tau_k^{(j)}, g_j}\left(x_k^{(j)} - \tau_k^{(j)} (A^\top y_k)^{(j)}\right),~~\forall j \in J(i_k)$
\STATE $y_{k+1}^{(i_{k})} = \prox_{\sigma_k^{(i_{k})}, h^\ast_{i_{k}}}(y_k^{(i_{k})} + \sigma_k^{(i_{k})} A_{i_k} \bar x_{k+1})$\label{st: alg_sparse_yk}
\STATE $y_{k+1}^{(i)} = y_k^{(i)}, \forall i \neq i_{k}$
\STATE $x_{k+1}^{(j)} = \bar x_{k+1}^{(j)} - \tau_k^{(j)} \theta_k^{(j)} A_{i_k, j} (y_{k+1}^{(i_k)} - y_k^{(i_k)}), \; \forall j \in J(i_{k}) $\label{st: rg4}
\STATE $x_{k+1}^{(j)} = x_k^{(j)}, \forall j \not\in J(i_{k})$
\ENDFOR
\end{algorithmic}
\end{algorithm}

For future reference, we define
\begin{equation} \label{eq:def.ybar}
    \bar y_{k+1} = \prox_{\Sigma_k, h^\ast}(y_k + \Sigma_k A\bar x_{k+1}),
\end{equation}
where $\Sigma_k = \diag (\sigma_k^{(1)}, \dotsc, \sigma_k^{(n)})$, as defined in \Cref{sec: notation}.
Let us note that in~\Cref{alg: purecd_sparse}, we have $\bar x_{k+1} = \prox_{\Tau_k, g}(x_k - \tau_k A^\top y_k)$.
However, since each iteration of~\Cref{alg: purecd_sparse} only uses $\bar x_{k+1}^{(j)}$ for $j\in J(i_k)$, we compute only these entries, as described earlier.

\section{Convergence in the Convex-Concave Case}\label{sec: sec_dense_cvxccv}

We focus now on the convex-concave case, where \Cref{asmp: asmp1} holds but  strong convexity does not necessarily hold for $g$ or $h^*$. 
\Cref{sec: cvx_ccv_dense} considers the case in which  $A$ is dense,  deriving the results in the PURE-CD row of  \Cref{tb: 1}.
We also specialize the  results to the  important special cases of linearly constrained optimization and ERM; see \Cref{sec: lincons_erm}.
In \Cref{sec: cvx_ccv_sparse}, we consider the case of $A$ sparse, deriving the results in the second column of \Cref{tb: 2} for PURE-CD. 
In \Cref{sec: cvx_ccv_sparse}, we also derive a result analyzing one iteration of \Cref{alg: purecd_sparse} that we use later in \Cref{sec: str_cvx_str_ccv} and \Cref{sec: str_cvx_ccv}.

\subsection{Convergence Analysis for Dense $A$ (\Cref{alg: purecd_dense})}\label{sec: cvx_ccv_dense}

We start with a lemma to analyze the behavior of a single iteration of~\Cref{alg: purecd_dense}.
Our analysis in this lemma follows~\cite{alacaoglu2020random} but is simpler, since we focus here on a specific case with dense $A$ and we do not include a general smooth function of $x$ in the objective, in addition to the function $g$ which is handled with a prox operation.
For clarity of presentation, we include the full proof.

\begin{lemma}[Single-iteration analysis of \Cref{alg: purecd_dense}]\label{lem: lem1_purecd}
Let~\Cref{asmp: asmp1} hold and recall the definition of $\bar y_{k+1}$ from~\eqref{eq:def.ybar}. 
In~\Cref{alg: purecd_dense}, set $\Theta_k = \theta I$. 
Then for any $(x, y) \in \mathbb{R}^d \times \mathbb{R}^n$ and for $k\geq 0$, we have
\begin{equation}\label{eq: purecd_one_iter}
\begin{aligned}
& 2(L(\bar x_{k+1}, y) - L(x, \bar y_{k+1})) + \| x-x_{k+1} \|^2_{\Tau_k^{-1}} + \| y- \bar y_{k+1}\|^2_{\Sigma_k^{-1}} \\
& \leq \| x-x_k\|^2_{\Tau_k^{-1}}+\| y-y_k\|^2_{\Sigma_k^{-1}} 
- \| \bar y_{k+1} - y_k\|^2_{\Sigma_k^{-1}} - \theta^2 \| A^\top P^{-1}(y_{k+1} - y_k)\|^2_{\Tau_k} \\
& \quad + 2\theta \langle x-x_{k+1}, A^\top P^{-1}(y_{k+1} - y_k) \rangle
 +2 \langle x-\bar x_{k+1}, A^\top(y_k - \bar y_{k+1}) \rangle.
\end{aligned}
\end{equation}
Moreover, when we set $\Tau_k = \tau I$ and $\Sigma_k = \Sigma = \diag(\sigma^{(1)}, \dots, \sigma^{(n)})$, with step sizes satisfying $\sigma^{(i)} \tau \left(p^{(i)}\right)^{-1}  \|A_i\|^2 \leq \gamma < 1$ for all $i$, and $\theta=1$, then we have for any solution $(x_\star, y_\star)$ that
\begin{equation}\label{eq: purecd_sum_resid_general1}
 \sum_{k=0}^{K-1} \mathbb{E}\| y_{k+1} - y_k\|^2_{\Sigma^{-1} P^{-1}} \leq \frac{1}{1-\gamma} \left(\frac{1}{\tau}\|x_\star -x_0\|^2 + \| y_\star -y_0\|^2_{\Sigma^{-1}P^{-1}} \right).
\end{equation}
\end{lemma}
\begin{remark}
In the first result of the lemma, the first term on the LHS of \eqref{eq: purecd_one_iter} is used to get the gap function in view of~\eqref{eq: opt_cond_def}, while the remaining terms on the LHS make up the Lyapunov function 
 and hence will telescope for appropriate choices of $\Tau_k$ and $\Sigma_k$.
The third and fourth terms on the RHS are used to cancel the fifth and sixth terms. 
\end{remark}

The following result is an immediate consequence of the bound \eqref{eq: purecd_sum_resid_general1} in \Cref{lem: lem1_purecd}. 
\begin{corollary}
Suppose the setup of \Cref{lem: lem1_purecd} holds.
By setting $p^{(i)} = \frac{\|A_i\|}{\sum_{i=1}^n \|A_i\|}$, $\tau = \frac{1}{\sum_{i=1}^n \|A_i\|}$, $\sigma_k^{(i)}=\sigma^{(i)} =\frac{\gamma}{\|A_i\|}$, for $\gamma < 1$ and $\theta = 1$ we have for any solution $(x_\star, y_\star)$ that
\begin{equation}\label{eq: purecd_sum_residual}
 \sum_{k=0}^{K-1} \mathbb{E}\| y_{k+1} - y_k\|^2_{\Sigma^{-1} P^{-1}} \leq \frac{\left(\sum_{i=1}^n \|A_i\|\right) (\|x_\star -x_0\|^2 + \| y_\star -y_0\|^2 )}{(1-\gamma)\gamma} .
\end{equation}
\end{corollary}
\begin{proof}[Proof of~\Cref{lem: lem1_purecd}]
By the prox-inequality \eqref{eq: prox_ineq} with $\mu=0$, applied to the definition of $\bar x_{k+1}$ in step~\ref{eq: wh8}, it follows for any $x$ that
\begin{align*}
\langle \bar x_{k+1} - x_k + \Tau_k A^\top y_k, \Tau_k^{-1}(x - \bar x_{k+1}) \rangle + g(x) -   g(\bar x_{k+1}) &\geq 0 \\
\Leftrightarrow \langle \Tau_k^{-1/2}(\bar x_{k+1} - x_k), \Tau_k^{-1/2}(x-\bar x_{k+1}) \rangle + \langle A^\top y_k, x-\bar x_{k+1} \rangle + g(x) - g(\bar x_{k+1}) &\geq 0.
\end{align*}
We use the standard equality $2\langle a, b \rangle = \|a+b\|^2- \|a\|^2 - \|b\|^2$ and rearrange to get
\begin{align}\label{eq: purecd_primal_ineq1}
\| x- \bar x_{k+1} \|^2_{\Tau_k^{-1}} + 2(g(\bar x_{k+1}) - g(x)) \leq \| x-x_k\|^2_{\Tau_k^{-1}} - \| \bar x_{k+1} - x_k \|^2_{\Tau_k^{-1}} + 2 \langle A^\top y_k, x-\bar x_{k+1} \rangle.
\end{align}
By the definition of $x_{k+1}$ in step~\ref{st: step_extrap_dense}, it follows that $\bar x_{k+1} = x_{k+1} + \Tau_k  \theta A^\top P^{-1}(y_{k+1}-y_k)$. This gives
\begin{align*}
\| x-\bar x_{k+1} \|^2_{\Tau_k^{-1}} = \| x-x_{k+1}\|^2_{\Tau_k^{-1}} + \theta^2 \| A^\top P^{-1}(y_{k+1} - y_k) \|^2_{\Tau_k} - 2 \theta \langle x-x_{k+1}, A^\top P^{-1}(y_{k+1} - y_k) \rangle.
\end{align*}
By substituting this expression into \eqref{eq: purecd_primal_ineq1}, we obtain
\begin{multline}\label{eq: purecd_primal_ineq2}
\| x-x_{k+1} \|^2_{\Tau_k^{-1}} + 2 (g(\bar x_{k+1})-g( x)) \leq \| x-x_k\|^2_{\Tau_k^{-1}} -\|\bar x_{k+1} - x_k\|^2_{\Tau_k^{-1}}+ 2 \langle A^\top y_k, x-\bar x_{k+1} \rangle \\
- \theta^2 \|A^\top P^{-1} (y_{k+1} - y_k) \|^2_{\Tau_k} + 2\theta  \langle x-x_{k+1}, A^\top P^{-1}(y_{k+1} - y_k) \rangle.
\end{multline}
From \eqref{eq: prox_ineq} with $\mu=0$ and the definition of $\bar y_{k+1}$ from~\eqref{eq:def.ybar}, we have for any $y$ that
\begin{align*}
 \langle \bar{y}_{k+1}-y_k - \Sigma_k A \bar{x}_{k+1}, \Sigma_k^{-1} (y-\bar{y}_{k+1}) \rangle + h^\ast(y)-h^\ast(\bar{y}_{k+1}) & \ge 0 \\
\Leftrightarrow  \;\; \langle \Sigma_k^{-1/2} (\bar{y}_{k+1} - y_k), \Sigma_k^{-1/2} (y-\bar{y}_{k+1}) \rangle  - \langle A \bar{x}_{k+1}, y-\bar{y}_{k+1}\rangle  + h^\ast(y)-h^\ast(\bar{y}_{k+1}) & \ge 0.
\end{align*}
By applying $2\langle a, b \rangle = \|a+b\|^2- \|a\|^2 - \|b\|^2$ to the first term and rearranging, we obtain 
\begin{align}\label{eq: purecd_dual_ineq1}
\| y-\bar y_{k+1} \|^2_{\Sigma_k^{-1}} + 2(h^\ast(\bar y_{k+1}) - h^\ast(y)) \leq \| y-y_k\|^2_{\Sigma_k^{-1}} - \| \bar y_{k+1} - y_k \|^2_{\Sigma_k^{-1}} - 2\langle A\bar x_{k+1}, y-\bar y_{k+1} \rangle.
\end{align}
We add~\eqref{eq: purecd_primal_ineq2} to~\eqref{eq: purecd_dual_ineq1} and drop the nonpositive term $-\| \bar x_{k+1} -x_k\|^2$ on the RHS to obtain
\begin{equation}\label{eq: pd_combined}
\begin{aligned}
    & 2\left(g(\bar x_{k+1}) + h^\ast(\bar y_{k+1}) - g(x) - h^\ast(y)\right) +  \| x- x_{k+1} \|^2_{\Tau_k^{-1}} + \| y-\bar y_{k+1}\|^2_{\Sigma_k^{-1}} \\
    & \leq  \| x-x_{k} \|^2_{\tau_k^{-1}} + \| y- y_{k}\|^2_{\Sigma_k^{-1}}
    -\| \bar y_{k+1} - y_k\|^2_{\Sigma_k^{-1}} - \theta^2 \|A^\top P^{-1} (y_{k+1} - y_k)\|^2_{\Tau_k} \\
    & \quad + 2 \langle A^\top y_k, x-\bar x_{k+1} \rangle +2\theta\langle x-x_{k+1}, A^\top P^{-1}(y_{k+1} - y_k) \rangle  - 2\langle A\bar x_{k+1}, y-\bar y_{k+1} \rangle.
\end{aligned}
\end{equation}
Recall that
\begin{equation}\label{eq: g_deff}
L(\bar x_{k+1}, y) - L(x, \bar y_{k+1}) = g(\bar x_{k+1}) + \langle A^\top y, \bar x_{k+1} -x \rangle - g(x) + h^\ast(\bar y_{k+1}) - \langle Ax, \bar y_{k+1} - y \rangle - h^\ast(y).
\end{equation}
We use this definition in~\eqref{eq: pd_combined} and note that $2\langle A^\top(y_k-y), x-\bar x_{k+1} \rangle + 2\langle x-\bar x_{k+1}, A^\top(y-\bar y_{k+1}) \rangle = 2\langle x-\bar x_{k+1}, A^\top(y_k-\bar y_{k+1}) \rangle$ to deduce \eqref{eq: purecd_one_iter}, the first claim of the lemma.

For the second result \eqref{eq: purecd_sum_resid_general1}, we start from~\eqref{eq: purecd_one_iter} and let $(x, y) = (x_\star, y_\star)$, $\Tau_k = \tau I$, and $\Sigma_k = \Sigma$. 
We estimate first the last two terms on the RHS of~\eqref{eq: purecd_one_iter}. 
By  noting that $\bar x_{k+1}$ does not depend on $i_{k}$ and $\mathbb{E}_k P^{-1}(y_k - y_{k+1}) = y_k - \bar y_{k+1}$, and by setting $\theta=1$ and using the definition of $\bar x_{k+1}$ from step~\ref{st: step_extrap_dense}, we obtain for the sum of these two terms that
\begin{align}
\nonumber
& \E_k\left[2\langle x_\star-x_{k+1}, \theta A^\top P^{-1}(y_{k+1} - y_k) \rangle + 2\langle x_\star -\bar x_{k+1}, A^\top(y_k - \bar y_{k+1}) \rangle \right] \\
\nonumber
&=  2 \E_k \langle x_\star-x_{k+1}, \theta A^\top P^{-1}(y_{k+1} - y_k) + 2 \E_k \langle x_\star -\bar x_{k+1}, A^\top P^{-1} (y_k - y_{k+1}) \rangle  \rangle \\
\nonumber
& =
2 \mathbb{E}_k \langle x_{k+1} - \bar x_{k+1}, A^\top P^{-1}(y_k - y_{k+1}) \rangle \\
\nonumber
& = -2\tau  \mathbb{E}_k\langle A^\top P^{-1}(y_{k+1} - y_k), A^\top P^{-1}(y_k - y_{k+1})  \rangle \\
\label{eq: 3_errors2}
& = 2\tau \mathbb{E}_k\| A^\top P^{-1}(y_{k+1} - y_k) \|^2.
\end{align}
Note that this expression partially cancels with the third-last term on the RHS of \eqref{eq: purecd_one_iter} which, when the conditional expectation $\E_k$ is taken, has the same form as   \eqref{eq: 3_errors2} but a different coefficient.
By taking the conditional expectation $\E_k$ of both sides of of  \eqref{eq: purecd_one_iter}, substituting \eqref{eq: 3_errors2} for the last two terms on the RHS of  \eqref{eq: purecd_one_iter}, using $\theta=1$, and noting that $L(\bar x_{k+1}, y_\star) - L(x_\star, \bar y_{k+1}) \geq 0$, we obtain
\begin{align}
\nonumber
&     \frac{1}{\tau} \mathbb{E}_k\| x_\star-x_{k+1} \|^2 + \| y_\star- \bar y_{k+1}\|^2_{\Sigma^{-1}} \\
\label{eq: purecd_one_iter X}
& \leq \frac{1}{\tau} \| x_\star-x_k\|^2+\| y_\star-y_k\|^2_{\Sigma^{-1}} 
- \| \bar y_{k+1} - y_k\|^2_{\Sigma^{-1}} + \tau \E_k \| A^\top P^{-1}(y_{k+1} - y_k)\|^2.
\end{align}
Next, we use~\Cref{lem: tech_lemma_one_iter} with $\Phi_k = \Sigma^{-1}P^{-1}$ and $y=y_\star$ to derive 
\[
\mathbb{E}_k \| y_\star - y_{k+1} \|^2_{\Sigma^{-1}P^{-1}}=\|y_\star-\bar y_{k+1} \|^2_{\Sigma^{-1}} + \| y_\star-y_k \|^2_{\Sigma^{-1}(P^{-1}-I)}.
\]
By using this identity to substitute for $\| y_\star - \bar y_{k+1} \|^2_{\Sigma^{-1}}$ on the LHS of \eqref{eq: purecd_one_iter X}, and rearranging, we obtain
\begin{equation}\label{eq: ineq_1}
\begin{aligned}
& \frac{1}{\tau} \mathbb{E}_k\| x_\star - x_{k+1} \|^2 + \mathbb{E}_k \| y_\star - y_{k+1} \|^2_{\Sigma^{-1} P^{-1}} \\
& \leq \frac{1}{\tau} \| x_\star-x_k\|^2 +  \| y_\star-y_k\|^2_{\Sigma^{-1} P^{-1}}
 -  \| \bar y_{k+1} - y_k\|^2_{\Sigma^{-1}}  + \tau \mathbb{E}_k \| A^\top P^{-1}(y_{k+1} - y_k) \|^2.
\end{aligned}
\end{equation}
Since $y_{k+1} - y_k$ is one-sparse for all $k \geq 0$, we have that
\begin{align}
\nonumber
    \E_k\| A^\top P^{-1} (y_{k+1} - y_k) \|^2 & = \mathbb{E}_k\sum_{j=1}^d \left(\left(A^\top P^{-1}(y_{k+1}-y_k)\right)^{(j)}\right)^2 \\
    \nonumber
    & = \mathbb{E}_k\sum_{j=1}^d \left(\left(A^\top((p^{(i_k)})^{-1}(\bar y_{k+1}^{(i_{k})}-y_k^{(i_{k})})e_{i_{k}})\right)^{(j)}\right)^2 \\
    \nonumber
& = \mathbb{E}_k\sum_{j=1}^d \left(\left(A_{i_{k}}^\top (p^{(i_k)})^{-1}(\bar y_{k+1}^{(i_{k})} - y_k^{(i_{k})})\right)^{(j)}\right)^2  \\
\nonumber
& =\mathbb{E}_k\sum_{j=1}^d A_{i_{k}, j}^2 (p^{(i_k)})^{-2} (\bar y_{k+1}^{(i_{k})} - y_k^{(i_{k})})^2 \\
\nonumber
& = \sum_{i=1}^n \sum_{j=1}^d (p^{(i)})^{-1} A_{i, j}^2 (\bar y_{k+1}^{(i)} - y_k^{(i)})^2 \\
& = \sum_{i=1}^n  (p^{(i)})^{-1} \| A_i \|^2 (\bar y_{k+1}^{(i)} - y_k^{(i)})^2.
\label{eq: purecd_error_sparsity}
\end{align}
Recalling the step size condition  $\sigma^{(i)} \tau (p^{(i)})^{-1} \|A_i\|^2 \leq \gamma < 1$, we have for the last two terms on the RHS of~\eqref{eq: ineq_1} and for $k\geq 0$ that
\begin{multline}\label{eq: lem1_error_terms_cancel}
    -\|\bar y_{k+1} - y_k\|^2_{\Sigma^{-1}} + \tau \mathbb{E}_k  \| A^\top P^{-1}(y_{k+1} - y_k) \|^2 = \sum_{i=1}^n \left(-\frac{1}{\sigma^{(i)}} + \tau (p^{(i)})^{-1} \|A_i\|^2 \right) \left( \bar y_{k+1}^{(i)} - y_k^{(i)} \right)^2 \\
    \leq \sum_{i=1}^n \frac{-1+\gamma}{\sigma^{(i)}} \left( \bar y_{k+1}^{(i)} - y_k^{(i)} \right)^2 = -(1-\gamma) \| \bar y_{k+1} - y_k\|^2_{\Sigma^{-1}}= -(1-\gamma)\E_k \| y_{k+1} - y_k\|^2_{\Sigma^{-1}P^{-1}},
\end{multline}
where the last step used 
\Cref{lem: tech_lemma_one_iter} with $\Phi = \Sigma^{-1}P^{-1}$ and $y=y_k$.
By substituting the estimate \eqref{eq: lem1_error_terms_cancel} onto~\eqref{eq: ineq_1}, taking total expectation, and summing the resulting inequality over $k=0,1,\dots,K-1$, we obtain the second result of the lemma. 
\end{proof}

\subsubsection{Two Special Cases: Linearly-constrained Optimization and ERM}
\label{sec: lincons_erm}

We recall from~\eqref{eq:prob} the problem
\begin{equation*}
    \min_{x\in\mathbb{R}^d} \max_{y\in\mathbb{R}^n} L(x, y):= \sum_{i=1}^n \langle A_ix, y^{(i)} \rangle - h_i^\ast(y^{(i)}) + g(x),
\end{equation*}
where $h(z) = \sum_{i=1}^n h_i(z^{(i)})$.
In this section, we consider two important special cases of $h$ where we obtain improved complexities for convergence in the primal iterate $x$.
We focus on these cases because of their many applications and because they admit stronger results
than the general case. 
These stronger results also come with simpler proofs --- we can bypass the main difficulty by showing guarantees for primal quantities rather than the primal-dual gap.
In both cases, we improve the best-known complexities in the literature (see~\Cref{tb: 1}, \Cref{sec:cc}, and \Cref{sec: sect_overview}).

In the first special case, we consider $h$ Lipschitz, so that \eqref{eq:prob} reduces to
\begin{equation}\label{eq: case1_prob}
    \min_{x\in\mathbb{R}^d} \sum_{i=1}^n h_i(A_i x) + g(x).
\end{equation}
This formulation has applications in empirical risk minimization with nonsmooth losses, where $h$ can be hinge loss or the  $\ell_1$ or $ \ell_2$ norm, covering applications such as SVM, least absolute deviation regression, TV-regularized regression, and many popular formulations.
Lipschitzness of $h_i$ is a common assumption for solving ERM problems with nonsmooth loss functions, see~\cite[Cor. 3.7]{allen2017katyusha},~\cite[Cor. 3]{zhang2015stochastic},~\cite[Thm. 2]{shalev2013stochastic}.
Our complexity result here is to bound number of iterations required to output $\xout$ such that
$\mathbb{E} \left[ h(A\xout) + g(\xout) - h(Ax_\star) - g(x_\star) \right] \leq \varepsilon$.
(This result follows from a bound like \eqref{eq: opt_cond_def}.)

The second special case is convex linearly constrained optimization, which is 
\begin{equation}\label{eq: case2_prob}
    \min_{x\in\R^d} \, g(x), \text{~such that~} Ax\in C,
\end{equation}
for a nonempty, closed, convex set $C$ with the separable structure $C = C_1\times \dots \times C_n$.
This can be obtained from \eqref{eq:prob} when we set $h_i=\delta_{C_i}$.
In the case of $C_i=\{b_i\}$, for $b_i\in\mathbb{R}$, an important instance of  formulation \eqref{eq: case2_prob} is distributed optimization, where the equality constraints enforce consistency of the variable vector across processors.
By setting $C$ to be a shifted nonpositive orthant (that is, $C_i = \{t \, | \, t \ge b_i \}$, this framework includes general linear constraints. 
The need to perform a prox operation with $g$ in \Cref{alg: purecd_dense} restricts the types of objectives for which this algorithm may be practical. 
Nevertheless this class includes such problems as linear programming ($g(x)=c^\top x$ for some $c \in \R^d$) and convex quadratic programming ($g(x)=\tfrac12 x^\top Qx + c^\top x$ for $Q$ symmetric and positive semidefinite), provided that systems with coefficient matrices of the form $Q+\tau D$ for positive diagonal $D$ can be solved efficiently (such as when $Q$ is banded). 
For dense $Q$, we can reformulate the convex QP by factoring $Q=LL^\top$ in a preprocessing step, introducing a variable $t = L^\top x$, and redefining the objective to be $\tfrac12 t^\top t + c^\top x$. 
The prox operation involving $g$ for this reformulated problem can be performed efficiently.
Using again a bound like  \eqref{eq: opt_cond_def}, we bound the number of iterates required to output $(\xout)$ such that $\mathbb{E} | g(\xout) - g(x_\star) | \leq \varepsilon$ and  $\mathbb{E}\dist(A\xout, C) \leq \varepsilon$.

We first prove an intermediate result on the convergence of an auxiliary function
\begin{equation*}
    \mathbb{E} \max_{y} \Big[L(x^K, y) - L(x_\star, \bar y^K) - \frac{1}{K} \| y-y_0 \|^2_{\Sigma^{-1}P^{-1}}\Big],
\end{equation*}
where $x^K, \bar y^K$ are defined in \Cref{lem: if3}.
The form of this expression is similar to the expected gap function in~\eqref{eq: opt_cond_def}, the differences being that the maximization is only over $y$ and that there an additional negative quadratic term.
Then in \Cref{th: rates_for_part1} we show how to convert this type of a guarantee to optimality measures for~\eqref{eq: case1_prob} and~\eqref{eq: case2_prob}.

The main novelty in the analysis of these cases, which yields the improved convergence rates is the special use of the full dimensional update $\bar y_{k+1}$. We maintain the running average $\bar y^{K}$ of these vectors as a potential dual output. However, in these special cases, we do not need the output of such a variable, so they are not actually maintained by the algorithm and thus do not incur any implementation costs. See also the discussion at the beginning of Section~\ref{sec: sif3}.

\begin{lemma}\label{lem: if3}
Let~\Cref{asmp: asmp1} hold. In Algorithm~\ref{alg: purecd_dense}, let 
\begin{align*}
    \Tau_k = \Tau = \tau I, ~~~~ \tau = \frac{1}{\sum_{i=1}^n \|A_i\|},~~~~ \theta =1,~~~~ \Sigma_k=\Sigma = \diag(\sigma^{(1)},\dots,\sigma^{(n)}),~~~~ \sigma^{(i)} = \frac{\gamma}{\|A_i\|},
\end{align*}
where $\gamma \in(0,1)$, $p^{(i)} = \frac{\|A_i\|}{\sum_{i=1}^n \|A_i\|}$, and $\Theta_k=\theta I = I$.
Define $x^K = \frac{1}{K} \sum_{k=1}^K \bar x_{k}$, $\bar y^K = \frac{1}{K} \sum_{k=1}^K \bar y_k$. Then we have 
\begin{equation*}
    \mathbb{E} \max_y \Big[L(x^K, y) - L(x_\star, \bar y^K) - \frac{1}{K} \| y-y_0 \|^2_{\Sigma^{-1}P^{-1}}\Big] \leq \frac{\sum_{i=1}^n \|A_i\|}{K\gamma(1-\gamma)} D_\star.
\end{equation*}
\end{lemma}
\begin{proof}[Proof of~\Cref{lem: if3}]
Let $x=x_\star$ in~\eqref{eq: purecd_one_iter} and $y$ be a fixed vector that can be random.
Let us set
\begin{equation}\label{eq: exp_max_err1}
\mathcal{E}_k(y) = \left( - \| y - \bar y_{k+1} \|^2_{\Sigma^{-1}} + \| y-y_{k+1} \|^2_{\Sigma^{-1}P^{-1}} + \| y-y_k\|^2_{\Sigma^{-1}(I-P^{-1})} \right).
\end{equation}
By using the settings $\theta=1$, $\Tau_k =\tau I$, $\Sigma_k = \Sigma$ in \eqref{eq: purecd_one_iter}, adding and subtracting certain terms, and using the definition \eqref{eq: exp_max_err1}, we obtain
\begin{align}
\nonumber
& 2(L(\bar x_{k+1}, y) - L(x_\star, \bar y_{k+1})) + \frac{1}{\tau} \| x_\star-x_{k+1} \|^2 + \| y- y_{k+1}\|^2_{\Sigma^{-1}P^{-1}} \\
\nonumber
& \leq \frac{1}{\tau} \| x_\star-x_k\|^2+\| y-y_k\|^2_{\Sigma^{-1}P^{-1}}
- \| \bar y_{k+1} - y_k\|^2_{\Sigma^{-1}} - \tau \| A^\top P^{-1}(y_{k+1} - y_k)\|^2 \\
\label{eq: purecd_one_iter_2}
& \quad + 2 \langle x_\star-x_{k+1}, A^\top P^{-1}(y_{k+1} - y_k) \rangle
+2 \langle x_\star-\bar x_{k+1}, A^\top(y_k - \bar y_{k+1}) \rangle + \mathcal{E}_k(y).
\end{align}
By summing~\eqref{eq: purecd_one_iter_2} over $k=0,1,\dotsc,K-1$, telescoping, removing nonnegative terms from the LHS, and subtracting  $2\|y-y_0\|^2_{\Sigma^{-1}P^{-1}}$ from both sides, we obtain 
\begin{align*}
& 2\sum_{k=0}^{K-1}\left[L(\bar x_{k+1}, y) - L(x_\star, \bar y_{k+1})\right] - 2\| y-y_0\|^2_{\Sigma^{-1}P^{-1}} \\
& \leq \frac{1}{\tau} \| x_\star-x_0\|^2 
- \sum_{k=0}^{K-1} \left[ \| \bar y_{k+1} - y_k\|^2_{\Sigma^{-1}} + \tau \| A^\top P^{-1}(y_{k+1} - y_k)\|^2 \right] \\
& \quad +\sum_{k=0}^{K-1} \left[ 2 \langle x_\star-x_{k+1}, A^\top P^{-1}(y_{k+1} - y_k) \rangle
+2 \langle x_\star-\bar x_{k+1}, A^\top(y_k - \bar y_{k+1}) \rangle + \mathcal{E}_k(y) \right] - \|y-y_0\|^2_{\Sigma^{-1}P^{-1}}.
\end{align*}
By taking the maximum of both sides over $y$, then taking the expectation of both sides over all random variables in the algorithm, we obtain
\begin{align}
\nonumber
& 2\mathbb{E} \max_y \left[ \sum_{k=0}^{K-1} (L(\bar x_{k+1}, y) - L(x_\star, \bar y_{k+1})) - \| y-y_0\|^2_{\Sigma^{-1}P^{-1} } \right] \\
\nonumber
& \leq \frac{1}{\tau}\| x_\star-x_0\|^2  
+ \mathbb{E} \max_y \left[\sum_{k=0}^{K-1}\mathcal{E}_k(y)-\|y-y_0\|^2_{\Sigma^{-1}P^{-1}}\right]\\
\nonumber
& \quad +\sum_{k=0}^{K-1} \left(
-  \mathbb{E}\| \bar y_{k+1} - y_k \|^2_{\Sigma^{-1}} -  \tau  \mathbb{E} \|A^\top P^{-1}(y_{k+1} - y_k) \|^2\right)  \\
\label{eq:ix1}
& \quad + \sum_{k=0}^{K-1}\mathbb{E}\left[2 \langle x_\star-x_{k+1}, A^\top P^{-1}(y_{k+1} - y_k) \rangle
+2 \langle x_\star-\bar x_{k+1}, A^\top(y_k - \bar y_{k+1}) \rangle \right].
\end{align}
We use \eqref{eq: 3_errors2} (with $\theta=1$) and the tower property of conditional expectation (since the free variable $y$ does not appear in \eqref{eq: 3_errors2}) to substitute for the terms of the final summation of this expression. By combining with the term involving $\tau$ in the third line,
we obtain
\begin{align}
\nonumber
& 2\mathbb{E} \max_y \left[ \sum_{k=0}^{K-1} (L(\bar x_{k+1}, y) - L(x_\star, \bar y_{k+1})) -\| y-y_0 \|^2_{\Sigma^{-1}P^{-1}} \right] \\
\nonumber
& \leq \frac{1}{\tau}\| x_\star-x_0\|^2 
+ \mathbb{E} \max_y \left[ \sum_{k=0}^{K-1}\mathcal{E}_k(y)-\| y-y_0 \|^2_{\Sigma^{-1}P^{-1}}\right] \\
\label{eq:tx2}
& \quad +\sum_{k=0}^{K-1} \left(-  \mathbb{E}\| \bar y_{k+1} - y_k \|^2_{\Sigma^{-1}} +  \tau  \mathbb{E} \|A^\top P^{-1}(y_{k+1} - y_k) \|^2\right).
\end{align}
Because the step sizes $\sigma^{(i)}$ satisfy $\sigma^{(i)} \tau (p^{(i)})^{-1} \|A_i\|^2 \leq \gamma < 1$, we can use \eqref{eq: lem1_error_terms_cancel}  to deduce that 
the final summation on the RHS  of \eqref{eq:tx2} is nonpositive. Thus, we can drop this term and write
\begin{align}
\nonumber
& 2\mathbb{E} \max_y \left[ \sum_{k=0}^{K-1} (L(\bar x_{k+1}, y) - L(x_\star, \bar y_{k+1})) -\| y-y_0 \|^2_{\Sigma^{-1}P^{-1}} \right] \\
\label{eq: ineq_2}
& \leq \frac{1}{\tau}\| x_\star-x_0\|^2 
+ \mathbb{E} \max_y \left[\sum_{k=0}^{K-1}  \mathcal{E}_k(y)-\| y-y_0 \|^2_{\Sigma^{-1}P^{-1}}\right].
\end{align}
For the term $\mathcal{E}_k(y)$ we rearrange the definition~\eqref{eq: exp_max_err1} to obtain
\begin{equation}\label{eq: sk3}
\mathcal{E}_k(y) =  2 \langle y, (\bar y_{k+1} - y_k) - P^{-1}(y_{k+1} - y_k) \rangle_{\Sigma^{-1}} + \|y_{k+1}\|^2_{\Sigma^{-1}P^{-1}} + \| y_k \|^2_{\Sigma^{-1}(I-P^{-1})} - \| \bar y_{k+1}\|^2_{\Sigma^{-1}}.
\end{equation}
Noting that the last three terms in this expression are independent of $y$, we obtain the following for the final summation on the RHS of \eqref{eq: ineq_2}:
\begin{align}
\nonumber
& \mathbb{E} \max_y \left[\sum_{k=0}^{K-1} \mathcal{E}_k(y)-\| y-y_0 \|^2_{\Sigma^{-1}P^{-1}}\right] \\
\nonumber
& \leq \mathbb{E} \max_y \left[ \sum_{k=0}^{K-1} 2 \langle y, (\bar y_{k+1} - y_k) - P^{-1}(y_{k+1} - y_k) \rangle_{\Sigma^{-1}} -\| y-y_0 \|^2_{\Sigma^{-1}P^{-1}} \right] \\
\nonumber
& \quad + \mathbb{E} \left[ \sum_{k=0}^{K-1}\|y_{k+1}\|^2_{\Sigma^{-1}P^{-1}} + \| y_k \|^2_{\Sigma^{-1}(I-P^{-1})} - \| \bar y_{k+1}\|^2_{\Sigma^{-1}} \right] \\
\nonumber
& =\mathbb{E} \max_y \left[ \sum_{k=0}^{K-1} 2 \langle y, (\bar y_{k+1} - y_k) - P^{-1}(y_{k+1} - y_k) \rangle_{\Sigma^{-1}} - \| y-y_0 \|^2_{\Sigma^{-1}P^{-1}}\right] \\
& \leq  \sum_{k=0}^{K-1} \mathbb{E} \| y_k-y_{k+1} \|^2_{\Sigma^{-1}P^{-1}},\label{eq: expmax_term1}
\end{align}
where the equality is due to tower property and by setting $\Phi = \Sigma^{-1}P^{-1}$ and $y=0$ in \Cref{lem: tech_lemma_one_iter}.
The final inequality uses Lemma~\ref{lem: exp_max_lemma} with $u_0 = y_0$, $\underline k = 0$, and $\mathcal{U}=\mathbb{R}^n$.
By substituting \eqref{eq: expmax_term1} into \eqref{eq: ineq_2}, we obtain
\[
 2\mathbb{E} \max_y \left[ \sum_{k=0}^{K-1} (L(\bar x_{k+1}, y) - L(x_\star, \bar y_{k+1})) -\| y-y_0 \|^2_{\Sigma^{-1}P^{-1}} \right] 
 \leq \frac{1}{\tau}\| x_\star-x_0\|^2 
+ \sum_{k=0}^{K-1} \mathbb{E} \| y_k-y_{k+1} \|^2_{\Sigma^{-1}P^{-1}}.
\]
By dividing both sides by $K$ and using the definitions of $\tau$ together with 
\eqref{eq: purecd_sum_residual} in the RHS, we obtain
\begin{align}
\nonumber
& 2\mathbb{E}\max_y \left[ \frac{1}{K}\sum_{k=0}^{K-1} (L(\bar x_{k+1}, y) - L(x_\star, \bar y_{k+1})) - \frac{1}{K} \| y-y_0 \|^2_{\Sigma^{-1}P^{-1}} \right] \\
&\leq \frac{\sum_{i=1}^n \| A_i\| }{K} \|x_\star - x_0 \|^2 + \frac{\sum_{i=1}^n \|A_i\|}{K \gamma(1-\gamma)} \left( \| x_\star - x_0 \|^2 + \| y_\star - y_0\|^2 \right).
\label{eq: purecd_lincons_eq1}
\end{align}
We recall $\bar y^K = \frac{1}{K} \sum_{k=1}^K \bar y_{k}$ and $\bar x^K = \frac{1}{K}\sum_{k=1}^K \bar x_k$, and use convexity of $L$ w.r.t. $x$ and concavity of $L$ w.r.t. $y$ to bound the summation term on the LHS of~\eqref{eq: purecd_lincons_eq1} as follows:
\[
\frac{1}{K}\sum_{k=0}^{K-1} (L(\bar x_{k+1}, y) - L(x_\star, \bar y_{k+1})) \ge  L(x^K, y) - L(x_\star, \bar y^{K}).
\]
By substituting into~\eqref{eq: purecd_lincons_eq1} we obtain
\begin{align}
\mathbb{E} \max_y \Big[L(x^K, y) &- L(x_\star, \bar y^K) - \frac{1}{K} \| y-y_0 \|^2_{\Sigma^{-1}P^{-1}}\Big]
\notag\\
&\leq \frac{\sum_{i=1}^n \| A_i\| }{2K}  \|x_\star - x_0 \|^2   + \frac{\sum_{i=1}^n \|A_i\|}{2K \gamma(1-\gamma)} \left( \| x_\star - x_0 \|^2 + \| y_\star - y_0\|^2 \right) \notag\\
&\leq \frac{\sum_{i=1}^n \| A_i\| }{K\gamma(1-\gamma)} \left( \| x_\star - x_0 \|^2 + \| y_\star - y_0\|^2 \right) = \frac{\sum_{i=1}^n \| A_i\| }{K\gamma(1-\gamma)} D_\star,\label{eq: cvx_ccv_gap_bd_jensen}
\end{align}
where the last inequality is due to $1\leq \frac{1}{\gamma(1-\gamma)}$ since $\gamma \in (0, 1)$ and the equality is by the definition \eqref{eq:Dstar} of $D_\star$.
\end{proof}

Next, we show how to use the result of~\Cref{lem: if3} to derive guarantees for the two special cases~\eqref{eq: case1_prob} and~\eqref{eq: case2_prob}.
\begin{theorem}\label{th: rates_for_part1}
Let~\Cref{asmp: asmp1} hold. In Algorithm~\ref{alg: purecd_dense}, let 
\begin{align*}
    \Tau_k = \Tau = \tau I, ~~~~ \tau = \frac{1}{\sum_{i=1}^n \|A_i\|},~~~~ \Sigma_k=\Sigma = \diag(\sigma^{(1)},\dots,\sigma^{(n)}),~~~~ \sigma^{(i)} = \frac{\gamma}{\|A_i\|},
\end{align*}
where $\gamma \in(0,1)$, $p^{(i)} = \frac{\|A_i\|}{\sum_{i=1}^n \|A_i\|}$, and $\Theta_k=\theta I = I$.
Define $x^K = \frac{1}{K} \sum_{k=1}^K \bar x_{k}$. Then we have the following.
\begin{itemize}
\item Case 1: If $h$ is $L_h$-Lipschitz continuous (see~\eqref{eq: case1_prob}), we have
\begin{equation*}
\mathbb{E} \left[ h(Ax^K) + g(x^K) - h(Ax_\star) - g(x_\star) \right] \leq \frac{2\sum_{i=1}^{n} \|A_i\|}{K\gamma(1-\gamma)} \left( 4L_h^2 + \| x_\star - x_0\|^2 \right).
\end{equation*}
\item Case 2: If $h=\delta_{C}$ with a nonempty convex closed set $C$ of the form $C=C_1\times \dots \times C_n$~(see~\eqref{eq: case2_prob}), we have
\begin{subequations}
\begin{align}
\mathbb{E}   |g(x^K) - g(x_\star)| &\leq \frac{8\sum_{i=1}^n \|A_i\|}{K\gamma(1-\gamma)} \left[\left( \| x_\star - x_0 \| + \| y_\star - y_0\| + \| y_0\| + \|y_\star \| \right)\|y_\star\| + D_\star+\|y_0\|^2\right], \label{eq: ta1}\\
\mathbb{E}  \left[\dist(Ax^K, C)\right]  &\leq \frac{8\sum_{i=1}^n \|A_i\|}{K\gamma(1-\gamma)} \left( \| x_\star - x_0 \| + \| y_\star - y_0\| + \| y_0\| + \|y_\star \| \right).\label{eq: ta2}
\end{align}
\end{subequations}
\end{itemize}
\end{theorem}
\begin{remark}
The results in this theorem correspond to the bounds claimed for PURE-CD in the first two columns of Table~\ref{tb: 1}. 
For the first case, we also use the description in Sec.~\ref{app: upper_lower_bd}.
\end{remark}
\begin{remark}
Although this worst-case upper bound may suggest the choice $\gamma=1/2$, our experience with practical behavior of the algorithm and also the prior literature on PDHG-based methods~\cite{chambolle2011first,chambolle2016ergodic,chambolle2018stochastic,alacaoglu2020random} suggests a choice of $\gamma$ that is closer to $1$, for example,  $\gamma=0.99$.
\end{remark}
\begin{proof}[Proof of~\Cref{th: rates_for_part1}]
In this proof, we will associate the result in~\Cref{lem: if3} to optimality guarantees for~\eqref{eq: case1_prob} and~\eqref{eq: case2_prob}.

$\bullet$ Case 1: We argue as in~\cite[Thm. 11]{fercoq2019coordinate} and \cite{alacaoglu2019convergence}. 
First, we restate the result in~\eqref{eq: cvx_ccv_gap_bd_jensen} by using the definition of $L(\cdot,\cdot)$ from~\eqref{eq:prob}:
\begin{multline}\label{eq: dense_result1_explicit_bd}
    \mathbb{E} \max_y \left[ g(x^K) + \langle A x^K, y \rangle - h^\ast(y) - g(x_\star) - \langle Ax_\star, \bar y^K \rangle + h^\ast(\bar y^K) - \frac{1}{K} \| y-y_0 \|^2_{\Sigma^{-1}P^{-1}}\right] \\
    = \mathbb{E} \max_y \left[L(x^K, y) - L(x_\star, \bar y^K) - \frac{1}{K} \| y-y_0 \|^2_{\Sigma^{-1}P^{-1}}\right]
\leq \frac{\sum_{i=1}^n \| A_i\| }{K\gamma(1-\gamma)} D_\star.
\end{multline}
We now work on the LHS of this inequality to obtain the objective suboptimality.
Since $h\colon\mathbb{R}^n\to\mathbb{R}$ is $L_h$-Lipschitz, we have $\max_{y\in\dom h^\ast} \| y\|^2 \leq L_h^2$ (see, for example,~\cite[Cor. 13.3.3]{rockafellar1970convex},~\cite[Cor. 17.19]{bauschke2011convex}). 
We can also choose $y = \tilde y \in\partial h(Ax^K)\neq\emptyset$~(see~\cite[Thm. 23.4]{rockafellar1970convex}) such that $\tilde y\in \partial h(Ax^K) \iff Ax^K \in \partial h^\ast(\tilde y) \iff \tilde y \in \arg\max_v \langle Ax^K, v \rangle - h^\ast(v)$ and hence 
\begin{equation}\label{eq: jg3}
    \langle Ax^K, \tilde y \rangle - h^\ast(\tilde y) = h(Ax^K).
\end{equation} 
We also have $ \| \tilde y\| \leq L_h$ (see~\cite[Cor. 17.19(i, ii)]{bauschke2011convex}). Next, by the Fenchel-Young inequality, we have
\begin{equation}\label{eq: jg4}
    h^\ast(\bar y^K) - \langle Ax_\star, \bar y^K \rangle \geq -h(Ax_\star).
\end{equation} 
By \eqref{eq: jg3}, \eqref{eq: jg4}, we lower-bound the expression on the first line of~\eqref{eq: dense_result1_explicit_bd}
\begin{multline*}
    \mathbb{E} \max_y \left[ g(x^K) + \langle A x^K, y \rangle - h^\ast(y) - g(x_\star) - \langle Ax_\star, \bar y^K \rangle + h^\ast(\bar y^K) - \frac{1}{K} \| y-y_0 \|^2_{\Sigma^{-1}P^{-1}}\right] \\
    \geq \mathbb{E} \left[ h(Ax^K) + g(x^K) - h(Ax_\star) - g(x_\star) - \frac{1}{K} \| \tilde y - y_0\|^2_{\Sigma^{-1}P^{-1}} \right],
\end{multline*}
which, in view of the inequality in~\eqref{eq: dense_result1_explicit_bd} gives
\begin{align*}
&\mathbb{E} \left[ h(Ax^K) + g(x^K) - h(Ax_\star) - g(x_\star) - \frac{1}{K} \| \tilde y - y_0\|^2_{\Sigma^{-1}P^{-1}} \right] \leq \frac{\sum_{i=1}^n \| A_i\| }{K\gamma(1-\gamma)} D_\star \\
&\iff \mathbb{E} \left[ h(Ax^K) + g(x^K) - h(Ax_\star) - g(x_\star) \right] \leq \frac{\sum_{i=1}^n \| A_i\| }{K\gamma(1-\gamma)} D_\star + \frac{1}{K} \| \tilde y - y_0\|^2_{\Sigma^{-1}P^{-1}}.
\end{align*}
We get the result after using $\max_{y\in \dom h^\ast} \|y\| \leq L_h$ and $\Sigma^{-1} P^{-1} = \frac{\sum_{i=1}^n \|A_i\|}{\gamma} I$ to estimate 
\[
\| \tilde y - y_0 \|^2_{\Sigma^{-1}P^{-1}} = \frac{\sum_{i=1}^n\|A_i\|}{\gamma} \| \tilde y - y_0 \|^2 \leq \frac{\sum_{i=1}^n\|A_i\|}{\gamma} \left(2\| \tilde y\|^2 + 2\| y_0 \|^2\right) \leq \frac{\sum_{i=1}^n\|A_i\|}{\gamma} 4L_h^2.
\]
Similarly, we have 
\[
D_\star = \| x_\star - x_0\|^2 + \| y_\star - y_0\|^2 \leq \| x_0  - x_\star \|^2 + 2\|y_\star \|^2 + 2\|y_0\|^2 \leq \| x_0-x_\star \|^2 + 4L_h^2
\]
and that $\frac{1}{\gamma} \leq \frac{1}{\gamma(1-\gamma)}$ due to $\gamma \in (0, 1)$.

$\bullet$ Case 2: We use the same reasoning as~\cite[Thm. 11]{fercoq2019coordinate},~\cite[Lemma 1]{tran2018smooth},~\cite{alacaoglu2019convergence}. As in Case 1, the argument requires a particular choice for the free variable $y$, depending on the random variable $x^K$, so it is important to have a bound for $\mathbb{E} \max_y L(x^K, y) - L(x_\star, \bar y^K)$.
The inclusion case is handled implicitly on~\cite{tran2018smooth} so we provide a detailed proof for handling this case, which result in somewhat involved expressions.

Let us now estimate the LHS of~\eqref{eq: cvx_ccv_gap_bd_jensen}.
By the definition of conjugate functions and $h(x) = \delta_C(x)$, we have $h^\ast(y) = \max_{u\in C} \langle u, y \rangle$ in this case.
We next derive 
\begin{align}
L(x^K, y) - L(x_\star, \bar y^K) &= g(x^K) -g(x_\star)+ \langle Ax^K, y \rangle - \max_{u\in C}\langle u, y \rangle - \langle Ax_\star, \bar y^K \rangle + \max_{u\in C} \langle u, \bar y^K \rangle\notag \\
&= g(x^K) -g(x_\star)+ \min_{u\in C} \langle Ax^K - u, y \rangle - \min_{u\in C}\langle Ax_\star-u, \bar y^K \rangle \notag \\
&\geq g(x^K) -g(x_\star)+ \min_{u\in C} \langle Ax^K - u, y \rangle,\label{eq: gap_to_obj} 
\end{align}
where the last step is due to $\min_{u\in C} \langle Ax_\star - u, \bar y^K \rangle \leq 0$ since $Ax_\star \in C$.

We use this inequality on the LHS~\eqref{eq: cvx_ccv_gap_bd_jensen} to deduce
\begin{multline*}
    \mathbb{E} \max_y \Big[L(x^K, y) - L(x_\star, \bar y^K) - \frac{1}{K} \| y-y_0 \|^2_{\Sigma^{-1}P^{-1}}\Big] \\
    \geq \mathbb{E}\max_y\left[ g(x^K) -g(x_\star)+ \min_{u\in C} \langle Ax^K - u, y \rangle - \frac{1}{K} \| y-y_0 \|^2_{\Sigma^{-1}P^{-1}} \right].
\end{multline*}
On~\eqref{eq: cvx_ccv_gap_bd_jensen}, we use this inequality with $\sigma^{(i)} p^{(i)} = \frac{\gamma}{\sum_{j=1}^n \|A_j\|}$ for all $i$ and hence $\Sigma^{-1}P^{-1} = \frac{\sum_{j=1}^n \|A_j\|}{\gamma} I$ to derive
\begin{equation}\label{eq: ek3}
    \mathbb{E}\max_y\Big[ g(x^K) - g(x_\star) + \min_{u\in C}\langle Ax^K-u, y \rangle  - \frac{\sum_{i=1}^n\|A_i\|}{K\gamma} \| y-y_0\|^2 \Big] \leq \frac{\sum_{i=1}^n \|A_i\|}{K\gamma(1-\gamma)} D_\star.
\end{equation}
We then use Young's inequality to write
\begin{equation*}
    - \frac{\sum_{i=1}^n\|A_i\|}{K\gamma} \| y-y_0\|^2 \geq - \frac{2\sum_{i=1}^n\|A_i\|}{K\gamma} \| y\|^2 - \frac{2\sum_{i=1}^n\|A_i\|}{K\gamma} \| y_0\|^2.
\end{equation*}
We use this inequality and $\frac{1}{K\gamma} \leq \frac{1}{K\gamma(1-\gamma)}$ in~\eqref{eq: ek3} to obtain
\begin{equation}\label{eq: ek4}
    \mathbb{E}\max_y\Big[ g(x^K) - g(x_\star) + \min_{u\in C}\langle Ax^K-u, y \rangle - \frac{2\sum_{i=1}^n\|A_i\|}{K\gamma} \| y\|^2 \Big] \leq \frac{2\sum_{i=1}^n \|A_i\|}{K\gamma(1-\gamma)} (D_\star+\|y_0\|^2).
\end{equation}
We set $y=\frac{K\gamma}{4\sum_{i=1}^n\|A_i\|}(Ax^K - P_C(Ax^K))$ and use $\|Ax^K - P_C(Ax^K) \|^2 = \dist^2(Ax^K, C)$ to deduce
\begin{align*}
    & \mathbb{E}\max_y\Big[ g(x^K) - g(x_\star) + \min_{u\in C}\langle Ax^K-u, y \rangle - \frac{2\sum_{i=1}^n\|A_i\|}{K\gamma} \| y\|^2 \Big] \\
    & \geq \mathbb{E}\bigg[ g(x^K) - g(x_\star) + \frac{K\gamma}{4\sum_{i=1}^n\|A_i\|} \min_{u\in C}\langle Ax^K-u, Ax^K - P_C(Ax^K) \rangle 
    - \frac{K\gamma}{8\sum_{i=1}^n\|A_i\|} \dist^2(Ax^K, C) \bigg].
\end{align*}
An elementary argument shows that $\min_{u\in C}\langle Ax^K-u, Ax^K - P_C(Ax^K) \rangle$ is solved by $u=P_C(Ax^K)$ with optimal value $\| Ax^K - P_C(Ax^K) \|^2=\dist^2(Ax^K, C)$. 
By substituting into the lower bound above, we obtain
\begin{align*}
    & \mathbb{E}\max_y\Big[ g(x^K) - g(x_\star) + \min_{u\in C}\langle Ax^K-u, y \rangle -  \frac{2\sum_{i=1}^n\|A_i\|}{K\gamma} \| y\|^2 \Big] \\
   &  \geq \mathbb{E}\bigg[ g(x^K) - g(x_\star) + \frac{K\gamma}{8\sum_{i=1}^n\|A_i\|} \dist^2(Ax^K, C) \bigg].
\end{align*}
Combining this bound with~\eqref{eq: ek4} gives
\begin{equation}\label{eq: ek5}
    \mathbb{E}\bigg[ g(x^K) - g(x_\star) + \frac{K\gamma}{8\sum_{i=1}^n\|A_i\|} \dist^2(Ax^K, C)  \bigg] \leq \frac{2\sum_{i=1}^n \|A_i\|}{K\gamma(1-\gamma)} (D_\star+\|y_0\|^2).
\end{equation}
We first find an upper bound for $\mathbb{E}[g(x^K) - g(x_\star)]$.
For this, we use that $\dist^2(Ax^K, C) \geq 0$ to derive
\begin{equation} \label{eq:me3}
    \mathbb{E}\left[ g(x^K) - g(x_\star) \right] \leq \frac{2\sum_{i=1}^n \|A_i\|}{K\gamma(1-\gamma)} (D_\star+\|y_0\|^2).
\end{equation}
We proceed to derive a lower bound for $\mathbb{E}[g(x^K) - g(x_\star)]$ and an upper bound for the feasibility.
By strong duality, the definition of a saddle point, and the Cauchy-Schwarz inequality:
\begin{align*}
    g(x_\star) &= L(x_\star, y_\star) \leq L(x^K, y_\star) = g(x^K) + \langle Ax^K, y_\star \rangle - \max_{u\in C} \langle u, y_\star \rangle \\
    &=g(x^K) + \min_{u\in C}\langle Ax^K- u, y_\star \rangle \leq g(x^K) + \min_{u\in C}\| Ax^K- u\|\| y_\star\| = g(x^K) + \dist(Ax^K, C) \| y_\star\|,
\end{align*}
    from which it follows that
\begin{equation}
    \mathbb{E}[g(x^K) - g(x_\star)] \geq -\mathbb{E}[\dist(Ax^K, C)] \| y_\star \|.\label{eq:wj3}
\end{equation}
We use this inequality on the LHS of~\eqref{eq: ek5} to deduce that
\begin{equation*}
    \mathbb{E}\left[ g(x^K) - g(x_\star) + \frac{K\gamma}{8\sum_{i=1}^n \|A_i\|} \dist^2(Ax^K, C) \right]
    \geq \mathbb{E}\left[ - \dist(Ax^K, C) \| y_\star\| + \frac{K\gamma}{8\sum_{i=1}^n \|A_i\|} \dist^2(Ax^K, C) \right],
\end{equation*}
which, in view of~\eqref{eq: ek5} gives
\begin{equation}\label{eq: lin_cons_eq2}
\frac{K\gamma}{8\sum_{i=1}^n \|A_i\|}\mathbb{E}\left[\dist^2(Ax^K, C)\right] - \| y_\star \| \mathbb{E}\left[\dist(Ax^K, C)\right] \leq \frac{2\sum_{i=1}^n \|A_i\|}{K\gamma(1-\gamma) } \left(D_\star + \| y_0\|^2 \right).
\end{equation}
Recall that
$\dist(Ax^K, C) = \| Ax^K - P_C(Ax^K)\|$.
We use $\mathbb{E} \left[\dist^2( Ax^K, C)\right] \geq \left(\mathbb{E}\left[\dist(Ax^K, C)\right]\right)^2$ by Jensen's inequality and complete the square on LHS by adding to both sides $\frac{2\sum_{i=1}^n\|A_i\|}{K\gamma} \| y_\star \|^2$ and using $\frac{1}{\gamma} \|y_\star \|^2 \leq \frac{1}{\gamma(1-\gamma)}\|y_\star\|^2$ due to $\gamma > 0$ to derive
\begin{align*}
\left( \sqrt{\frac{K\gamma}{8\sum_{i=1}^n \|A_i\|}}\mathbb{E}\dist(Ax^K, C) - \sqrt{\frac{2\sum_{i=1}^n \|A_i\|}{K\gamma}} \| y_\star \| \right)^2 \leq \frac{2\sum_{i=1}^n \|A_i\|}{K\gamma(1-\gamma)} \left(D_\star + \|y_0\|^2 + \| y_\star \|^2\right).
\end{align*}
Taking the square root of both sides, we have
\begin{align*}
 \sqrt{\frac{K\gamma}{8\sum_{i=1}^n \|A_i\|}}\mathbb{E}\dist(Ax^K, C) \leq \sqrt{\frac{2\sum_{i=1}^n \|A_i\|}{K\gamma(1-\gamma)} (D_\star+ \|y_0\|^2 + \| y_\star \|^2)} + \sqrt{\frac{2\sum_{i=1}^n \|A_i\|}{K\gamma}} \| y_\star \|.
\end{align*}
Now we multiply both sides with $\sqrt{\frac{8\sum_{i=1}^n\|A_i\|}{K\gamma}}$, use 
\[
\sqrt{D_\star} = \sqrt{\|x_\star - x_0\|^2 + \| y_0 - y_\star \|^2} \leq \| x_0 - x_\star \| + \| y_0 - y_\star \|,
\]
and use $1\leq\frac{1}{\sqrt{1-\gamma}}\leq \frac{1}{1-\gamma}$ (by $\gamma \in (0, 1)$) to obtain the feasibility bound, which is 
\begin{align}
\mathbb{E}\dist(Ax^K, C) &\leq \frac{4\sum_{i=1}^n \|A_i\|}{K\gamma\sqrt{1-\gamma}} \left( \| x_\star - x_0 \| + \| y_\star - y_0\| + \| y_0 \| + \|y_\star \| \right) + \frac{4\sum_{i=1}^n \|A_i\|}{K\gamma} \| y_\star \| \notag \\
&\leq \frac{8\sum_{i=1}^n \|A_i\|}{K\gamma(1-\gamma)} \left( \| x_\star - x_0 \| + \| y_\star - y_0\| + \| y_0\| + \|y_\star \| \right),\label{eq: pq1}
\end{align}
giving \eqref{eq: ta2}. 

To obtain~\eqref{eq: ta1}, we combine the lower bound from~\eqref{eq:wj3} with the upper bound in~\eqref{eq:me3} and use the upper bound for the feasibility in~\eqref{eq: pq1}.
\end{proof}

\subsubsection{General Convex-Concave Problems}\label{sec: sif3}
In this section, we focus on problems that are convex-concave but possibly not included in the special cases considered in \Cref{sec: lincons_erm}.
In this case, in contrast to standard PURE-CD, we need to make use of  ideas from~\cite{song2021variance}.
The main reason for this is the following.
In the previous section, since we needed to output only primal vector, we could work with the average of $\bar y_k$ in the analysis and then eliminate it,  significantly simplifying our treatment of the duality gap.
However, when we wish to output a dual vector, we need to switch $\bar y_k$ in the definition of duality gap to the computed vector $y_k$.
This creates a difficulty that can be handled using the techniques of~\cite{song2021variance}.

We also use ideas from~\cite{alacaoglu2019convergence,alacaoglu2020random} to obtain results in the expected gap.
It is not clear to us if these ideas can be used in the framework of~\cite{song2021variance} to strengthen the results given therein.
The main difference is that our algorithm and the algorithms analyzed  in~\cite{alacaoglu2019convergence,alacaoglu2020random} build on gradient descent whereas~\cite{song2021variance} builds on dual averaging.

In particular, we add deterministic initialization, and use step sizes from $\tau_k=\tau a_k$ and $\sigma=\sigma a_k$ as in~\cite{song2021variance} for special choices of $a_k$.
To avoid notational clashes, we use $\lambda_k$ in place of $a_k$ used in \cite{song2021variance}, and $\Lambda_k$ in place of $A_k$ used in \cite{song2021variance}.

Let $\lambda_0 = 1$, $\lambda_1 = \frac{1}{n-1}$, and define $\lambda_{k+1} = \min\left( 1, \frac{n}{n-1} \lambda_k \right)$ for $k \ge 1$. 
We also set $\tau = \frac{1}{n\max_i \|A_i\|}$, and for some $\gamma \in (0,1)$, we set $\sigma^{(i)} = \frac{\gamma}{\|A_i\|}$ and 
$\underline\sigma=\frac{\min_i\sigma^{(i)}}{n} = \frac{\gamma }{n\max_i \|A_i \|}$. 
We replace the first iteration  of \Cref{alg: purecd_dense} by 
\begin{equation}\label{eq: purecd_init}
\begin{aligned}
\bar x_1 = x_1=\prox_{\tau, g}(x_0 - \tau  A^\top y_0)\\
y_1 = \prox_{\underline\sigma, h^\ast}(y_0 + \underline\sigma  A x_1).
\end{aligned}
\end{equation}
The convergence result is as follows.
\begin{theorem}\label{th: rates_for_part1.2}
Let~\Cref{asmp: asmp1} hold and $n\geq 2$. Use \Cref{alg: purecd_dense} with the modification that the first iteration ($k=0$) is replaced by~\eqref{eq: purecd_init}, with the iterations $k\geq 1$ proceedings exactly as before. 
Let
\begin{align*}
    \Tau_k = \tau\lambda_k I, ~~~ \tau = \frac{1}{n\max_i \|A_i\|}, ~~~ \Sigma_k =\lambda_k \Sigma = \lambda_k\diag(\sigma^{(1)}, \dots, \sigma^{(n)}), ~~~ \sigma^{(i)} = \frac{\gamma}{\|A_i\|}, ~~~  \Theta_k=I, ~~~ p^{(i)}=\frac{1}{n},
\end{align*}
 with $\theta=1$, $\gamma < 1$, $\lambda_0 = 1$, $\lambda_1 = \frac{1}{n-1}$, and  $\lambda_{k+1} = \min\left( 1, \frac{n}{n-1} \lambda_k \right)$ for $k\geq 1$. 
 Define the following quantities
 \begin{equation*}
   \Lambda_K = \sum_{k=0}^{K-1} \lambda_k, \quad 
   x^K = \frac{1}{\Lambda_K} \sum_{k=0}^{K-1} \lambda_k \bar x_{k+1}, ~~~~y^K = \frac{n\lambda_{K-1} y_{K} + \sum_{k=1}^{K-2}(n\lambda_k - (n-1)\lambda_{k+1})y_{k+1}}{\Lambda_K}.
 \end{equation*} 
 Then, it follows that for $K_0$ satisfying $(n-1)\log(n-1)\leq K_0\leq 1+n\log(n-1)$ and any compact set $\mathcal{Z}=\mathcal{X}\times\mathcal{Y}$,
\begin{equation*}
\mathbb{E} \max_{(x, y)\in\mathcal{Z}}  \left[ L(x^K, y) - L(x, y^K) \right] \leq \begin{cases} \left(1+\frac{1}{n-1}\right)^{-K} \frac{6n\max_i\|A_i\|}{\gamma(1-\gamma)},  & \mbox{if  $K\leq K_0$}, \\  \frac{6n\max_i \|A_i\|D_{\mathcal{Z}}}{(K-K_0-1 + (n-1)^2/n)\gamma(1-\gamma)}  , & \mbox{otherwise,} \end{cases}
\end{equation*}
where $D_{\mathcal{Z}}$ is defined in \eqref{eq:DZ}.
Hence, the number of iterations $K$ to ensure  that $\mathbb{E}\Gap(x^K, y^K) \leq \varepsilon$ is $O\left((n+n\min\{\log n, \log\varepsilon^{-1}\} + \varepsilon^{-1} D_{\mathcal{Z}} n\max_i \|A_i\| )d \right)$ or $\tilde O\left( nd+nd\max_i\|A_i\|D_{\mathcal{Z}} \varepsilon^{-1}\right)$.
\end{theorem}
\begin{remark}
The result in this theorem corresponds to the entry for PURE-CD in the third column of Table~\ref{tb: 1}.
It also implies the result in the last column of Table~\ref{tb: 1}.
\end{remark}
\begin{remark}
This theorem shows that the result in~\cite{song2021variance} which was for $\max_{(x,y)} \mathbb{E} G(x^K, y^K, x,y)$ can be strengthened to expected duality gap $ \mathbb{E}\max_{(x,y)\in \mathcal{Z}} G(x^K, y^K, x,y)$ for PURE-CD and that the initialization technique introduced in~\cite{song2021variance} is remarkably general. A byproduct of the PURE-CD analysis is that the step size rule is more flexible, that is, the product of two step sizes can be twice as large as for VRPDA in~\cite{song2021variance}. It also shows that dual averaging is not essential for using this technique and obtaining improved complexity bounds.
\end{remark}

\begin{proof}[Proof of~\Cref{th: rates_for_part1.2}]
In~\eqref{eq: purecd_one_iter} (note that this result holds for $k\geq 1$ for the algorithm described in this theorem), we use $\Tau_k = \tau \lambda_k I$, $\Sigma_k = \lambda_k \Sigma$, $\theta= 1$, and $P^{-1} = n I$, and multiply both sides of the resulting inequality with $\lambda_k$ to deduce for $k\geq 1$
\begin{equation}\label{eq: gen_cvx_ccv_eq1}
\begin{aligned}
& 2\lambda_k(L(\bar x_{k+1}, y) - L(x, \bar y_{k+1})) + \frac{1}{\tau} \| x-x_{k+1} \|^2 + \| y- \bar y_{k+1}\|^2_{\Sigma^{-1}} \\
& \leq \frac{1}{\tau} \| x-x_k\|^2+\| y-y_k\|^2_{\Sigma^{-1}} - \| \bar y_{k+1} - y_k\|^2_{\Sigma^{-1}} \\ 
& \quad - \tau\lambda_k^2n^2 \| A^\top (y_{k+1} - y_k)\|^2 + 2n\lambda_k \langle x-x_{k+1}, A^\top (y_{k+1} - y_k) \rangle
+2\lambda_k \langle x-\bar x_{k+1}, A^\top(y_k - \bar y_{k+1}) \rangle.
\end{aligned}
\end{equation}
To handle the coupling of expectation and supremum, we introduce three error terms:
\begin{subequations}
\begin{align}
\mathcal{E}_{k, 1}(y) &= - \| y - \bar y_{k+1} \|^2_{\Sigma^{-1}} + n\| y-y_{k+1} \|^2_{\Sigma^{-1}} + (1-n) \| y-y_k\|^2_{\Sigma^{-1}},\label{eq: err_term1}\\
\mathcal{E}_{k, 2}(x) &= 2\lambda_k\langle x-\bar x_{k+1}, A^\top(y_k - \bar y_{k+1}) \rangle - 2n\lambda_k \langle x-\bar x_{k+1}, A^\top(y_k - y_{k+1}) \rangle \notag \\
&= 2\lambda_k\langle x-\bar x_{k+1}, A^\top(y_k - \bar y_{k+1}) - n A^\top(y_k - y_{k+1}) \rangle. \label{eq: err_term2} \\
\mathcal{E}_{k, 3}(x) &= \lambda_k\left(- \langle x, A^\top \bar y_{k+1} \rangle +  \langle x, A^\top(ny_{k+1} - (n-1) y_k) \rangle + h^\ast(\bar y_{k+1}) - nh^\ast(y_{k+1}) + (n-1) h^\ast(y_k)\right).\label{eq: err_term3}
\end{align}
\end{subequations}
We first note that for deterministic $x$ and $y$, all these terms would be zero-mean
However, we have to bound their expectations after taking supremum over random variables $x, y$ by using~\Cref{lem: exp_max_lemma}.

In~\eqref{eq: gen_cvx_ccv_eq1}, we use the definitions in~\eqref{eq: err_term1},~\eqref{eq: err_term2},~\eqref{eq: g_from_l} by adding and subtracting certain terms, to write
\begin{equation}\label{eq: cvx_ccv_gap_eq5}
\begin{aligned}
& 2 \lambda_k G(\bar x_{k+1}, \bar y_{k+1}, x, y) + \frac{1}{\tau}\| x-x_{k+1} \|^2 + n\| y-y_{k+1}\|^2_{\Sigma^{-1}} \\
& \leq \frac{1}{\tau}\| x-x_k\|^2 + n\| y-y_k \|^2_{\Sigma^{-1}} 
 - \| \bar y_{k+1} - y_k\|^2_{\Sigma^{-1}}  +  \mathcal{E}_{k, 1}(y)+\mathcal{E}_{k, 2}(x) \\
& \quad - \tau \lambda_k^2 n^2 \| A^\top(y_{k+1} - y_k) \|^2 +2n\lambda_k\langle x-x_{k+1}, A^\top(y_{k+1} - y_k) \rangle + 2n\lambda_k \langle x-\bar x_{k+1}, A^\top(y_k - y_{k+1}) \rangle.
\end{aligned}
\end{equation}
Since $\bar x_{k+1} - x_{k+1} = \tau \lambda_k n A^\top(y_{k+1} -y_k)$ from step~\ref{st: step_extrap_dense}, for $k\geq 1$, we have in similar fashion to ~\eqref{eq: 3_errors2} that
\begin{align*}
    2n\lambda_k\langle x-x_{k+1}, A^\top(y_{k+1} - y_k) \rangle + 2n\lambda_k \langle x-\bar x_{k+1}, A^\top(y_k - y_{k+1}) \rangle & = 2n\lambda_k\langle \bar x_{k+1}-x_{k+1}, A^\top(y_{k+1} - y_k) \rangle\\
    & =2\tau n^2\lambda_k^2 \|A^\top(y_{k+1} - y_k)\|^2,
\end{align*}
which can be combined with the third last term in the RHS of~\eqref{eq: cvx_ccv_gap_eq5}.
We then sum the resulting inequality for $k=1, 2, \dots, K-1$ to derive
\begin{equation}\label{eq: end_of_std}
\begin{aligned}
& 2\sum_{k=1}^{K-1} \lambda_k G(\bar x_{k+1}, \bar y_{k+1}, x, y) + \frac{1}{\tau}\| x-x_{K+1} \|^2 + n\| y-y_{K+1}\|^2_{\Sigma^{-1}} \\
& \leq \frac{1}{\tau}\| x-x_1\|^2 + n\| y-y_1 \|^2_{\Sigma^{-1}} \\
& \quad +\sum_{k=1}^{K-1} \left( - \| \bar y_{k+1} - y_k\|^2_{\Sigma^{-1}} + \tau n^2 \lambda_k^2 \| A^\top (y_{k+1} - y_k) \|^2 \right) + \sum_{k=1}^{K-1} \left(\mathcal{E}_{k, 1}(y)+\mathcal{E}_{k, 2}(x)\right).
\end{aligned}
\end{equation}
Next, we lower-bound the term involving $G$ on the LHS of this expression by~\Cref{lem: lemma_swd21_trick}, which uses the technique of~\cite{song2021variance} to derive
\begin{multline*}
\sum_{k=1}^{K-1} \lambda_k G(\bar x_{k+1}, \bar y_{k+1}, x, y)  \geq \Lambda_K G(x^K, y^K, x, y) +\sum_{k=1}^{K-1} \mathcal{E}_{k, 3}(x) \\
+ \frac{n\max_i\|A_i\|}{2 \gamma } \Big(\| y-y_1\|^2 - \| y-y_0 \|^2 \Big)
 + \frac{n\max_i\|A_i\|}{2}  \| x-x_1\|^2 -n\max_i\|A_i\| \| x-x_0\|^2.
\end{multline*}
We use this estimate on LHS of~\eqref{eq: end_of_std} and drop nonnegative terms on the LHS to obtain
\begin{align*}
    & 2\Lambda_K G(x^K, y^K, x, y) \\
    & \leq 
    \frac{1}{\tau} \| x-x_1\|^2 + n\| y-y_1 \|^2_{\Sigma^{-1}} \\
    & \quad - \frac{n\max_i\|A_i\|}{ \gamma } \Big(\| y-y_1\|^2 - \| y-y_0 \|^2 \Big)\\
 & \quad - n\max_i\|A_i\|  \| x-x_1\|^2 + 2n\max_i\|A_i\| \| x-x_0\|^2
+\sum_{k=1}^{K-1} \left[ - \| \bar y_{k+1} - y_k\|^2_{\Sigma^{-1}} + \tau n^2 \lambda_k^2 \| A^\top (y_{k+1} - y_k) \|^2 \right] \\
& \quad + \sum_{k=1}^{K-1}\left( \mathcal{E}_{k, 1}(y)+\mathcal{E}_{k, 2}(x)-2\mathcal{E}_{k, 3}(x)\right),
\end{align*}
where $\mathcal{E}_{k, 3}$ is defined in~\eqref{eq: err_term3}.
By taking the maximum over $z=(x, y) \in \mathcal{Z}$ and the expectation, and using the definitions $\tau=\frac{1}{n\max_i\|A_i\|}$ and $\sigma^{(i)}=\frac{\gamma}{\|A_i\|}$ 
to cancel the terms $\|x-x_1\|^2 + \|y-y_1\|^2$, we obtain
\begin{equation}\label{eq: end_of_std_2}
\begin{aligned}
2\mathbb{E}\max_{(x,y)\in\mathcal{Z}}\Lambda_K G(x^{K}, y^{K}, x, y) & \leq \frac{3 n \max_i \|A_i\|}{\gamma} \max_{(x,y)\in\mathcal{Z}}\left( \| x-x_0\|^2 + \|y-y_0\|^2\right) \\
& \quad +\mathbb{E}\max_{(x,y)\in\mathcal{Z}} \sum_{k=1}^{K-1} \left(\mathcal{E}_{k, 1}(y)+\mathcal{E}_{k, 2}(x)-2\mathcal{E}_{k, 3}(x)\right) \\
& \quad +\sum_{k=1}^{K-1} \mathbb{E}\left[ - \| \bar y_{k+1} - y_k\|^2_{\Sigma^{-1}} + \tau n^2 \lambda_k^2 \| A^\top (y_{k+1} - y_k) \|^2 \right].
\end{aligned}
\end{equation}
By $\lambda_k\leq 1$, the tower property, \eqref{eq: purecd_error_sparsity}, \eqref{eq: lem1_error_terms_cancel}, and the step size rules, each term in the final summation on the RHS of~\eqref{eq: end_of_std_2} is nonpositive, as the following argument shows.
Since $y_{k+1} - y_k$ is one-sparse for $k\geq 1$, we can derive as~\cref{eq: lem1_error_terms_cancel,eq: purecd_error_sparsity} by using $p^{(i)}=\frac{1}{n}$ and the step size rule $\tau \sigma^{(i)}(p^{(i)})^{-1} \| A_i\|^2 \leq \gamma < 1$,
\begin{align}
    \mathbb{E}\left[ - \| \bar y_{k+1} - y_k\|^2_{\Sigma^{-1}} + \tau n^2 \lambda_k^2 \| A^\top (y_{k+1} - y_k) \|^2 \right] &= \mathbb{E}\left[ - \| \bar y_{k+1} - y_k\|^2_{\Sigma^{-1}} + \tau n^2 \lambda_k^2 \mathbb{E}_k\| A^\top (y_{k+1} - y_k) \|^2 \right] \notag \\
    &\leq \mathbb{E}\left[ - \| \bar y_{k+1} - y_k\|^2_{\Sigma^{-1}} + \tau n^2  \mathbb{E}_k\| A^\top (y_{k+1} - y_k) \|^2 \right] \notag\\
    &\leq -(1-\gamma)n \mathbb{E}\left[ \mathbb{E}_k\|y_{k+1} - y_k\|^2_{\Sigma^{-1}}\right] \leq 0, \label{eq: rs3}
\end{align}
as claimed.
Thus, from \eqref{eq: end_of_std_2}, and using the definition of $D_{\mathcal{Z}}$ from \Cref{sec: notation}, we obtain
\begin{align}\label{eq: gap_bd_fin1}
\mathbb{E}\max_{(x,y)\in\mathcal{Z}} G(x^K, y^K, x, y) \leq  \frac{2n\max_i\|A_i\|}{\gamma\Lambda_K} D_{\mathcal{Z}}  + \frac{1}{2\Lambda_K} \mathbb{E}\max_{x, y} \sum_{k=1}^{K-1} \left(\mathcal{E}_{k, 1}(y)+\mathcal{E}_{k, 2}(x) - 2\mathcal{E}_{k, 3}(x)\right).
\end{align}

We need to estimate the last three error terms.
Note first that by the definitions on~\eqref{eq: err_term2},~\eqref{eq: err_term3}
\begin{align}
\mathbb{E}\max_{x, y}\sum_{k=1}^{K-1} \left(\mathcal{E}_{k, 2}(x) - 2\mathcal{E}_{k,3}(x)\right) &= -2\mathbb{E} \sum_{k=1}^{K-1} \lambda_k \langle \bar x_{k+1}, A^\top(y_k - \bar y_{k+1}) - nA^\top(y_k - y_{k+1})  \rangle \notag \\
&\qquad -2 \mathbb{E} \sum_{k=1}^{K-1} \lambda_k \left[ h^\ast(\bar y_{k+1}) - n h^\ast(y_{k+1}) +(n-1) h^\ast(y_k) \right] \notag \\
&= -2 \sum_{k=1}^{K-1} \lambda_k \mathbb{E}\left[\mathbb{E}_k\langle \bar x_{k+1}, A^\top(y_k - \bar y_{k+1}) - nA^\top(y_k - y_{k+1})  \rangle\right] \notag \\
&\qquad -2 \sum_{k=1}^{K-1} \lambda_k \mathbb{E}\left[ \mathbb{E}_k\left[ h^\ast(\bar y_{k+1}) - n h^\ast(y_{k+1}) +(n-1) h^\ast(y_k) \right]\right], \label{eq: sd3} \\
&=0, \notag
\end{align}
where we used the tower property, $\bar x_{k+1}$ being deterministic under $\mathbb{E}_k$, the identity $A^\top(y_k - \bar y_{k+1}) - \mathbb{E}_k\left[nA^\top(y_k - y_{k+1})\right] = 0$ to show that the first term on the RHS of~\eqref{eq: sd3} is $0$ and the third result in~\Cref{lem: tech_lemma_one_iter} to show that the second term in the RHS of~\eqref{eq: sd3} is $0$, since $\E_k h^*(y_{k+1}) = \frac{1}{n} h^*(\bar y_{k+1}) + (1-\frac{1}{n}) h^*(y_k)$.

It remains to estimate $\mathbb{E}\max_{x, y}\sum_{k=1}^{K-1} \mathcal{E}_{k, 1}(y)$.
Starting from~\eqref{eq: err_term1}, we expand the squared norms to get (similar to~\eqref{eq: sk3}) that
\[
    \mathcal{E}_{k, 1}(y) = 2\langle y, \bar y_{k+1} - n y_{k+1} - (1-n) y_k \rangle_{\Sigma^{-1}} - \|\bar y_{k+1} \|^2_{\Sigma^{-1}} +n\|y_{k+1} \|^2_{\Sigma^{-1}} + (1-n) \| y_k\|^2_{\Sigma^{-1}}.
\]
Similar to~\eqref{eq: expmax_term1}, we obtain
\begin{align*}
& \mathbb{E} \max_{y\in\mathcal{Y}} \left[\sum_{k=1}^{K-1} \mathcal{E}_{k, 1}(y) - n\| y-y_1 \|^2_{\Sigma^{-1}}\right] \\
& \leq \mathbb{E} \max_{y\in\mathcal{Y}} \left[ \sum_{k=1}^{K-1} 2 \langle y, (\bar y_{k+1} - y_k) - n(y_{k+1} - y_k) \rangle_{\Sigma^{-1}} - n\| y-y_1 \|^2_{\Sigma^{-1}} \right] \\
& \quad + \mathbb{E} \left[ \sum_{k=1}^{K-1} - \|\bar y_{k+1} \|^2_{\Sigma^{-1}} +n\|y_{k+1} \|^2_{\Sigma^{-1}} + (1-n) \| y_k\|^2_{\Sigma^{-1}} \right] \\
& =\mathbb{E} \max_{y\in\mathcal{Y}} \left[ \sum_{k=1}^{K-1} 2 \langle y, (\bar y_{k+1} - y_k) - n(y_{k+1} - y_k) \rangle_{\Sigma^{-1}} - n\| y-y_1 \|^2_{\Sigma^{-1}} \right] \\
& \leq  \sum_{k=1}^{K-1} n\mathbb{E} \| y_k-y_{k+1} \|^2_{\Sigma^{-1}},
\end{align*}
where the equality is due to $\mathbb{E}_k \|y_{k+1}\|^2_{\Sigma^{-1}} = \frac{1}{n} \|\bar y_{k+1}\|^2_{\Sigma^{-1}} + \left( 1 - \frac{1}{n} \right) \| y_k\|^2_{\Sigma^{-1}}$ (\Cref{lem: tech_lemma_one_iter}) and the last inequality is the application of \Cref{lem: exp_max_lemma} with $P=1/n$, $u_0 = y_0$, $\underline k = 1$, and $\mathcal{U}=\mathcal{Y}$.
It follows from this bound that 
\begin{align}
\nonumber
\mathbb{E} \max_{y\in\mathcal{Y}} \sum_{k=1}^{K-1} \mathcal{E}_{k,1}(y) 
& \leq \max_{y\in\mathcal{Y}} n \| y-y_1\|^2_{\Sigma^{-1}} + \mathbb{E} \max_{y\in\mathcal{Y}} \left[\sum_{k=1}^{K-1} \mathcal{E}_{k, 1}(y) - n\| y-y_1 \|^2_{\Sigma^{-1}}\right]  \\
&\leq \max_{y\in\mathcal{Y}} n\| y-y_1\|^2_{\Sigma^{-1}} + \sum_{k=1}^{K-1} n\mathbb{E} \| y_k-y_{k+1} \|^2_{\Sigma^{-1}} \notag\\
&\leq \max_{y\in\mathcal{Y}} \frac{n\max_i \|A_i\|}{\gamma} \| y-y_1\|^2+ \sum_{k=1}^{K-1} n\mathbb{E} \| y_k-y_{k+1} \|^2_{\Sigma^{-1}} \notag\\
&\leq  \max_{y\in\mathcal{Y}} \frac{2n\max_i \|A_i\|}{\gamma} \| y-y_0\|^2 + \frac{2n\max_i \|A_i\|}{\gamma} \| y_1-y_0\|^2 +  \sum_{k=1}^{K-1} n\mathbb{E} \| y_k-y_{k+1} \|^2_{\Sigma^{-1}} \notag\\
&\leq \frac{2n\max_i\|A_i\|}{\gamma} D_{\mathcal{Z}} + \frac{4n\max_i\|A_i\|}{\gamma(1-\gamma)} D_\star + \sum_{k=1}^{K-1} n \mathbb{E} \| y_k-y_{k+1} \|^2_{\Sigma^{-1}},\label{eq: exp_max_gap_eq1}
\end{align}
where the last step is by the definition of $D_{\mathcal{Z}}$ from \Cref{sec: notation}, $\sigma_i^{-1} = \|A_i\|/\gamma \leq \max_i \|A_i\|/\gamma$ and the third result in Lemma~\ref{lem: lem_init}.
To finish upper bounding this term, we need to estimate $\sum_{k=1}^{K-1} n \mathbb{E} \| y_k-y_{k+1} \|^2_{\Sigma^{-1}}$.

For this task, we proceed as in the derivation of the second result in~\Cref{lem: lem1_purecd}. 
By letting $(x, y) = (x_\star, y_\star)$ and taking expectation of~\eqref{eq: end_of_std}, we have
\begin{multline}\label{eq: end_of_std2}
2\mathbb{E}\sum_{k=1}^{K-1} \lambda_k G(\bar x_{k+1}, \bar y_{k+1}, x_\star, y_\star) \leq \frac{1}{\tau}\| x_\star-x_1\|^2 + n\| y_\star-y_1 \|^2_{\Sigma^{-1}} \\
+\mathbb{E}\sum_{k=1}^{K-1} \left( - \| \bar y_{k+1} - y_k\|^2_{\Sigma^{-1}} + \tau n^2 \lambda_k^2 \| A^\top (y_{k+1} - y_k) \|^2 \right) + \mathbb{E}\sum_{k=1}^{K-1} \left(\mathcal{E}_{k, 1}(y_\star)+\mathcal{E}_{k, 2}(x_\star)\right).
\end{multline}
Now we can use $G(\bar x_{k+1}, \bar y_{k+1}, x_\star, y_\star) \geq 0$ from the definition of $G$ from~\eqref{eq: g_from_l}.
Moreover, from \eqref{eq: err_term1} and \eqref{eq: err_term2}, we have $\mathbb{E}[\mathcal{E}_{k, 1}(y_\star)]=0$ and $\mathbb{E}[\mathcal{E}_{k, 2}(x_\star)]=0$.
The former claim follows from the second result of Lemma~\ref{lem: tech_lemma_one_iter}. For the latter, we use the fact that for any fixed $x$ (in this case, $x_\star$), under the conditioning of $\mathbb{E}_k$, $\bar x_{k+1}$ is deterministic. Thus, we have
\begin{align*}
    \mathbb{E}\left[ \mathcal{E}_{k, 2}(x_\star) \right] &= \mathbb{E}\left[ \mathbb{E}_k [\mathcal{E}_{k, 2}(x_\star)] \right] = 2\lambda_k\mathbb{E}\left[ \mathbb{E}_k \langle x_\star - \bar x_{k+1}, A^\top(y_k - \bar y_{k+1}) - n A^\top(y_k - y_{k+1}) \rangle \right] \\
    &=2\lambda_k\mathbb{E}\left[ \langle x_\star - \bar x_{k+1}, A^\top(y_k - \bar y_{k+1}) - \mathbb{E}_k[n A^\top (y_k - y_{k+1})] \rangle \right] = 0,
\end{align*}
since $\mathbb{E}_k[n(y_k - y_{k+1})] = y_k - \bar y_{k+1}$.

Finally, we use the derivations in~\cref{eq: lem1_error_terms_cancel,eq: purecd_error_sparsity} for $k\geq 1$, since $y_{k+1} - y_k$ is one-sparse, $\lambda_k\leq1$ and tower property, for the first summation in the last line of~\eqref{eq: end_of_std2} to get for $k\geq 1$ that (see also~\eqref{eq: rs3})
\begin{align*}
    - \| \bar y_{k+1} - y_k\|^2_{\Sigma^{-1}} + \tau n^2 \lambda_k^2 \mathbb{E}_k \| A^\top (y_{k+1} - y_k) \|^2 & \leq - \| \bar y_{k+1} - y_k\|^2_{\Sigma^{-1}} + \tau n^2 \mathbb{E}_k\| A^\top (y_{k+1} - y_k) \|^2 \\
    & \leq - (1-\gamma) \mathbb{E}_k n\|y_{k+1} - y_k\|^2_{\Sigma^{-1}}.
\end{align*}
By using these estimations in~\eqref{eq: end_of_std2}, we have
\begin{equation*}
 \sum_{k=1}^{K-1} (1-\gamma) \mathbb{E}_k n\|y_{k+1} - y_k\|^2_{\Sigma^{-1}} \leq \frac{1}{\tau}\| x_\star-x_1\|^2 + n\| y_\star-y_1 \|^2_{\Sigma^{-1}}  = n\max_i\|A_i\|\left( \|x_\star - x_1\|^2 + \frac{1}{\gamma}\| y_\star - y_1\|^2 \right).
\end{equation*}
By using the step size choices in this theorem, we  have from the second result in~\Cref{lem: lem_init} that
\begin{equation*}
    \sum_{k=1}^{K-1} n\mathbb{E} \| y_{k+1} -y_k \|^2_{\Sigma^{-1}} \leq \frac{n\max_i \|A_i\|\left( \| x_\star - x_1 \|^2 + \gamma^{-1}\| y_\star - y_1 \|^2 \right)}{1-\gamma} \leq \frac{2n\max_i\|A_i\|}{\gamma(1-\gamma)} D_\star.
\end{equation*}
With this bound,~\eqref{eq: exp_max_gap_eq1} becomes
\begin{equation}
\mathbb{E} \max_{y\in\mathcal{Y}} \sum_{k=1}^{K-1} \mathcal{E}_{k,1}(y) 
\leq  \frac{2n\max_i\|A_i\|}{\gamma} D_{\mathcal{Z}} + \frac{6n\max_i\|A_i\|}{\gamma(1-\gamma)} D_{\star}.
\end{equation}
By combining the results for the last summation in~\eqref{eq: gap_bd_fin1}, we obtain
\begin{equation}\label{eq:sh7}
    \mathbb{E}\max_{(x,y)\in\mathcal{Z}} G(x^K, y^K, x, y) \leq \frac{1}{\Lambda_K}\left( \frac{3n\max_i\|A_i\|}{\gamma}D_{\mathcal{Z}} + \frac{3n\max_i\|A_i\|}{\gamma(1-\gamma)}D_\star \right) \leq \frac{6n\max_i\|A_i\|}{\Lambda_K\gamma(1-\gamma)}D_{\mathcal{Z}}.
\end{equation}
We have to lower-bound $\Lambda_K=\sum_{k=0}^{K-1} \lambda_k$, where $\lambda_0=1$, $\lambda_1=(n-1)^{-1}$ and $\lambda_{k+1} = \min\left( 1, \frac{n}{n-1} \lambda_k \right)$ for $k\geq1$.
For the sake of being self-contained, we provide a brief proof here using the same arguments as in~\cite{song2021variance}.
Here, $\lambda_k$ for $k\geq 1$, increases geometrically with the factor $\left( 1+ \frac{1}{n-1} \right)$ until it reaches $1$. We have for $k\geq 1$ that $\lambda_{k} = \frac{1}{n-1}\left(1+\frac{1}{n-1}\right)^{k-1}= 1$ when $k-1= \log(n-1)/\left(\log n - \log(n-1)\right)$. Thus for  $k > K_0=\lfloor1+\log(n-1)/\left(\log n - \log(n-1)\right)\rfloor$, we have $\lambda_k = 1$.
By summing the geometric series, we have
\[
\Lambda_{K_0}  =1+\sum_{k=1}^{K_0} \frac{1}{n-1} \left( 1+ \frac{1}{n-1}\right)^{k-1}  
= 1+\frac{1}{n-1}\frac{1-\left(1+\frac{1}{n-1} \right)^{K_0}}{-\frac{1}{n-1}} = \left(1+\frac{1}{n-1} \right)^{K_0} \geq \frac{(n-1)^2}{n}
\]
since, by the definition of $K_0$, $\frac{1}{n-1}\left(1+\frac{1}{n-1} \right)^{K_0+1}\geq 1$. 
Since $\lambda_k = 1$ for $k>K_0$, we have for any $K>K_0$ that
\[
\Lambda_K =\Lambda_{K_0} + \sum_{k=K_0+1}^{K-1} \lambda_k \geq \frac{(n-1)^2}{n} + K-K_0-1.
\]
One part of the complexity result follows by substituting this lower bound into \eqref{eq:sh7}.
For the other part of the result (the case of $K \le K_0$), we note from the argument above that  $\Lambda_{K} = \left(1+\frac{1}{n-1}\right)^{K}$ for $K \le K_0$, completing the proof of complexity. 
For the upper and lower bounds of $K_0$, we use the definition of $K_0$ and the inequality $\frac{1}{n}\leq \log\left( 1+\frac{1}{n-1} \right)\leq \frac{1}{n-1}$.
\end{proof}

\subsection{Convergence Analysis with Sparse $A$ (\Cref{alg: purecd_sparse})}\label{sec: cvx_ccv_sparse}
We proceed to analyze a single iteration of~\Cref{alg: purecd_sparse}.
Similar to~\Cref{lem: lem1_purecd}, our results use the arguments from~\cite{alacaoglu2020random} but are simplified because our setting does not have an additional smooth nonlinear term in the objective and because we assume uniform sampling for~\Cref{alg: purecd_sparse}.
Our analysis accounts for possible strong convexity in $g$ and  $h^\ast$, so that we can leverage the results in later sections.

With sparsity, for keeping the computational cost cheap, we  assume separability of $g$ in addition to $h$.
The next remark explains the notion of complexity that we use for sparse $A$.
\begin{remark}\label{eq: sparse_cost_rem}
In the results with sparse data, we use the notion of \emph{expected complexity}.
This is due to the per-iteration cost being random since it depends on the selected row at each iteration.
We compute the complexity by using the \emph{expected cost} of each iteration.
Since each row is selected with the same probability $1/n$, the expected cost per iteration is $\sum_{i=1}^n \frac{1}{n} \nnz(A_i) = \frac{1}{n} \nnz(A)$.
\end{remark}

\begin{assumption}\label{asmp: asmp2}
The function $g$ in \eqref{eq:prob} is separable, that is, $g(x) = \sum_{j=1}^d g_j(x^{(j)})$.
\end{assumption}

The following lemma analyzes a single  iteration of \Cref{alg: purecd_sparse}.
\begin{lemma}\label{lem: sparse_one_it}
Let~\Cref{asmp: asmp1} and \Cref{asmp: asmp2} hold, and suppose that $g$ is $\mu_g \geq 0$ strongly convex and $h^\ast$ is $\mu_h\geq 0$ strongly convex. 
In \Cref{alg: purecd_sparse}, we set
\begin{equation*}
    \Theta_k = n\Pi (I+\mu_g\Tau_k)^{-1} \iff \theta_k^{(j)} = \frac{n\pi^{(j)}}{1+\mu_g\tau_k^{(j)}}.
\end{equation*}
Then we have for any deterministic $x, y$ that
\begin{align*}
& 2G(\bar x_{k+1}, \bar y_{k+1}, x, y)+ \mathbb{E}_k \| x-x_{k+1} \|^2_{(\Tau_k^{-1}+\mu_g)\Pi^{-1}} + \mathbb{E}_k\| y-  y_{k+1}\|^2_{(\Sigma_k^{-1} + \mu_h)n} \\
& \leq \| x-x_k\|^2_{(\Tau_k^{-1}+\mu_g)\Pi^{-1} - \mu_g}  
+\| y-y_k\|^2_{(\Sigma_k^{-1}+\mu_h)n-\mu_h} - \| \bar y_{k+1} - y_k\|^2_{\Sigma_k^{-1}} + \|\bar y_{k+1} - y_k\|^2_{M_y((\Tau_k^{-1}+\mu_g)\Pi^{-1})},
\end{align*}
where $M_y^{(i)}((\Tau_k^{-1}+\mu_g)\Pi^{-1}) = \sum_{j=1}^d \frac{1+\mu_g\tau_k^{(j)}}{n\pi^{(j)}} \tau_k^{(j)} (\theta_k^{(j)})^2 A_{i, j}^2$ and $\pi^{(j)} = \frac{|i\in I(j)|}{n}$.
\end{lemma}
\begin{proof}
By the prox-inequality \eqref{eq: prox_ineq} applied on the definition of $\bar x_{k+1}$, it follows that for any $x$, we have 
\begin{align*}
\langle \bar x_{k+1} - x_k + \Tau_k A^\top y_k, \Tau_k^{-1}(x - \bar x_{k+1}) \rangle + g(x) - g(\bar x_{k+1}) \geq \frac{\mu_g}{2} \| x-\bar x_{k+1} \|^2.
\end{align*}
Similar to the derivation of \eqref{eq: purecd_primal_ineq1}, we obtain
\begin{align}\label{eq: maxexp_eq0}
\| x- \bar x_{k+1} \|^2_{\Tau_k^{-1}+\mu_g} + 2(g(\bar x_{k+1}) - g(x)) \leq \| x-x_k\|^2_{\Tau_k^{-1}} - \| \bar x_{k+1} - x_k \|^2_{\Tau_k^{-1}} + 2 \langle A^\top y_k, x-\bar x_{k+1} \rangle.
\end{align}
By the update rule of $x_{k+1}$ from step~\ref{st: rg4}, we have for any $\Beta_k = \diag(\beta_k^{(1)}, \dots, \beta_k^{(d)}) \succ 0$, and for any deterministic $x$, (see~\cite[Lemma 2]{alacaoglu2020random} and \Cref{lem: tech_lemma_one_iter}) that 
\begin{multline}\label{eq: sparse_lemma_output}
    \mathbb{E}_k \| x-x_{k+1} \|^2_{\Beta_k} = \| x-\bar x_{k+1}  \|^2_{\Beta_k\Pi} - \|x- x_k  \|^2_{\Beta_k \Pi} + \| x-x_k \|^2_{\Beta_k}\\
    -2\langle \bar x_{k+1} - x, n^{-1} \Beta_k \Tau_k \Theta_k A^\top (\bar y_{k+1} - y_k) \rangle
    +\| \bar y_{k+1} - y_k\|^2_{M_y(\Beta_k)},
\end{multline}
where $M_y^{(i)}(\Beta_k) = \sum_{j=1}^d \frac{1}{n} \beta_k^{(j)} (\tau_k^{(j)})^2 (\theta_k^{(j)})^2 A_{i, j}^2$ and $\pi^{(j)} = \frac{|i\in I(j)|}{n}$.
By setting $\beta_k^{(j)} = (1+\mu_g\tau_k^{(j)})(\tau_k^{(j)} \pi^{(j)})^{-1}$ (that is, $\Beta_k = (\Tau_k^{-1}+\mu_g)\Pi^{-1} = (I+\mu_g\Tau_k)\Tau_k^{-1}\Pi^{-1}$) in \eqref{eq: sparse_lemma_output}, we obtain
\begin{multline*}
    \mathbb{E}_k \| x- x_{k+1} \|^2_{(\Tau_k^{-1}+\mu_g)\Pi^{-1}} = \| x-\bar x_{k+1}\|^2_{\Tau_k^{-1}+\mu_g} - \| x- x_k  \|^2_{\Tau_k^{-1}+\mu_g} + \| x - x_k\|^2_{(\Tau_k^{-1}+\mu_g)\Pi^{-1}} \\
    -2\langle \bar x_{k+1} - x, n^{-1} (I+\mu_g \Tau_k)\Pi^{-1} \Theta_k A^\top (\bar y_{k+1}-y_k) \rangle + \| \bar y_{k+1} - y_k\|^2_{M_y((\Tau_k^{-1}+\mu_g)\Pi^{-1})}.
\end{multline*}
We add this equality to~\eqref{eq: maxexp_eq0} and drop $-\|\bar x_{k+1} -x\|^2_{\Tau_k^{-1}}$ on the RHS to deduce that
\begin{multline}\label{eq: maxexp_eq1}
\mathbb{E}_k \| x-x_{k+1} \|^2_{(\Tau_k^{-1}+\mu_g)\Pi^{-1}} + 2(g(\bar x_{k+1})-g( x)) \leq \| x-x_k\|^2_{(\Tau_k^{-1}+\mu_g)\Pi^{-1}-\mu_g}  + 2 \langle A^\top y_k, x-\bar x_{k+1} \rangle \\
-2\langle \bar x_{k+1} - x, n^{-1} (I+\mu_g\Tau_k)\Pi^{-1} \Theta_k A^\top (\bar y_{k+1}-y_k) \rangle + \|\bar y_{k+1} - y_k\|^2_{M_y((\Tau_k^{-1}+\mu_g)\Pi^{-1})}.
\end{multline}
For the dual variable we have similarly for any $y$ (see~\cref{eq: purecd_dual_ineq1,eq: prox_ineq} and the definition of $\bar y_{k+1}$),
\begin{align}\label{eq: maxexp_eq2}
\| y-\bar y_{k+1} \|^2_{\Sigma_k^{-1}+\mu_h} + 2(h^\ast(\bar y_{k+1}) - h^\ast(y)) \leq \| y-y_k\|^2_{\Sigma_k^{-1}} - \| \bar y_{k+1} - y_k \|^2_{\Sigma_k^{-1}} - 2\langle A\bar x_{k+1}, y-\bar y_{k+1} \rangle.
\end{align}
Since $\mathbb{E}_k \| y-y_{k+1} \|^2_{\Sigma_k^{-1}+\mu_h} = n^{-1} \| y-\bar y_{k+1} \|^2_{\Sigma_k^{-1}+\mu_h} + (1-n^{-1}) \| y-y_k\|^2_{\Sigma_k^{-1}+\mu_h}$ (see~\Cref{lem: tech_lemma_one_iter}), we can substitute for $\| y-\bar y_{k+1} \|^2_{\Sigma_k^{-1}+\mu_h}$ in~\eqref{eq: maxexp_eq2} to obtain
\begin{equation}\label{eq: maxexp_eq2.5}
\begin{aligned}
& \| y-y_{k+1} \|^2_{(\Sigma_k^{-1}+\mu_h)n} + 2(h^\ast(\bar y_{k+1}) - h^\ast(y)) \\
& \leq \| y-y_k\|^2_{(\Sigma_k^{-1}+\mu_h)n - \mu_h} - \| \bar y_{k+1} - y_k \|^2_{\Sigma_k^{-1}} 
- 2\langle A\bar x_{k+1}, y-\bar y_{k+1} \rangle.
\end{aligned}
\end{equation}
Recall~\eqref{eq: g_deff} and
\begin{align}
&G(\bar x_{k+1}, \bar y_{k+1}, x, y) = g(\bar x_{k+1}) + \langle A\bar x_{k+1}, y \rangle - h^\ast(y) - g(x) - \langle Ax, \bar y_{k+1} \rangle + h^\ast(\bar y_{k+1}) \notag\\
\iff &G(\bar x_{k+1}, \bar y_{k+1}, x, y) = g(\bar x_{k+1}) + \langle A(\bar x_{k+1}-x), y \rangle - h^\ast(y) - g(x) - \langle Ax, \bar y_{k+1} - y \rangle + h^\ast(\bar y_{k+1}) \notag\\
\iff  & 2G(\bar x_{k+1}, \bar y_{k+1}, x, y) -2\left[ g(\bar x_{k+1}) - g(x) +h^\ast(\bar y_{k+1})- h^\ast(y) \right] = 2\langle A(\bar x_{k+1}-x), y \rangle - 2\langle Ax, \bar y_{k+1} - y \rangle.\label{eq: ig3}
\end{align}
Combine the inequalities for primal and dual (that is, \cref{eq: maxexp_eq1,eq: maxexp_eq2.5}) and add to them the identity in~\eqref{eq: ig3} to deduce
\begin{equation}\label{eq: main_ineq1_gaplike}
\begin{aligned}
& 2G(\bar x_{k+1}, \bar y_{k+1}, x, y)+ \mathbb{E}_k \| x-x_{k+1} \|^2_{(\Tau_k^{-1}+\mu_g)\Pi^{-1}} + \| y-  y_{k+1}\|^2_{(\Sigma_k^{-1}+\mu_h)n} \\
& \leq \| x\!-\!x_k\|^2_{(\Tau^{-1}+\mu_g)\Pi^{-1}-\mu_g} 
+\| y\!-\!y_k\|^2_{(\Sigma_k^{-1}+\mu_h)n-\mu_h}
- \| \bar y_{k+1} \!-\! y_k\|^2_{\Sigma_k^{-1}} + \|\bar y_{k+1} \!-\! y_k\|^2_{M_y((\Tau_k^{-1}+\mu_g)\Pi^{-1})} \\
& \quad - 2 \langle \bar x_{k+1}-x, n^{-1}(I+\mu_g\Tau_k)\Pi^{-1}\Theta_k A^\top(\bar y_{k+1} - y_k) \rangle \\
& \quad +2 \langle A^\top(y_k - y), x-\bar x_{k+1} \rangle + 2\langle \bar x_{k+1} - x, A^\top(\bar y_{k+1} - y) \rangle.
\end{aligned}
\end{equation}
First, note that the terms in the last line can be combined to write
\begin{align*}
    2 \langle A^\top(y_k - y), x-\bar x_{k+1} \rangle + 2\langle \bar x_{k+1} - x, A^\top(\bar y_{k+1} - y) \rangle = 2\langle x-\bar x_{k+1}, y_k-\bar y_{k+1} \rangle.
\end{align*}
Moreover, we see that by combining this with the inner product in the third line of~\eqref{eq: main_ineq1_gaplike}, we get
\begin{align*}
  &  - 2 \langle \bar x_{k+1}-x, n^{-1}(I+\mu_g\Tau_k)\Pi^{-1}\Theta_k A^\top(\bar y_{k+1} - y_k) \rangle 
+2 \langle A^\top(y_k - y), x-\bar x_{k+1} \rangle + 2\langle \bar x_{k+1} - x, A^\top(\bar y_{k+1} - y) \rangle \\
& = 2\langle x-\bar x_{k+1}, [I - n^{-1}(I+\mu_g\Tau_k)\Pi^{-1}\Theta_kA^\top (y_k - \bar y_{k+1})] \rangle = 0,
\end{align*}
which is equivalent to setting $\Theta_k =n\Pi(I+\mu_g\Tau_k)^{-1}$, as we do in the theorem.
We complete the proof by using this estimate on the last two lines of~\eqref{eq: main_ineq1_gaplike}.
\end{proof}

\subsubsection{General Convex-Concave Problems with Relaxed Gap-like Measure}\label{subsec: random_iter}
As illustrated in Table~\ref{tb: 1} and also highlighted in~\Cref{sec: conclusions}, the ideal result would be an improved complexity for $\mathbb{E}\max_{x, y} G(\xout, \yout, x, y)$ compared to deterministic methods, as we derived for ~\Cref{alg: purecd_dense} with dense $A$ in Section~\ref{sec: cvx_ccv_dense}.
However, the sparse-friendly~\Cref{alg: purecd_sparse} brings additional difficulties that prevent such a result being proved with our current techniques.
We include the next result to show the difference on the difficulty of deriving a guarantee on $\mathbb{E}\max_{(x, y)\in\mathcal{Z}}  G(\xout, \yout, x, y)$ compared to $\max_{(x, y)\in\mathcal{Z}} \mathbb{E} G(\xout, \yout, x, y)$ and also to be able to compare our results with some of the existing literature (see \Cref{tb: 2}).
On the other hand, as we already mentioned in Ex.~\ref{ex: np3}, $\max_{(x, y)\in\mathcal{Z}} \mathbb{E} G(\xout, \yout, x, y)$ has serious deficiencies as an optimality measure and it remains an important open question to derive guarantees for expected gap.

Our result shows that when the expectation is inside the $\max$, we can exploit its properties to derive a bound with a short and simple proof by simply {\em outputting a random iterate}. 
We cannot do such a trick for expected duality gap due to the existence of the nonlinear $\max$ operator between the expectation and $G$.
As shown in~\cite{alacaoglu2020random}, even the guarantee $\mathbb{E}\max_{(x, y)\in\mathcal{Z}}  G(x^K, y^K, x, y) \leq O\left(\frac{n\|A\|}{K}\right)$, which can only recover the complexity of the deterministic method in the sparse case, requires a much more intricate analysis.

\begin{theorem}\label{th: thm_max_e}
Let~\Cref{asmp: asmp1} and \Cref{asmp: asmp2} hold. 
Use~\Cref{alg: purecd_sparse} with the parameters 
$\tau_k^{(j)} \equiv \tau^{(j)} = \frac{1}{\pi^{(j)} n\max_i\|A_i\|} $, $\sigma_k^{(i)} \equiv \sigma^{(i)}=\frac{1}{\|A_i\|}$ and 
$ \theta^{(j)} =n\pi^{(j)}$ where $\pi^{(j)} = \frac{| i\in I(j)|}{n}$. Let us pick $\hat k$ uniformly at random from $\{1,\dots,K\}$, independent of all the other randomness in the algorithm.
Then, we have $\max_{(x, y)\in\mathcal{Z}} \mathbb{E} G(\bar x_{\hat k}, \bar y_{\hat k}, x, y) \leq \frac{n\max_i\|A_i\|}{K} D_{\mathcal{Z}}$ and hence,
\begin{equation*}
\max_{(x, y)\in\mathcal{Z}} \mathbb{E} G(\bar x_{\hat k}, \bar y_{\hat k}, x, y) \leq \varepsilon,
\end{equation*}
with complexity
\begin{equation*}
O\left( \nnz(A)+ \nnz(A) \max_i\|A_i\| D_{\mathcal{Z}} \varepsilon^{-1} \right).
\end{equation*}
\end{theorem}
\begin{remark}
In practice, to run this algorithm, one would select $\hat k$ and then compute the vectors $\bar x_{k+1}, \bar y_{k+1}$ at $k=\hat k-1$, only once, to not incur additional computational cost.
Such a trick is used before in~\cite[Remark 5]{fercoq2019coordinate},~\cite[Section~6.1.1.1]{lan2020first},~\cite[SDCA with random option]{shalev2013stochastic}, and in general for nonconvex optimization.
In practice, one can obtain this guarantee by only running the algorithm until iteration $\hat k-1$.
\end{remark}


\begin{proof}[Proof of Thm.~\ref{th: thm_max_e}]
We use \Cref{lem: sparse_one_it} with $\mu_g=\mu_h=0$, $\Tau_k = \Tau$, $\Sigma_k = \Sigma$ and $\Theta_k = \theta I = n\Pi$  to obtain
\begin{equation}\label{eq: maxexp_eq3}
\begin{aligned}
& 2G(\bar x_{k+1}, \bar y_{k+1}, x, y) + \mathbb{E}_k \| x-x_{k+1} \|^2_{\Tau^{-1}\Pi^{-1}} + \mathbb{E}_k\| y- y_{k+1}\|^2_{n\Sigma^{-1}} \\
& \leq \| x-x_k\|^2_{\Tau^{-1}\Pi^{-1}}+\| y-y_k\|^2_{n\Sigma^{-1}} 
- \| \bar y_{k+1} - y_k\|^2_{\Sigma^{-1}} + \|\bar y_{k+1} - y_k\|^2_{M_y(\Tau^{-1}\Pi^{-1})}.
\end{aligned}
\end{equation}
To cancel the last two terms in the RHS, by the definition of $M_y$ in~\Cref{lem: sparse_one_it}, we need
\begin{equation}
    \frac{1}{\sigma^{(i)}} - \sum_{j=1}^m \frac{1}{n\pi^{(j)}}\tau^{(j)}(\theta^{(j)})^2 A_{j, i}^2 \geq 0 \iff
    \frac{1}{\sigma^{(i)}} - \sum_{j=1}^m n\pi^{(j)}\tau^{(j)} A_{j, i}^2 \geq 0,\label{eq: tg4}
\end{equation}
where the equivalence is due to the definition of $\theta^{(j)}$.
Recall $\sigma^{(i)} = \frac{1}{\|A_i\|}\geq\frac{1}{\max_i\|A_i\|}$ and $\tau^{(j)} = \frac{1}{ \pi^{(j)} n\max_i \|A_i\|}$ and hence~\eqref{eq: tg4} is satisfied.
As a result, we get after taking total expectation and summation in~\eqref{eq: maxexp_eq3} that
\[
\mathbb{E} \frac{1}{K}\sum_{k=0}^{K-1} G(\bar x_{k+1}, \bar y_{k+1}, x, y) \leq \|x-x_0\|^2_{\Tau^{-1}\Pi^{-1}} + \| y-y_0\|^2_{n\Sigma^{-1}} =\frac{n\max_i\|A_i\|}{K}\left(\| x-x_0\|^2+\| y-y_0\|^2\right).
\]
Since the choice of $\hat k$ is independent of other randomness in the algorithm, we have
\begin{equation*}
\mathbb{E} G(\bar x_{\hat k}, \bar y_{\hat k}, x, y)] = \mathbb{E} [\mathbb{E}_{\hat k} G(\bar x_{\hat k}, \bar y_{\hat k}, x, y)] = \mathbb{E} \frac{1}{K} \sum_{k=1} ^{K} G(\bar x_{ k}, \bar y_{ k}, x, y) \leq \frac{n\max_i\|A_i\|}{K}\left(\| x-x_0\|^2+\| y-y_0\|^2\right).
\end{equation*}
We take maximum and use the definition of $D_\mathcal{Z}$ from \Cref{sec: notation} to derive the first inequality in the proof.
For the complexity, we use the fact that one needs ${n\max_i\|A_i\|D_{\mathcal{Z}}}{\varepsilon^{-1}}$ iterations to obtain $\max_{x, y} \mathbb{E}G(\bar x_{\hat k}, \bar y_{\hat k}, x, y) \leq \varepsilon$ and we multiply by the expected cost per iteration, which is  $\frac{1}{n}\nnz(A)$ (see~\Cref{eq: sparse_cost_rem}).
The additional term $\nnz(A)$ is for the initial computation $A^\top y_0$.
\end{proof}

\section{Convergence in the Strongly Convex-Strongly Concave Case~(\Cref{alg: purecd_sparse})}\label{sec: str_cvx_str_ccv}
We focus now on the convergence behavior of \Cref{alg: purecd_sparse} when $g$ and $h^\ast$ in~\eqref{eq:prob} are both strongly convex and $A$ is sparse.
In the context of ERM~\eqref{eq: case2_prob}, this case corresponds to a smooth loss function and a strongly convex regularizer, which has been the main focus of many papers in the field~\cite{allen2017katyusha,zhang2015stochastic,tan2020accelerated}.
In~\cite{zhang2015stochastic,tan2020accelerated}, lazy update strategies are described, to adapt the iteration cost to sparsity of the data matrix.
Our complexity will be matching the ones given in these works.

The previous approaches essentially analyze the algorithms in the dense setting and use lazy update techniques to exploit the sparsity.
Our approach is to analyze the sparse-friendly algorithm directly, which is more challenging due to additional randomness in the primal updates.
The advantage of this approach is that it requires no special implementation features like lazy updates; \Cref{alg: purecd_sparse} adapts to sparsity in $A$ naturally.

The following theorem establishes the bound claimed in the last column of Table~\ref{tb: 3} for PURE-CD.
\begin{theorem}\label{th: strcvx_strccv}
Let~\Cref{asmp: asmp1} and \Cref{asmp: asmp2} hold and suppose that $g$ is $\mu_g > 0$ strongly convex and $h_i^\ast$ is $\mu_h > 0$ strongly convex. Recall $\pi^{(j)} = \frac{| i\in I(j)|}{n}$. In Alg~\ref{alg: purecd_sparse}, set
\begin{equation*}
\theta^{(j)} = \frac{\pi^{(j)} n}{1+\mu_g\tau^{(j)}},~~~~ \tau^{(j)} = \frac{\sqrt{\mu_h}}{\sqrt{\mu_g}\max_i \|A_i\| \pi^{(j)} n},~~~~ \sigma^{(i)} = \frac{\sqrt{\mu_g}}{\sqrt{\mu_h} \|A_i\| }.
\end{equation*}
Then we have that $\mathbb{E}\left[ \| x_{K} - x_\star \|^2 + \|y_{K} - y_\star \|^2 \right] \leq \varepsilon$ with expected complexity
\begin{equation*}
\tilde{O}\left( \left( \nnz(A) + \nnz(A)\frac{ \max_i\| A_i\|}{\sqrt{\mu_h\mu_g}} \right)\log \varepsilon^{-1}\right).
\end{equation*}
\end{theorem}
\begin{proof}
Let $z_\star=(x_\star,y_\star)$ be the unique solution of \eqref{eq:prob2}.
We use the result of Lemma~\ref{lem: sparse_one_it} with $\mu_g>0$, $\mu_h > 0$, $\Tau_k = \Tau$, $\Sigma_k = \Sigma$, and $\Theta_k = \Theta= (I+\mu_g\Tau)^{-1}n\Pi$, together with $G(\bar{x}_{k+1},\bar{y}_{k+1},x_\star,y_\star) \ge 0$, to write
\begin{align}
    \label{eq: lin_eq1}
    & \mathbb{E}_k \| x_\star-x_{k+1} \|^2_{(\Tau^{-1}+\mu_g)\Pi^{-1}} + \mathbb{E}_k\| y_\star- y_{k+1}\|^2_{(\Sigma^{-1} + \mu_h)n}  \\
    \nonumber
    & \le \| x_\star-x_k\|^2_{(\Tau^{-1}+\mu_g)\Pi^{-1} - \mu_g}+\| y_\star-y_k\|^2_{(\Sigma^{-1}+\mu_h)n-\mu_h}
 - \| \bar y_{k+1} - y_k\|^2_{\Sigma^{-1}} + \|\bar y_{k+1} - y_k\|^2_{M_y((\Tau^{-1} + \mu_g)\Pi^{-1})}.
\end{align}
By the definition of $M_y$ from Lemma~\ref{lem: sparse_one_it}, and the definition of $\theta^{(j)}$, the aggregate of the last two terms on the RHS of \eqref{eq: lin_eq1} is nonpositive provided that
\begin{equation}\label{eq: som_ss_cond}
    \frac{1}{\sigma^{(i)}} - \sum_{j=1}^d \frac{1+\mu_g\tau^{(j)}}{n\pi^{(j)}}\tau^{(j)}(\theta^{(j)})^2 A_{i, j}^2 \geq 0 \iff
    \frac{1}{\sigma^{(i)}} - \sum_{j=1}^d \frac{n\pi^{(j)}}{1+\mu_g\tau^{(j)}}\tau^{(j)} A_{i, j}^2 \geq 0,
\end{equation}
for which a sufficient condition is
\begin{equation}\label{eq: som_ss_cond2}
    \frac{1}{\sigma^{(i)}} - \sum_{j=1}^d n\pi^{(j)} \tau^{(j)} A_{i,j}^2 \geq 0.
\end{equation}
From our definitions of $\tau^{(j)}$ and $\sigma^{(i)}$, it can be verified that \eqref{eq: som_ss_cond2} holds. 
We can thus simplify \eqref{eq: lin_eq1} to
\begin{equation}
    \label{eq: lin_eq1a}
    \begin{aligned}
    & \mathbb{E}_k \| x_\star-x_{k+1} \|^2_{(\Tau^{-1}+\mu_g)\Pi^{-1}} + \mathbb{E}_k\| y_\star- y_{k+1}\|^2_{(\Sigma^{-1} + \mu_h)n} \\
    & \le \| x_\star-x_k\|^2_{(\Tau^{-1}+\mu_g)\Pi^{-1} - \mu_g}+\| y_\star-y_k\|^2_{(\Sigma^{-1}+\mu_h)n-\mu_h}.
    \end{aligned}
\end{equation}
To get a contraction in~\eqref{eq: lin_eq1a}, we seek a constant $c>1$ such that
\begin{equation}\label{eq: scsc_contr_factors}
\begin{aligned}
\| x_\star - x_{k+1} \|^2_{(\Tau^{-1} + \mu_g)\Pi^{-1}} &\geq c \| x_\star - x_{k+1} \|^2_{(\Tau^{-1} + \mu_g)\Pi^{-1}-\mu_g},\\
\| y_\star - y_{k+1} \|^2_{(\Sigma^{-1}+\mu_h)n} &\geq c \| y_\star - y_{k+1} \|^2_{(\Sigma^{-1}+\mu_h)n-\mu_h},
\end{aligned}
\end{equation}
which, because of the diagonality of the matrices defining the weighted norms, are implied by
\begin{subequations} 
\label{eq:c12}
\begin{align}
\label{eq:c1}
    ((\Tau^{-1}+\mu_g)\Pi^{-1}) ((\Tau^{-1}+\mu_g)\Pi^{-1}-\mu_g)^{-1} & \geq c I, \\
\label{eq:c2}
    ((\Sigma^{-1}+\mu_h)n) ((\Sigma^{-1}+\mu_h)n-\mu_h)^{-1}& \geq c I.
\end{align}
\end{subequations}
Defining $\kappa^{(i)} := \frac{\|A_i\|}{\sqrt{\mu_g\mu_h}}$, we show that the bounds \eqref{eq:c12} hold for 
\begin{equation}
    \label{eq:cdef}
    c := 1+ \frac{1}{n - 1 + n \max_i \kappa^{(i)}}.
\end{equation}
For \eqref{eq:c1}, we have from $\mu_g\tau^{(j)}\pi^{(j)} = \frac{\sqrt{\mu_g \mu_h}}{n\max_i\|A_i\|} = \frac{1}{n \max_i \kappa^{(i)}}$ and $\pi^{(j)} \geq \frac{1}{n}$ that
\begin{multline*}
\frac{(\tau^{(j)})^{-1}(\pi^{(j)})^{-1}+\mu_g(\pi^{(j)})^{-1}}{(\tau^{(j)})^{-1}(\pi^{(j)})^{-1}+\mu_g(\pi^{(j)})^{-1} - \mu_g}
= 
1 + \frac{\mu_g}{(\tau^{(j)})^{-1}(\pi^{(j)})^{-1}+\mu_g(\pi^{(j)})^{-1} - \mu_g} \\
= 1+ \frac{1}{\frac{1}{\pi^{(j)}} - 1 + \frac{1}{\mu_g\tau^{(j)}\pi^{(j)}}} 
= 1+ \frac{1}{\frac{1}{\pi^{(j)}} - 1 + n \max_i \kappa^{(i)}} 
\geq 1+ \frac{1}{n - 1 + n \max_i \kappa^{(i)}} = c.
\end{multline*}
For \eqref{eq:c2}, we have from $\sigma^{(i)}\mu_h = \frac{\sqrt{\mu_g\mu_h}}{\|A_i\|} = \frac{1}{\kappa^{(i)}}$ that
\[
\frac{(\sigma^{(i)}) ^{-1}n+ \mu_h n}{(\sigma^{(i)})^{-1}n+n\mu_h - \mu_h} 
= 1 + \frac{\mu_h}{(\sigma^{(i)})^{-1}n+n\mu_h-\mu_h}
= 1 + \frac{1}{n(\sigma^{(i)})^{-1} \mu_h^{-1} + n-1} = 1+\frac{1}{n-1+n\kappa^{(i)} }\geq c.
\]
By substituting \eqref{eq: scsc_contr_factors} and \eqref{eq:cdef} into~\eqref{eq: lin_eq1a} after taking total expectation we obtain
\begin{multline*}
\left( 1+ \frac{1}{n-1+n\max_i \kappa^{(i)}} \right)\mathbb{E}\left[ \| x_\star-x_{k+1} \|^2_{(\Tau^{-1}+\mu_g)\Pi^{-1} - \mu_g} + \| y_\star-y_{k+1} \|^2_{(\Sigma^{-1} + \mu_h)n-\mu_h} \right] \\
\leq \mathbb{E}\left[ \| x_\star-x_{k} \|^2_{(\Tau^{-1}+\mu_g)\Pi^{-1} - \mu_g} +\| y_\star-y_{k} \|^2_{(\Sigma^{-1} + \mu_h)n-\mu_h}\right].
\end{multline*}
Defining $D_{\star}^{sc} := \| x_\star-x_{0} \|^2_{(\Tau^{-1}+\mu_g)\Pi^{-1} - \mu_g} +\| y_\star-y_{0} \|^2_{(\Sigma^{-1} + \mu_h)n-\mu_h}$, we can  iterate the inequality above to obtain
\begin{equation}
\mathbb{E}\left[ \| x_\star-x_{K} \|^2_{(\Tau^{-1}+\mu_g)\Pi^{-1} - \mu_g} + \| y_\star-y_{K} \|^2_{(\Sigma^{-1} + \mu_h)n-\mu_h} \right]
\leq \left( 1+ \frac{1}{n-1+n\max_i \kappa^{(i)}} \right)^{-K} D_{\star}^{sc}.
\end{equation}
By using the bound $t\geq \log(1+t)\geq t/2$ for $t\in[0,1]$, we can deduce that the number of iterations needed to make the LHS smaller than $\varepsilon$ is $O\left((n+n\max_i\kappa^{(i)})\log\left(\frac{D_\star^{sc}}{\varepsilon}\right)\right)$.
The complexity result follows by multiplying the number of iterations and the expected per iteration cost, which is $\frac{1}{n}\nnz(A)$ (see~\Cref{eq: sparse_cost_rem}).
The complicated metrics in the norms only affect the constant in the logarithmic terms in this bound, and are suppressed in the final result in the notation $\tilde O$.
\end{proof}

\section{Convergence with One-Sided Strong Convexity (\Cref{alg: purecd_sparse})}\label{sec: str_cvx_ccv}
We focus now on two cases of~\eqref{eq:prob} in which either the primal or dual is strongly convex, but not both.
We analyze~\Cref{alg: purecd_sparse} and derive guarantees that depend on the sparsity of $A$.

Let us recall the problem  \eqref{eq:prob}:
\begin{equation*}
    \min_x \max_y \sum_{i=1}^n \langle A_ix, y^{(i)} \rangle - h_i^\ast(y^{(i)}) + g(x).
\end{equation*}
A well known instance of the case in which $g(x)$ is strongly convex is SVM. 
When $h^\ast(x)$ is strongly convex, we have that $h$ is smooth, and the class of problems with this property includes  Lasso ($\ell_1$-regularized linear least squares).
In these cases, it is well known that we can obtain iteration complexity bounds of $O\left( \frac{1}{\sqrt{\varepsilon}} \right)$ by using accelerated methods in the primal~\cite{allen2017katyusha}, stochastic dual coordinate ascent~\cite{shalev2013stochastic,shalev2014accelerated}, or primal-dual methods~\cite{chambolle2011first,chambolle2016ergodic,chambolle2018stochastic}.
We now derive similar complexities for PURE-CD, with an explicit dependence on complexity depending on $\nnz(A)$.
\subsection{Analysis for the Strongly Convex-Concave Case}

We deal first with the case in which $g$ is strongly convex while $h^\ast$ is only convex. 
The following theorem establishes the complexity bound for PURE-CD from the first column of \Cref{tb: 3}.
\begin{theorem}\label{th: scc}
Suppose that~\Cref{asmp: asmp1},~\ref{asmp: asmp2} hold and that $g$ is $\mu_g > 0$ strongly convex. Recall $\pi^{(j)} = \frac{| i\in I(j)|}{n}$. In \Cref{alg: purecd_sparse}, we set
\begin{equation} \label{eq:is9}
\begin{aligned}
\tilde \tau_0 &= \min\left\{ \frac{1}{n}, \frac{\mu_g}{n\max_i\|A_i\|}\right\},~~~~ \sigma_0 = \frac{1}{\max_i \|A_i\|}, ~~~~\Sigma_k=\sigma_k I\\
 \tilde \tau_{k+1} &= \frac{\tilde \tau_k}{\sqrt{1+\tilde \tau_k}},~~~~  \tau_{k}^{(j)} = \frac{\tilde \tau_k}{\mu_g\pi^{(j)} - \mu_g(1-\pi^{(j)})\tilde \tau_k},~~~~ \sigma_{k+1} = \sigma_k \sqrt{1+\tilde\tau_k},~~~~
 \theta_{k}^{(j)} = \frac{\pi^{(j)} n}{1+\mu_g\tau_{k}^{(j)}}.
\end{aligned}
\end{equation}
Then we  have $\mathbb{E}\left[ \| x_{K} - x_\star \|^2 \right] \leq \varepsilon$ for $K$ satisfying
\begin{equation} \label{eq:K5}
K \ge \frac{\sqrt{D_\star}}{\sqrt{\varepsilon}} 3n \max \left( 1, \frac{\max_i \|A_i\|}{\mu_g} \right),
\end{equation}
with expected complexity
\begin{equation*}
{O}\left( \nnz(A) \left( 1 + \frac{\sqrt{D_\star}}{\sqrt{\varepsilon}}  \max \left(1, \frac{\max_i \| A_i\|}{\mu_g} \right) \right) \right),
\end{equation*}
where $D_\star$ is defined in \eqref{eq:Dstar}.
\end{theorem}
\begin{proof} 
First, we note from \eqref{eq:is9} that $\tilde\tau_k \in (0,1/n]$ for all $k$ and that  $\pi^{(j)}  \in [1/n,1]$ for all $j$. Thus, the denominator in the definition of $\tau_k^{(j)}$ satisfies 
\[
\mu_g\pi^{(j)} - \mu_g(1-\pi^{(j)})\tilde \tau_k \ge \mu_g\left[\frac{1}{n}- \left(1-\frac{1}{n}\right)\frac{1}{n} \right] > \mu_g \left[ \frac{1}{n} - \frac{1}{n} \right] = 0,
\]
so we have $\tau_k^{(j)} >0$ for all $j,k$.
By rearranging the definition of $\tau_k^{(j)}$, we obtain
\begin{align}\label{eq: tildetau_from_tau}
    \tilde \tau_k = \frac{\mu_g \pi^{(j)} \tau_{k}^{(j)}}{1+(1-\pi^{(j)})\mu_g\tau_{k}^{(j)}}.
\end{align}
We use Lemma~\ref{lem: sparse_one_it} with $\mu_h=0$, $\mu_g > 0$, $\theta_k^{(j)}$ defined above, with $y_\star$ being any dual solution, and the nonnegativity of $G(\bar{x}_{k+1},\bar{y}_{k+1},x_\star,y_\star)$ to write
\begin{equation}\label{eq: acc_rate_eq1}
\begin{aligned}
& \mathbb{E}_k \| x_\star-x_{k+1} \|^2_{(\Tau_k^{-1} + \mu_g)\Pi^{-1}} + \frac{n}{\sigma_k}\mathbb{E}_k\| y_\star- y_{k+1}\|^2 \\
& \leq \| x_\star-x_k\|^2_{(\Tau_k^{-1} + \mu_g)\Pi^{-1} - \mu_g}+\frac{n}{\sigma_k}\| y_\star-y_k\|^2 
- \frac{1}{\sigma_k}\| \bar y_{k+1} - y_k\|^2 + \|\bar y_{k+1} - y_k\|^2_{M_y((\Tau_k^{-1} + \mu_g)\Pi^{-1})}.
\end{aligned}
\end{equation}
By using the definition of $M_y$ from Lemma~\ref{lem: sparse_one_it} and the definition of $\theta_k^{(j)}$ from \eqref{eq:is9}, similar to~\eqref{eq: som_ss_cond}, we obtain that the final two terms on the RHS will combine to a nonpositive result if 
\begin{equation}\label{eq: wp3}
    -\frac{1}{\sigma_k} + \sum_{j=1}^d \frac{1+\mu_g\tau_k^{(j)}}{n\pi^{(j)}}\tau_k^{(j)}(\theta_k^{(j)})^2 A_{i, j}^2 \leq 0 \iff
    - \frac{1}{\sigma_k} + \sum_{j=1}^d \frac{n\pi^{(j)}}{1+\mu_g\tau_k^{(j)}}\tau_k^{(j)} A_{i, j}^2 \leq 0
\end{equation}
A sufficient condition for this inequality is 
\begin{equation}\label{eq: scc_ss_cond}
    -\frac{1}{\sigma_k} + \sum_{j=1}^d \frac{n\pi^{(j)}}{1+(1-\pi^{(j)})\mu_g\tau_{k}^{(j)}}\tau_{k}^{(j)} A_{i, j}^2 \leq 0 \iff \frac{\sigma_k \tilde\tau_k n}{\mu_g} \|A_i\|^2 \leq 1,
\end{equation}
where the equivalence follows from \eqref{eq: tildetau_from_tau}.
By the definitions \eqref{eq:is9}, we have
\[
\tilde\tau_{k+1} = \frac{\tilde\tau_k}{\sqrt{1+\tilde\tau_k}} = \frac{\tilde\tau_k}{\sigma_{k+1}/\sigma_k} \Longrightarrow \tilde\tau_{k+1} \sigma_{k+1} = \tilde\tau_k \sigma_k,
\]
so that $\tilde\tau_k \sigma_k = \tilde\tau_0 \sigma_0$ for all $k$.
Since $\tilde \tau_0 \leq \frac{\mu_g}{n\max_i\|A_i\|}$ and $\sigma_0 = \frac{1}{\max_i \|A_i\|}$, we have that~\eqref{eq: scc_ss_cond} and consequently~\eqref{eq: wp3} is satisfied.
Thus we can drop the final two terms in the RHS of~\eqref{eq: acc_rate_eq1} and write
\begin{equation}\label{eq: acc_rate_eq1.5}
\mathbb{E}_k \| x_\star-x_{k+1} \|^2_{(\Tau_k^{-1} + \mu_g)\Pi^{-1}} + \frac{n}{\sigma_k}\mathbb{E}_k\| y_\star- y_{k+1}\|^2 \leq \| x_\star-x_k\|^2_{(\Tau_k^{-1} + \mu_g)\Pi^{-1} - \mu_g}+\frac{n}{\sigma_k}\| y_\star-y_k\|^2.
\end{equation}
By the definition in~\eqref{eq: tildetau_from_tau}, we have
\[
\begin{aligned}
(\Tau_k^{-1}+\mu_g I)\Pi^{-1}-\mu_g I & =(1+\mu_g\Tau_k)\Pi^{-1}\Tau_{k}^{-1}-\mu_g I \\
& =
\diag \left( \frac{1+\mu_g \tau_k^{(j)}}{\pi^{(j)} \tau_k^{(j)}} \right) - \mu_g I =
\diag \left( \frac{1+(1-\pi^{(j)}) \mu_g \tau_k^{(j)}}{\pi^{(j)} \tau_k^{(j)}} \right) = \frac{\mu_g}{\tilde\tau_k} I.
\end{aligned}
\]
Thus, we can write~\eqref{eq: acc_rate_eq1.5} as
\begin{equation}\label{eq: acc_rate_eq1.6}
\mathbb{E}_k \| x_\star-x_{k+1} \|^2_{(\Tau_k^{-1} + \mu_g)\Pi^{-1}} + \frac{n}{\sigma_k}\mathbb{E}_k\| y_\star- y_{k+1}\|^2 \leq\frac{\mu_g}{\tilde \tau_k} \| x_\star-x_k\|^2+\frac{n}{\sigma_k}\| y_\star-y_k\|^2.
\end{equation}
Now define 
\begin{equation}
    \label{eq:defck}
    c_k := \sqrt{1+\tilde\tau_k} = \frac{\sigma_{k+1}}{\sigma_k} = \frac{\tilde\tau_k}{\tilde\tau_{k+1}}.
\end{equation}
In~\eqref{eq: acc_rate_eq1.5}, we want to replace the first term on the LHS by 
\begin{equation} \label{eq:is4}
\| x_\star - x_{k+1 \|^2_{(\Tau_k^{-1} + \mu_g)\Pi^{-1}}} = c_k \frac{\mu_g}{\tilde\tau_{k+1}} \| x_\star - x_{k+1} \|^2.
\end{equation}
This equivalence  follows from (see~\eqref{eq: tildetau_from_tau})
\begin{align*}
(\Tau_k^{-1} + \mu_g)\Pi^{-1} = \diag \left( \frac{1+\mu_g \tau_k^{(j)}}{\tau_k^{(j)} \pi^{(j)}} \right) & =
\diag \left( \mu_g + \frac{1+(1-\pi^{(j)})\mu_g \tau_k^{(j)}}{\pi^{(j)} \tau_k^{(j)}} \right) \\
& =
\mu_g \left( 1+\frac{1}{\tilde\tau_k} \right) I =
\mu_g \frac{\tilde\tau_k+1}{\tilde\tau_k} I = \mu_g \frac{c_k}{\tilde\tau_{k+1}} I.
\end{align*}
By substituting \eqref{eq:is4} into \eqref{eq: acc_rate_eq1.6}, using $c_k = \sigma_{k+1}/\sigma_k$ from \eqref{eq:defck}, and taking total expectation on both sides, we obtain
\begin{equation*}
    c_k\mathbb{E} \left[ \frac{\mu_g}{\tilde\tau_{k+1}}\| x_\star-x_{k+1} \|^2 +\frac{n}{\sigma_{k+1}}\|y_\star-y_{k+1}\|^2\right] \leq \mathbb{E}\left[\frac{\mu_g}{\tilde\tau_{k}}\| x_\star-x_{k} \|^2 +\frac{n}{\sigma_{k}}\|y_\star-y_{k}\|^2\right].
\end{equation*}
Using \eqref{eq:defck} again, we have
\begin{equation*}
    \frac{1}{\tilde \tau_{k+1}}\mathbb{E} \left[ \frac{\mu_g}{\tilde\tau_{k+1}}\| x_\star-x_{k+1} \|^2 +\frac{n}{\sigma_{k+1}}\|y_\star-y_{k+1}\|^2\right] \leq \frac{1}{\tilde \tau_{k}}\mathbb{E}\left[\frac{\mu_g}{\tilde\tau_{k}}\| x_\star-x_{k} \|^2 +\frac{n}{\sigma_{k}}\|y_\star-y_{k}\|^2\right].
\end{equation*}
By iterating the inequality for $k=0,1,\dotsc,K-1$, we obtain
\begin{equation*}
    \frac{1}{\tilde \tau_K}\mathbb{E}\left[\frac{\mu_g}{\tilde\tau_{K}} \| x_\star-x_{K} \|^2 +\frac{n}{\sigma_K}\|y_\star-y_{K}\|^2\right] \leq \frac{1}{\tilde \tau_0}\left(\frac{\mu_g}{\tilde\tau_{0}}\| x_\star-x_{0} \|^2 + \frac{n}{\sigma_0}\|y_\star-y_{0}\|^2\right).
\end{equation*}
Next, we drop the second term on LHS and multiply both sides by $\tilde \tau_K^2/\mu_g$ to obtain
\begin{equation}\label{eq: scc_so_eq1}
    \mathbb{E}\| x_\star-x_{K} \|^2 \leq \left( \frac{\tilde\tau_K^2}{\mu_g\tilde \tau_0} \right)\left(\frac{\mu_g}{\tilde\tau_{0}}\| x_\star-x_{0} \|^2 + \frac{n}{\sigma_0}\|y_\star-y_{0}\|^2\right).
\end{equation}
Since $\tilde\tau_k\leq \tilde \tau_0 \leq 1$, we will estimate as in \cite{chambolle2011first}~(see~\Cref{lem: acc_rate_seq}) to deduce that $\tilde\tau_K \ge 3/K$ for all $K>0$ and hence
$\tilde\tau_K^2 \leq \frac{9}{K^2}$.
We thus have from~\eqref{eq: scc_so_eq1} that
\begin{equation}\label{eq: ge2}
    \mathbb{E}\| x_\star-x_{K} \|^2 \leq \frac{9}{K^2} \left(\frac{1}{\tilde\tau_{0}^2}\| x_\star-x_{0} \|^2 + \frac{n}{\mu_g\sigma_0\tilde\tau_0}\|y_\star-y_{0}\|^2\right).
\end{equation}
We now derive upper bounds on the coefficients in the RHS of this bound. 
By our choice of $\sigma_0$ and $\tilde \tau_0$ in~\eqref{eq:is9}, we have
\[
    \frac{1}{\tilde \tau_0^2} = \frac{1}{\left(\min\left( \frac{1}{n}, \frac{\mu_g}{n\max_i\|A_i\|} \right)\right)^2} = n^2 \max\left( 1, \frac{\max_i\|A_i\|^2}{\mu_g^2} \right). 
\]
and
\begin{align*}
    \frac{n}{\mu_g \sigma_0 \tilde \tau_0} = \frac{n\max_i\|A_i\|}{\mu_g \min\left( \frac{1}{n}, \frac{\mu_g}{n\max_i\|A_i\|} \right)} 
    & = \frac{n\max_i\|A_i\|}{\mu_g} \max\left( n, \frac{n\max_i\|A_i\|}{\mu_g} \right) \\
    & = n^2 \max \left( \frac{\max_i\|A_i\|}{\mu_g},
    \frac{\max_i\|A_i\|^2}{\mu_g^2} \right)
    \le 
    n^2 \max \left( 1,
    \frac{\max_i\|A_i\|^2}{\mu_g^2} 
    \right).
\end{align*}
By substituting these bounds in \eqref{eq: ge2}, we obtain
\begin{equation} \label{eq:wh7}
    \mathbb{E}\| x_\star - x_K \|^2\leq \frac{9n^2}{K^2}  \max\left( 1, 
    \frac{\max_i \|A_i\|^2}{\mu_g^2} \right) D_{\star}.
\end{equation}
Thus to ensure that $\mathbb{E}\| x_\star - x_K \|^2 \le \varepsilon$, it is sufficient by \eqref{eq:wh7} that \eqref{eq:K5} holds. We obtain the result for complexity by multiplying this bound on $K$ by the expected number of operations per iteration, which is $O(\nnz(A)/n)$ (see \Cref{eq: sparse_cost_rem})  and also accounting for the additional $\nnz(A)$ cost of computing $A^\top y_0$.
\end{proof}

\subsection{Analysis for the Convex-Strongly Concave Case}

We deal now with the case in which $h^*$ is strongly convex while $g$ is only convex.
The following theorem establishes the complexity result for PURE-CD from the second column of \Cref{tb: 3}
\begin{theorem}\label{th: thm_dual_sc}
Suppose that~\Cref{asmp: asmp1},~\ref{asmp: asmp2} hold and that $h^\ast$ is $\mu_h > 0$ strongly convex. Recall $\pi^{(j)} = \frac{| i\in I(j)|}{n}$. In \Cref{alg: purecd_sparse}, we set 
\begin{equation} \label{eq:is0}
\begin{split}
\tilde \tau_0 &= \frac{\mu_h}{n \max_i \|A_i\|^2},~~~~\tilde \sigma_0 = \frac{1}{2n-1}, ~~~~\Sigma_k =\sigma_k I,\\
 \tilde \tau_{k+1} &= \tilde \tau_k\sqrt{1+\tilde \sigma_k},~~~~\tau_k^{(j)} = \frac{\tilde \tau_k}{\pi^{(j)}} ~~~~\tilde \sigma_{k+1} = \frac{\tilde \sigma_k}{\sqrt{1+\tilde\sigma_k}},~~~~ \sigma_{k} = \frac{n\tilde\sigma_k}{\mu_h-(n-1)\mu_h\tilde\sigma_k},~~~~\theta_{k}^{(j)} =\pi^{(j)} n.
\end{split}
\end{equation}
Then we have that $\mathbb{E}\left[ \| y_{K} - y_\star \|^2 \right] \leq \varepsilon$ for
\begin{equation} \label{eq:K6}
K \ge 6n \sqrt{\frac{D_\star}{\varepsilon}} \max \left( \frac{\max_i \|A_i \|}{\mu_h},1 \right),
\end{equation}
with expected complexity
\begin{equation*}
{O}\left( \nnz(A) + \nnz(A) \sqrt{\frac{D_\star}{\varepsilon}} \max \left( \frac{\max_i \| A_i\|}{\mu_h},1 \right) \right),
\end{equation*}
where $D_\star$ is defined in \Cref{sec: notation}.
\end{theorem}
\begin{proof} 
First, we note from~\eqref{eq:is0} that $\{\tilde\sigma_k\}$ is a positive decreasing sequence, so $\tilde \sigma_k \in \left(0, \frac{1}{2n-1} \right]$ for all $k$. 
As a result, we have for the denominator of $\sigma_k$ in~\eqref{eq:is0} that 
\begin{equation*}
    \mu_h - \mu_h(n-1)\tilde \sigma_k \geq  \mu_h\left[ 1-\frac{n-1}{2n-1} \right] > 0.
\end{equation*}
It follows  that $\sigma_k > 0$ for all $k$.
By rearranging the definition of $\sigma_k$ we obtain
\begin{equation}\label{eq: csc_tildesigma1}
\tilde \sigma_k = \frac{\sigma_k \mu_h}{n+(n-1)\mu_h\sigma_k}.
\end{equation}
We apply Lemma~\ref{lem: sparse_one_it} with $\mu_h > 0$, $\mu_g= 0$, $\theta^{(j)} = \pi^{(j)}n$ where $x_\star$ is any primal solution, and drop the nonnegative term $G(\bar{x}_{k+1},\bar{y}_{k+1},x_\star,y_\star)$ on the LHS to obtain
\begin{equation}\label{eq: csc_ineq1}
\begin{aligned}
& \mathbb{E}_k \| x_\star-x_{k+1} \|^2_{\Tau_k^{-1} \Pi^{-1}} + \frac{n(1+\sigma_k\mu_h)}{\sigma_k}\mathbb{E}_k\| y_\star- y_{k+1}\|^2 \\
& \leq \| x_\star-x_k\|^2_{\Tau_k^{-1} \Pi^{-1}}+\frac{n(1+\sigma_k\mu_h)-\sigma_k\mu_h}{\sigma_k}\| y_\star-y_k\|^2 
- \frac{1}{\sigma_k}\| \bar y_{k+1} - y_k\|^2 + \|\bar y_{k+1} - y_k\|^2_{M_y(\Tau_k^{-1}\Pi^{-1})}.
\end{aligned}
\end{equation}
To drop the last two terms in this inequality, we need, by the definition of $M_y$ in Lemma~\ref{lem: sparse_one_it}, that
\begin{equation}\label{eq: ss_cond_lp1}
    \frac{1}{\sigma_k} - \sum_{j=1}^d \frac{1}{n\pi^{(j)}}\tau_{k}^{(j)}(\theta_{k}^{(j)})^2 A_{i, j}^2 \geq 0 \iff
    \frac{1}{\sigma_k} - \sum_{j=1}^d n\pi^{(j)} \tau_{k}^{(j)} A_{i,j}^2 \geq 0 \iff \sigma_k \tilde \tau_k n \|A_i\|^2 \leq 1,
\end{equation}
where the first equivalence is due to the definition of $\theta^{(j)}_k$ and the second equivalence is by the definition of $\tilde \tau_k$ from~\eqref{eq:is0}.
We next write \eqref{eq: ss_cond_lp1} in terms of $\sigma_0$ and $\tilde \tau_0$.
We define
\begin{equation}
    \label{def:ck2}
    c_k:= \sqrt{1+\tilde\sigma_k} = \frac{\tilde\sigma_k}{\tilde\sigma_{k+1}} = \frac{\tilde\tau_{k+1}}{\tilde\tau_k} \ge 1.
\end{equation}
Similar to \cite[Thm. 5.1]{chambolle2018stochastic}, using \eqref{eq:is0}, we have
\[
\frac{\sigma_{k+1}}{\sigma_k} = \frac{n\tilde \sigma_{k+1}}{\mu_h-(n-1)\mu_h\tilde \sigma_{k+1}}\frac{\mu_h-(n-1)\mu_h\tilde \sigma_k}{n\tilde \sigma_k} =\frac{1}{c_k}\frac{1-(n-1)\tilde\sigma_{k}}{1-(n-1)\tilde \sigma_{k+1}}= \frac{1}{c_k}\frac{1-(n-1)\tilde\sigma_{k+1}c_k}{1-(n-1)\tilde \sigma_{k+1}}\leq\frac{1}{c_k},
\]
where the last inequality uses $c_k \ge 1$.
It follows from this inequality together with \eqref{def:ck2} that $\tilde \tau_{k} \sigma_{k} = \tilde \tau_{k-1}c_{k-1} \sigma_{k} \leq \tilde \tau_{k-1}c_{k-1} \frac{\sigma_{k-1}}{c_{k-1}} = \tilde\tau_{k-1}\sigma_{k-1}$, so by recursion we have $\tilde \tau_{k} \sigma_{k} \leq \tilde \tau_0 \sigma_0$ for all $k$. 
Therefore, \eqref{eq: ss_cond_lp1} is implied by
\[
\sigma_0 \tilde \tau_0 n \|A_i\|^2 \leq 1.
\]
We verify this condition using the definitions \eqref{eq:is0}:
\[
\sigma_0 \tilde \tau_0 n \|A_i\|^2  = \frac{n \tilde\sigma_0}{\mu_h - (n-1)\mu_h \tilde\sigma_0} \frac{\mu_h}{n \max_i \| A_i \|^2} n \| A_i \|^2 
\le \frac{n \tilde\sigma_0}{1-(n-1) \tilde\sigma_0} = \frac{\frac{n}{2n-1}}{1-\frac{n-1}{2n-1}} = 1,
\]
as required.

At this point we have established that \eqref{eq: ss_cond_lp1} holds, so we can drop the last two terms on~\eqref{eq: csc_ineq1} and write
\begin{equation} \label{eq:bs9}
\mathbb{E}_k \| x_\star-x_{k+1} \|^2_{\Tau_k^{-1} \Pi^{-1}} + \frac{n(1+\sigma_k\mu_h)}{\sigma_k}\mathbb{E}_k\| y_\star- y_{k+1}\|^2 \leq \| x_\star-x_k\|^2_{\Tau_k^{-1} \Pi^{-1}}+\frac{n(1+\sigma_k\mu_h)-\sigma_k\mu_h}{\sigma_k}\| y_\star-y_k\|^2.
\end{equation}
By the definition of $\tilde \tau_k$ from~\eqref{eq:is0}, we have that $\Tau_k^{-1}\Pi^{-1} = \tilde \tau_k^{-1} I$. We also have from \eqref{eq: csc_tildesigma1} that 
\[
\frac{n(1+\sigma_k\mu_h)-\sigma_k\mu_h}{\sigma_k} = \frac{\mu_h}{\tilde\sigma_k}.
\]
We can thus simplify \eqref{eq:bs9} as follows:
\begin{equation}\label{eq: jy3}
\frac{1}{\tilde \tau_{k}}\mathbb{E}_k \| x_\star-x_{k+1} \|^2 + \frac{n(1+\sigma_k\mu_h)}{\sigma_k}\mathbb{E}_k\| y_\star- y_{k+1}\|^2 \leq \frac{1}{\tilde \tau_k} \| x_\star-x_k\|^2+\frac{\mu_h}{\tilde \sigma_k}\| y_\star-y_k\|^2.
\end{equation}
We claim that  $\frac{n(1+\sigma_k\mu_h)}{\sigma_k}\| y_\star - y_{k+1} \|^2 \ge c_k \frac{\mu_h}{\tilde \sigma_{k+1}}\| y_\star - y_{k+1} \|^2$, which is a consequence of 
$\frac{c_k}{\tilde\sigma_{k+1}} \le \frac{n(1+\mu_h \sigma_k)}{\mu_h \sigma_k}$.
This bound follows from \eqref{def:ck2} and \eqref{eq: csc_tildesigma1}, since
\[
\frac{c_k}{\tilde\sigma_{k+1}}  = \frac{1+\tilde\sigma_k}{\tilde\sigma_k} =
1+\frac{1}{\tilde\sigma_k} = 1 + \frac{n+(n-1)\mu_h \sigma_k}{\mu_h \sigma_k} =
\frac{n(1+\mu_h \sigma_k)}{\mu_h \sigma_k}.
\]
By making this substitution into the second term on the LHS of \eqref{eq: jy3}, together with $\frac{1}{\tilde\tau_k} = \frac{c_k}{\tilde\tau_{k+1}}$ in the first term, we obtain 
\[
c_k \left[ \frac{1}{\tilde \tau_{k+1}}\mathbb{E}_k \| x_\star-x_{k+1} \|^2 + \frac{\mu_h}{\tilde\sigma_{k+1}} \mathbb{E}_k\| y_\star- y_{k+1}\|^2 \right] \leq \frac{1}{\tilde \tau_k} \| x_\star-x_k\|^2+\frac{\mu_h}{\tilde \sigma_k}\| y_\star-y_k\|^2.
\]
By taking total expectation on both sides, we obtain
\[
c_k\mathbb{E}\left[ \frac{1}{\tilde \tau_{k+1}} \| x_\star - x_{k+1} \|^2 + \frac{\mu_h}{\tilde \sigma_{k+1}} \| y_\star-y_{k+1} \|^2 \right] \leq  \mathbb{E}\left[\frac{1}{\tilde \tau_k}\| x_\star-x_k\|^2 + \frac{\mu_h}{\tilde \sigma_k} \| y_\star-y_k\|^2\right].
\]
By iterating the inequality and using $c_k = \frac{\tilde \sigma_k}{\tilde \sigma_{k+1}}$ from \eqref{def:ck2}, we obtain
\[
\frac{1}{\tilde \sigma_{K}}\mathbb{E}\left[ \frac{1}{\tilde \tau_{K}} \| x_\star - x_{K} \|^2 + \frac{\mu_h}{\tilde \sigma_{K}} \| y_\star-y_{K} \|^2 \right] \leq \frac{1}{\tilde \sigma_0} \left(\frac{1}{\tilde \tau_0} \| x_\star-x_0\|^2 + \frac{\mu_h}{\tilde \sigma_0} \| y_\star-y_0\|^2\right)
\]
which, after dropping the first term on LHS and multiplying both sides by $\frac{\tilde \sigma_K^2}{\mu_h}$, yields
\begin{align*}
\mathbb{E}\left[\| y_\star-y_{K} \|^2 \right] \leq \frac{\tilde \sigma_K^2}{\mu_h\tilde \sigma_0} \left(\frac{1}{\tilde \tau_0} \| x_\star-x_0\|^2 + \frac{\mu_h}{\tilde \sigma_0} \| y_\star-y_0\|^2\right).
\end{align*}
By applying Lemma~\ref{lem: acc_rate_seq} to $\tilde\sigma_k$, we have
$\tilde\sigma_K^2\leq9/K^2$ 
With $\tilde \sigma_0 = \frac{1}{2n-1} \ge \frac{1}{2n}$ and $\tilde \tau_0 = \frac{\mu_h}{n\max_i\|A_i\|^2}$, we have
\begin{align*}
\mathbb{E}\| y_\star - y_K\|^2 &\leq 
\tilde\sigma_K^2 \left( \frac{1}{\mu_h \tilde\sigma_0 \tilde\tau_0} \| x_\star - x_0 \|^2 + \frac{1}{\tilde\sigma_0^2} \| y_\star-y_0 \|^2 \right) \\
&= \frac{9}{K^2}\max \left( \frac{2n^2 \max_i \|A_i\|^2}{\mu_h^2}, 4n^2 \right) D_\star \le \frac{36n^2}{K^2} \max \left( \frac{\max_i \|A_i\|^2}{\mu_h^2},1 \right) D_\star.
\end{align*}
Thus, the condition \eqref{eq:K6} on $K$ is sufficient for $\mathbb{E}\| y_\star - y_K\|^2 \le \varepsilon$. We obtain the complexity by multiplying by the expected per-iteration cost of $\nnz(A)/n$ (see \Cref{eq: sparse_cost_rem}) and adding a cost of $\nnz(A)$ for the computation of $A^\top y_0$.
\end{proof}

\section{Conclusions}\label{sec: conclusions}
Our aim in this paper is to provide a comprehensive understanding of primal-dual coordinate methods for convex-concave min-max problems with bilinear coupling.
Our results complement the developments in~\cite{alacaoglu2020random} and~\cite{song2021variance} to show that with suitable analysis, we can match and improve the best-known complexities in a wide array of special cases, by using the relatively simple PURE-CD algorithm.
We conclude by highlighting an open question.

As we highlighted in Table~\ref{tb: 2}, our analysis in the sparse case does not provide an improved complexity guarantee for optimality measure $\mathbb{E} \Gap(\xout, \yout)$ compared to deterministic methods.
As we show in Thm.~\ref{th: thm_max_e}, it is relatively straightforward to prove this complexity for the quantity $\max_{x, y}\mathbb{E} G(\xout, \yout, x, y)$.
Moreover, as we showed in \Cref{sec: str_cvx_str_ccv} and~\ref{sec: str_cvx_ccv}, it is also possible to obtain improved complexity results depending on sparsity when there is strong convexity in the problem.
However the interplay of expectation and maximum in the sparse case seems to be a roadblock to prove the desired result for expected gap for solving convex-concave problems.
It is necessary to resolve this important question to complete our understanding of primal-dual coordinate methods that can adapt to sparsity in the data.

\section*{Acknowledgments}

Research of A. Alacaoglu was supported in part by NSF award 2023239 and DOE ASCR Subcontract 8F-30039 from Argonne National Laboratory.

Research of V. Cevher was supported in part by the European Research Council (ERC) under the European Union’s Horizon 2020 research and innovation programme (grant agreement no 725594 - timedata); the Swiss National Science Foundation (SNSF) under grant number 200021\_178865/1; the Department of the Navy, Office of Naval Research (ONR) under a grant number N62909-17-1-2111; and the Hasler Foundation Program: Cyber Human Systems (project number 16066).

Research of S. J. Wright was supported in part by NSF awards 1934612 and 2023239, Subcontract 8F-30039 from Argonne National Laboratory, and an AFOSR subcontract UTA20-001224 from UT-Austin. 
Part of this work was done while this author was visiting the Simons Institute for the Theory of Computing (Berkeley), the Hausdorff Institute of Mathematics (Bonn), and the Oden Institute for Computational Engineering and Sciences at UT-Austin.

\appendix

\section{Preliminary Lemmas}

Here we collect several technical results that are needed in the proofs of the main results.

The first lemma concerns the inequality resulting from the initialization. 
It is a standard one-iteration inequality for alternating gradient descent-ascent, but we include it here to make this report self-contained.
\begin{lemma}\label{lem: lem_init}
Given $x_1, y_1$ defined as~\eqref{eq: purecd_init}, and $\tau = \frac{1}{n\max_i\|A_i\|}$, $\underline \sigma = \frac{\gamma}{n\max_i\|A_i\|}$ for $\gamma < 1$, it holds for any $x, y$ that
\begin{multline*}
 h^\ast(y_1) -  \langle Ax, y_1 \rangle \leq  h^\ast(y) + g(x) -  g(x_1) -  \langle Ax_1, y \rangle - \frac{n\max_i\|A_i\|}{2 \gamma } \Big(\| y-y_1\|^2 - \| y-y_0 \|^2 \Big)\\
 - \frac{n\max_i\|A_i\|}{2}  \| x-x_1\|^2 +n\max_i\|A_i\| \| x-x_0\|^2.
\end{multline*}
Moreover, for any solution $(x_\star, y_\star)$, we have 
\begin{equation*}
    \| x_\star - x_1\|^2 + \frac{1}{\gamma} \| y_\star - y_1\|^2 \leq \frac{2}{\gamma}  \left( \| x_\star - x_0\|^2 + \| y_\star - y_0\|^2 \right) = \frac{2}{\gamma}  D_\star,
\end{equation*}
and
\begin{equation*}
    \| y_1 - y_0\|^2 \leq \frac{2}{1-\gamma}  \left( \| x_\star - x_0\|^2 + \| y_\star - y_0\|^2 \right) = \frac{2}{1-\gamma}  D_\star,
\end{equation*}
where $D_\star$ is defined in \eqref{eq:Dstar}.
\end{lemma}
\begin{proof}
By prox-inequality (see~\cref{eq: prox_ineq,eq: purecd_dual_ineq1}) applied on~\eqref{eq: purecd_init} for $y_1$, we have for any $y$ that
\begin{align*}
&\frac{1}{2\underline\sigma} \| y-y_1 \|^2 + h^\ast(y_1)- h^\ast(y) \leq \frac{1}{2\underline\sigma}\| y-y_0\|^2 - \langle Ax_1, y- y_{1} \rangle - \frac{1}{2\underline\sigma}  \|y_1-y_0\|^2 \\
\iff  &h^\ast(y_1) \leq  h^\ast(y) +  \langle Ax_1, y_1 - y \rangle
- \frac{1}{2\underline\sigma} \Big(\| y-y_1\|^2 - \| y-y_0 \|^2 + \|y_0-y_1\|^2\Big).
\end{align*}
By adding to both sides $-\langle Ax, y_1 \rangle$, we obtain
\begin{equation}\label{eq: agda_eq1}
 h^\ast(y_1) -  \langle Ax, y_1 \rangle \leq  h^\ast(y) +  \langle Ax_1, y_1 - y \rangle -  \langle Ax, y_1 \rangle \\
- \frac{1}{2\underline\sigma} \Big(\| y-y_1\|^2 - \| y-y_0 \|^2 + \|y_0-y_1\|^2\Big).
\end{equation}
Similarly, using prox-inequality (see~\cref{eq: prox_ineq,eq: purecd_primal_ineq1}) on~\eqref{eq: purecd_init} for $x_1$, we have for any $x$ that
\begin{align}
&\frac{1}{2\tau} \| x-x_1\|^2 + g(x_1) - g(x) \leq \frac{1}{2\tau} \| x-x_0\|^2 + \langle x - x_1, A^\top y_0 \rangle - \frac{1}{2\tau} \|x_1 -  x_0\|^2 \notag \\
 \iff &0 \leq  g(x) - g(x_1) -  \langle x_1 - x, A^\top y_0 \rangle - \frac{1}{2\tau} \Big( \| x-x_1\|^2 - \| x-x_0\|^2 + \| x_1 - x_0\|^2 \Big).\label{eq: agda_eq2}
\end{align}
When we sum~\eqref{eq: agda_eq1} and~\eqref{eq: agda_eq2}, we have
\begin{multline*}
h^\ast(y_1) -  \langle Ax, y_1 \rangle \leq  h^\ast(y) + g(x) -  g(x_1)
- \frac{1}{2\underline \sigma} \Big(\| y-y_1\|^2 - \| y-y_0 \|^2 + \|y_1-y_0\|^2\Big) \\
- \frac{1}{2\tau} \Big( \| x-x_1\|^2 - \| x-x_0\|^2 + \| x_1 - x_0\|^2 \Big)
+\langle Ax_1, y_1 - y \rangle - \langle Ax, y_1 \rangle - \langle x_1 - x, A^\top y_0 \rangle.
\end{multline*}
We next manipulate the inner products on the RHS as
\begin{align*}
    \langle Ax_1, y_1 - y \rangle - \langle Ax, y_1 \rangle - \langle x_1 - x, A^\top y_0 \rangle & = -\langle Ax_1, y \rangle + \langle A(x_1 - x), y_1 \rangle - \langle x_1 - x, A^\top y_0\rangle \\
    & = -\langle Ax_1, y \rangle + \langle A(x_1 - x), y_1 - y_0 \rangle.
\end{align*}
 We consequently have
\begin{multline}\label{eq: df3}
 h^\ast(y_1) -  \langle Ax, y_1 \rangle \leq  h^\ast(y) + g(x) -  g(x_1) -  \langle Ax_1, y \rangle
+ \langle A(x_1 - x), y_1 - y_0 \rangle \\
- \frac{1}{2\underline \sigma} \Big(\| y-y_1\|^2 - \| y-y_0 \|^2 + \|y_1-y_0\|^2\Big) - \frac{1}{2\tau} \Big( \| x-x_1\|^2 - \| x-x_0\|^2 + \| x_1 - x_0\|^2 \Big).
\end{multline}
For the inner product $ \langle A(x_1 - x), y_1 - y_0 \rangle$, we use $\|A\| \leq \sqrt{n}\max_i \|A_i\|$ and Young's inequality twice to write
\begin{align*}
     \langle A(x_1 - x), y_1 - y_0 \rangle &\leq \|A\|\| y_1 - y_0\|\|x_1 - x\| \\
    &\leq \sqrt{n}\max_i\|A_i\| \| y_1 - y_0\|\|x_1 - x\| \\
    &\leq  \frac{n\max_i\|A_i\|}{2} \| y_1 - y_0\|^2 + \frac{\max_i\|A_i\|}{2} \|x_1 - x\|^2 \\
    &\leq  \frac{n\max_i\|A_i\|}{2} \| y_1 - y_0\|^2 + \max_i\|A_i\| \|x_0 - x\|^2 + \max_i\|A_i\| \|x_0 - x_1\|^2.
\end{align*}
Since $n \geq 2$, we have that $\max_i\|A_i\| \leq \frac{n}{2} \max_i\|A_i\|$ and hence we have
\begin{align*}
     \langle A(x_1 - x), y_1 - y_0 \rangle \leq  \frac{n\max_i\|A_i\|}{2} \left( \| y_1 - y_0\|^2 + \|x_0 - x\|^2 + \|x_0 - x_1\|^2 \right).
\end{align*}
By substituting this bound into~\eqref{eq: df3} and using  $\tau = \frac{1}{n\max_i\|A_i\|}$, $\underline \sigma = \frac{\gamma}{n\max_i\|A_i\|}$, we obtain
\begin{multline*}
 h^\ast(y_1) -  \langle Ax, y_1 \rangle \leq  h^\ast(y) + g(x) -  g(x_1) -  \langle Ax_1, y \rangle \\
+\frac{n\max_i\|A_i\|}{2} \left( \| y_1 - y_0\|^2 + \|x_0 - x\|^2 + \|x_0 - x_1\|^2 \right) \\
- \frac{n\max_i\|A_i\|}{2 \gamma } \Big(\| y-y_1\|^2 - \| y-y_0 \|^2 + \|y_1-y_0\|^2\Big) - \frac{n\max_i\|A_i\|}{2} \Big( \| x-x_1\|^2 - \| x-x_0\|^2 + \| x_1 - x_0\|^2 \Big).
\end{multline*}
By combining the coefficients of $\| y_1-y_0 \|^2$, $\| x_1-x_0 \|^2$, and $\|x-x_0 \|^2$, we obtain
\begin{multline}\label{eq: gr4}
 h^\ast(y_1) -  \langle Ax, y_1 \rangle \leq h^\ast(y) + g(x) - g(x_1) -  \langle Ax_1, y \rangle - \frac{n\max_i\|A_i\|}{2 \gamma } \Big(\| y-y_1\|^2 - \| y-y_0 \|^2 \Big) \\
 - \frac{n\max_i\|A_i\|}{2}  \| x-x_1\|^2 +n\max_i\|A_i\| \| x-x_0\|^2 - \frac{(1-\gamma) n \max_i\|A_i\|}{2\gamma} \| y_1 - y_0\|^2.
\end{multline}
This gives the first result of the lemma after dropping the final term on the RHS.

For the other results, we substitute $(x, y) = (x_\star, y_\star)$ into~\eqref{eq: gr4} and rearrange to obtain
\begin{multline*}
    \left[g(x_1) + \langle Ax_1, y_\star \rangle - h^\ast(y_\star) - g(x_\star) - \langle Ax_\star, y_1 \rangle + h^\ast(y_1) \right] + \frac{n\max_i\|A_i\|}{2}\left( \| x_\star - x_1 \|^2 + \frac{1}{\gamma} \| y_\star - y_1 \|^2 \right) \\
 + \frac{(1-\gamma) n \max_i\|A_i\|}{2\gamma} \| y_1 - y_0\|^2 \leq \frac{n\max_i\|A_i\|}{2}\left( 2\| x_\star - x_0 \|^2 + \frac{1}{\gamma} \| y_\star - y_0 \|^2 \right).
\end{multline*}
We have that the terms inside the bracket in LHS is nonnegative due to the definition of a saddle point (see also~\eqref{eq: g_deff}).
We then divide both sides by $\frac{n\max_i\|A_i\|}{2}$ to get
\begin{align*}
     \| x_\star - x_1 \|^2 + \frac{1}{\gamma} \| y_\star - y_1 \|^2 
 + \frac{(1-\gamma)}{\gamma} \| y_1 - y_0\|^2
 & \leq  2\| x_\star - x_0 \|^2 + \frac{1}{\gamma} \| y_\star - y_0 \|^2 \\
 & \leq \frac{2}{\gamma}\left(\| x_\star - x_0 \|^2 + \| y_\star - y_0 \|^2\right) \le \frac{2}{\gamma} D_\star,
\end{align*}
where we have used $\gamma < 1$.
The second result of the lemma follows since the third term on the LHS is nonnegative.
For the last result, we use that the first two terms on the LHS are nonnegative and multiply both sides of the resulting inequality with $\frac{\gamma}{1-\gamma}$.
\end{proof}

Next, we give the lemma used to  decouple supremum and expectation, using ideas from~\cite{alacaoglu2019convergence,nemirovski2009robust}.
This lemma is in a slightly more general form to accommodate different forms of error terms, but is otherwise the same as in these earlier works~\cite[Lemmas~3.1 and 6.1]{nemirovski2009robust},~\cite[Lemma~4.8]{alacaoglu2019convergence}. 
\begin{lemma}\label{lem: exp_max_lemma}
Given $\underline k \geq 0$ and a fixed vector $u_{\underline k}\in\mathbb{R}^n$, let $P, \Sigma$ be as defined in Algorithms~\ref{alg: purecd_sparse} and \ref{alg: purecd_dense}, and suppose that  $v_{k+1} = (\bar y_{k+1}-y_k) - P^{-1} (y_{k+1}-y_k)$.
Let $\mathcal{U} \subseteq \mathbb{R}^n$ be any set. Then for all $K>\underline{k}$, we have
\begin{equation*}
\mathbb{E}\max_{u\in\mathcal{U}} \left[\sum_{k=\underline k}^{K-1}2\langle u, v_{k+1} \rangle_{\Sigma^{-1}} - \|  u_{\underline k} - u \|^2_{\Sigma^{-1}P^{-1}}\right] \leq \sum_{k=\underline k}^{K-1} \mathbb{E} \| y_k - y_{k+1} \|^2_{\Sigma^{-1}P^{-1}}.
\end{equation*}
\end{lemma}
\begin{proof}
For $k \ge \underline{k}$, define $ u_{k+1} =  u_k + P v_{k+1}$. We have
\begin{equation*}
\| u_{k+1} - u \|^2_{\Sigma^{-1}P^{-1}} = \| u_k - u \|^2_{\Sigma^{-1}P^{-1}} + 2 \langle P v_{k+1}, u_k - u \rangle_{\Sigma^{-1}P^{-1}} + \| P v_{k+1} \|^2_{\Sigma^{-1}P^{-1}}.
\end{equation*}
By rearranging this inequality, summing from $k=\underline{k}, \dotsc,K-1$, telescoping, and dropping the term $\| u_K-u\|_{\Sigma^{-1}P^{-1}}^2$, we obtain
\begin{equation}\label{eq: ho5}
\sum_{k=\underline k}^{K-1} 2\langle v_{k+1}, u \rangle_{\Sigma^{-1}} - \| u_{\underline k} - u \|^2_{\Sigma^{-1}P^{-1}} \leq \sum_{k=\underline k}^{K-1} \| v_{k+1} \|^2_{\Sigma^{-1}P} + 2 \langle v_{k+1}, u_k \rangle_{\Sigma^{-1}}.
\end{equation}
We wish to bound the $\mathbb{E} \max_u$ of the LHS, which we can do by bounding the 
expectation of the RHS.

First, when we condition on knowing up to $y_k$, by construction $u_k$ is deterministic since it depends on $v_k$ and consequently $y_k$. We recall the definition of $\mathbb{E}_k$ as the conditional expectation defined in \Cref{sec: notation}.
We next have 
$\mathbb{E}_k [v_{k+1}] = 0$, since $\mathbb{E}_k [P^{-1}(y_{k+1} - y_k)] = \bar y_{k+1} - y_k$, we have
\begin{equation} 
\label{eq:hu3}
    \mathbb{E}_k \langle v_{k+1}, u_k \rangle_{\Sigma^{-1}} = 0.
\end{equation}
Second, we use $\mathbb{E} \| X - \mathbb{E} X \|^2 = \mathbb{E} \| X\|^2 - \| \mathbb{E} X \|^2 \leq \mathbb{E} \| X \|^2$ with $X = P^{-1}(y_k-y_{k+1})$ to obtain
\begin{align*}
\mathbb{E} \| v_{k+1} \|^2_{\Sigma^{-1}P} &= \mathbb{E} \| \bar y_{k+1} - y_k - P^{-1}(y_{k+1} - y_k) \|^2_{\Sigma^{-1}P} \\
& = \mathbb{E}\left[\mathbb{E}_k \| \bar y_{k+1} - y_k - P^{-1}(y_{k+1} - y_k) \|^2_{\Sigma^{-1}P}\right]  \\
&=\mathbb{E}\left[\mathbb{E}_k \| \mathbb{E}_k[P^{-1}( y_{k+1} - y_k)] - P^{-1}(y_{k+1} - y_k) \|^2_{\Sigma^{-1}P}\right] \\
& \leq \mathbb{E} \left[\mathbb{E}_k\| y_k-y_{k+1} \|^2_{\Sigma^{-1}P^{-1}}\right] \\
&=\mathbb{E} \| y_k-y_{k+1} \|^2_{\Sigma^{-1}P^{-1}}.
\end{align*}
We obtain the result by taking the expectation of the RHS in~\eqref{eq: ho5} and using  the bound just derived in combination with \eqref{eq:hu3}.
\end{proof}

The next lemma to compute the expectation of $\| x_{k+1} - x\|^2$ which is random when we use PURE-CD with sparsity.
This result is from~\cite[Lemma 2]{alacaoglu2020random}, but we include the proof for the sake of being self-contained.
\begin{lemma}{\cite[Lemma 2]{alacaoglu2020random}}\label{lem: tech_lemma_one_iter}
Let $x_{k+1}$ be computed as~\Cref{alg: purecd_sparse}, and let $\bar{y}_{k+1}$ be as defined in \eqref{eq:def.ybar}.
For any deterministic $x$ and $\Beta_k=\diag(\beta_k^{(1)},\dots,\beta_k^{(d)})\succ 0$, we have for any $k\geq 0$ that
\begin{multline}
\mathbb{E}_k \| x_{k+1} - x \|^2_{\Beta_k} = \| \bar x_{k+1} - x \|^2_{\Beta_k \Pi}  - \| x_k - x\|^2_{\Beta_k\Pi} + \|x_k - x\|^2_{\Beta_k} \\
-\frac{2}{n} \langle \Beta_k \Tau_k \Theta_k A^\top(\bar y_{k+1} - y_k), \bar x_{k+1} - x \rangle + \| \bar y_{k+1} - y_k\|^2_{M_y(\Beta_k)},
\end{multline}
 where $\pi^{(j)} = \frac{| i\in I(j)|}{n}$ and $M_y^{(i)}(\Beta_k) = n^{-1}\sum_{j=1}^d \beta_k^{(j)} (\tau_k^{(j)})^2 (\theta_k^{(j)})^2 A_{i, j}^2$.

Let $y_{k+1}$ be computed as \Cref{alg: purecd_sparse} or \Cref{alg: purecd_dense}. 
For any deterministic $y$ and $\Phi_k=\diag(\phi_k^{(1)}, \dots, \phi_k^{(n)})\succ 0$, we have for $k\geq 0$ (and for $k\geq 1$ in Sec.~\ref{sec: sif3}) that
\[
    \mathbb{E}_k \| y_{k+1} - y \|^2_{\Phi_k} = \|\bar y_{k+1} - y \|^2_{\Phi_k P} + \| y_k - y \|^2_{\Phi_k(I-P)}.
\]
Moreover, we have
\[
\mathbb{E}_k [h^\ast(y_{k+1})] = \sum_{i=1}^n p^{(i)} h_i^\ast(y_{k+1}^{(i)}) + \sum_{i=1}^n (1-p^{(i)}) h^\ast_i(y_k^{(i)}).
\]
\end{lemma}
\begin{proof}
For convenience, recall the definitions $J(i)=\{ j\in[d]\colon A_{i, j} \neq 0 \}\text{~~and~~}I(j)=\{ i\in[n]\colon A_{i, j} \neq 0 \}$.

By using the update rule of $x_{k+1}$ in step~\ref{st: rg4} and expanding the square, we write
\begin{align}
\mathbb{E}_k \| x_{k+1} - x\|^2_{\Beta_k} &= \mathbb{E}_k\sum_{j=1}^d \beta_k^{(j)} \left( x_{k+1}^{(j)} - x^{(j)} \right)^2\notag \\
&= \mathbb{E}_k\sum_{j\in J(i_k)} \beta_k^{(j)} \left( \bar x_{k+1}^{(j)} - \tau_k^{(j)}\theta_k^{(j)}(A^\top(y_{k+1} - y_k))^{(j)} - x^{(j)} \right)^2 + \sum_{j\not\in J(i_k)} \beta_k^{(j)} \left( x_k^{(j)} - x^{(j)} \right)^2 \notag \\
&= \mathbb{E}_k\sum_{j\in J(i_k)} \beta_k^{(j)} \left( \bar x_{k+1}^{(j)} - x^{(j)} \right)^2 - \mathbb{E}_k\sum_{j\in J(i_k)} 2\beta_k^{(j)}\tau_k^{(j)}\theta_k^{(j)}(A^\top(y_{k+1} - y_k))^{(j)}(\bar x_{k+1}^{(j)} - x^{(j)}) \notag \\
&\quad+\mathbb{E}_k\sum_{j\in J(i_k)} \beta_k^{(j)}\left(\tau_k^{(j)}\theta_k^{(j)}(A^\top(y_{k+1} - y_k))^{(j)}\right)^2 + \mathbb{E}_k\sum_{j\not\in J(i_k)} \beta_k^{(j)} \left( x_k^{(j)} - x^{(j)} \right)^2.\label{eq: tech_lemma_ineq1}
\end{align}
We now estimate all the terms separately.
First, we have
\begin{multline}\label{eq: sparse_purecd_tech_eq1}
\mathbb{E}_k \sum_{j\in J(i_k)} \beta_k^{(j)}\left( \bar x_{k+1}^{(j)} - x^{(j)} \right)^2 = \sum_{i=1}^n \frac{1}{n} \sum_{j\in J(i)} \beta_k^{(j)}\left( \bar x_{k+1}^{(j)} - x^{(j)} \right)^2 = \sum_{j=1}^d \sum_{i\in I(j)} \frac{\beta_k^{(j)}}{n}\left( \bar x_{k+1}^{(j)} - x^{(j)} \right)^2 \\
= \sum_{j=1}^d \pi^{(j)}\beta_k^{(j)}\left( \bar x_{k+1}^{(j)} - x^{(j)} \right)^2
= \| \bar x_{k+1} - x \|^2_{\Beta_k\Pi},
\end{multline}
where the second step used $\sum_{i=1}^n \sum_{j\in J(i)}\alpha_{i, j} = \sum_{j=1}^d \sum_{i\in I(j)}\alpha_{i, j}$ for any sequence $(\alpha_{i, j})_{i\in[n], j\in[d]}$ since both sums goes over all the indices corresponding to nonzero elements of matrix $A$ and sums the values of $\alpha$ in these indices. The second to last step used the definition $\pi^{(j)} = \sum_{i\in I(j)} 1/n$.

We next note that since $y_{k+1} - y_k$ is one sparse for $k\geq 0$, and that $y_{k+1}^{(i_k)} = \bar y_{k+1}^{(i_k)}$, we have
\begin{equation*}
(A^\top(y_{k+1} - y_k))^{(j)} = (A^\top(\bar y_{k+1}^{(i_k)} - y_k^{i_k})e_{i_k})^{(j)} = A_{i_k, j} (\bar y_{k+1}^{(i_k)} - y_k^{(i_k)}).
\end{equation*}
We use the last equality to exchange the order of summations to derive the following:
\begin{align}
\nonumber
& \mathbb{E}_k\sum_{j\in J(i_k)} 2\beta_k^{(j)}\tau_k^{(j)}\theta_k^{(j)}(A^\top(y_{k+1} - y_k))^{(j)}(\bar x_{k+1}^{(j)} - x^{(j)}) \\
\nonumber
& = \mathbb{E}_k\sum_{j\in J(i_k)} 2\beta_k^{(j)}\tau_k^{(j)}\theta_k^{(j)}A_{i_k, j}(\bar y_{k+1}^{(i_k)} - y_k^{(i_k)})(\bar x_{k+1}^{(j)} - x^{(j)}) \\
\nonumber
& = 
\sum_{i=1}^n \sum_{j\in J(i)} \frac{1}{n} 2\beta_k^{(j)}\tau_k^{(j)}\theta_k^{(j)}A_{i, j}(\bar y_{k+1}^{(i)} - y_k^{(i)})(\bar x_{k+1}^{(j)} - x^{(j)}) \\
\nonumber
& = 
\sum_{i=1}^n \sum_{j=1}^d \frac{1}{n} 2\beta_k^{(j)}\tau_k^{(j)}\theta_k^{(j)}A_{i, j}(\bar y_{k+1}^{(i)} - y_k^{(i)})(\bar x_{k+1}^{(j)} - x^{(j)}) \\
\nonumber
&= \frac{2}{n} \sum_{j=1}^d \beta_k^{(j)}\tau_k^{(j)}\theta_k^{(j)} \left( \sum_{i=1}^n A_{i, j}(\bar y_{k+1}^{(i)} - y_k^{(i)}) \right) (\bar x_{k+1}^{(j)} - x^{(j)}) \\
& = \frac{2}{n} \langle \Beta_k \Tau_k \Theta_k A^\top(\bar y_{k+1} - y_k), \bar x_{k+1} - x \rangle,\label{eq: tech_lemma_ineq2}
\end{align}
where the third equality is due to $A_{i, j} = 0$ for $j\not \in J(i)$.

With a similar estimation to the last equality, we have
\begin{align*}
\nonumber
& \mathbb{E}_k\sum_{j\in J(i_k)} \beta_k^{(j)}\left(\tau_k^{(j)}\theta_k^{(j)}(A^\top(y_{k+1} - y_k))^{(j)}\right)^2 \\
\nonumber
& = \mathbb{E}_k\sum_{j\in J(i_k)} \beta_k^{(j)}\left(\tau_k^{(j)}\theta_k^{(j)}A_{i_k, j}(\bar y_{k+1}^{(i_k)} - y_k^{(i_k)})\right)^2 
\\
\nonumber
& = \sum_{i=1}^n \sum_{j\in J(i)}  \frac{1}{n}\beta_k^{(j)}\left(\tau_k^{(j)}\theta_k^{(j)}A_{i, j}(\bar y_{k+1}^{(i)} - y_k^{(i)})\right)^2 \\
\nonumber
& = \sum_{i=1}^n \sum_{j=1}^d  \frac{1}{n}\beta_k^{(j)}\left(\tau_k^{(j)}\theta_k^{(j)}A_{i, j}(\bar y_{k+1}^{(i)} - y_k^{(i)})\right)^2 \\
\nonumber
& = \sum_{i=1}^n \left( \sum_{j=1}^d  \frac{1}{n}\beta_k^{(j)}(\tau_k^{(j)})^2(\theta_k^{(j)})^2A_{i, j}^2 \right) (\bar y_{k+1}^{(i)} - y_k^{(i)})^2 \\
&= \sum_{i=1}^n M_y^{(i)}(\Beta_k) (\bar y_{k+1}^{(i)} - y_k^{(i)})^2= \| \bar y_{k+1} - y_k \|^2_{M_y(\Beta_k)},\label{eq: tech_lemma_ineq3}
\end{align*}
where we have defined $M_y^{(i)}(\Beta_k) = n^{-1}\sum_{j=1}^d \beta_k^{(j)} (\tau_k^{(j)})^2 (\theta_k^{(j)})^2 A_{i, j}^2$.

We finally derive
\begin{align}
\mathbb{E}_k \sum_{j\not\in J(i_k)} \beta_k^{(j)} (x_k^{(j)} - x^{(j)})^2 & =  \sum_{j=1}^d \beta_k^{(j)} (x_k^{(j)} - x^{(j)})^2 - \mathbb{E}_k \sum_{j\in J(i_k)} \beta_k^{(j)} (x_k^{(j)} - x^{(j)})^2 \\
& = \| x_k - x \|^2_{\Beta_k} - \| x_k - x\|^2_{\Beta_k\Pi}.\label{eq: tech_lemma_ineq4}
\end{align}
where the last step uses the same estimation in~\eqref{eq: sparse_purecd_tech_eq1} with the difference of having $x_k$ instead of $\bar x_{k+1}$.

We insert~\cref{eq: sparse_purecd_tech_eq1,eq: tech_lemma_ineq2,eq: tech_lemma_ineq3,eq: tech_lemma_ineq4} into~\eqref{eq: tech_lemma_ineq1} to deduce the first result.

For the second result, we have  (using general probabilities $p^{(i)}$ in the expectations $\mathbb{E}_k$) that
\begin{align*}
\mathbb{E}_k\| y_{k+1} - y \|^2_{\Phi_k} &= \mathbb{E}_k \sum_{i=1}^n \phi_k^{(i)} (y_{k+1}^{(i)} - y^{(i)})^2  \\
& = \mathbb{E}_k \left[ \phi_k^{(i_k)} (\bar y_{k+1}^{(i_k)}- y^{(i_k)})^2 + \sum_{i\neq i_k} \phi_k^{(i)} (y_{k}^{(i)}- y^{(i)})^2\right] \\
&= \mathbb{E}_k \left[ \phi_k^{(i_k)} (\bar y_{k+1}^{(i_k)}- y^{(i_k)})^2 - \phi_k^{(i_k)} (y_k^{(i_k)}- y^{(i_k)})^2 + \sum_{i=1}^n \phi_k^{(i)} (y_{k}^{(i)}- y^{(i)})^2\right] \\
&= \sum_{i=1}^n p^{(i)} \phi_k^{(i)} (\bar y_{k+1}^{(i)}- y^{(i)})^2 - \sum_{i=1}^n p^{(i)} \phi_k^{(i)} (y_{k}^{(i)}- y^{(i)})^2 + \sum_{i=1}^n \phi_k^{(i)} (y_k^{(i)}- y^{(i)})^2 \\
&= \| \bar y_{k+1}-y \|^2_{\Phi_k P} + \| y_k-y \|^2_{\Phi_k (I - P)}.
\end{align*}
The final equality is derived by using the same derivation since $h^\ast(y)=\sum_{i=1}^n h_i^\ast(y^{(i)})$.
\end{proof}

The following technical result is used in the proof of Theorem~\ref{th: rates_for_part1}, immediately after \eqref{eq: end_of_std} and is based on the ideas from~\cite{song2021variance}. 
\begin{lemma}\label{lem: lemma_swd21_trick}[based on~\cite{song2021variance}]
Let $x_1, y_1$ be computed as~\eqref{eq: purecd_init} and let $\lambda_0 = (n-1)\lambda_1 = 1$, and for $k\geq 1$, $\lambda_{k+1} = \min\{ 1, \lambda_k n/(n-1)\}$.
For all $K \ge 1$, define (as in Theorem~\ref{th: rates_for_part1}) the following averaged iterates:
\[
\Lambda_K = \sum_{k=0}^{K-1} \lambda_k, \quad 
y^{K} = \frac{n\lambda_{K-1} y_{K} + \sum_{k=1}^{K-2} (n\lambda_k - (n-1)\lambda_{k+1})y_{k+1}}{\Lambda_K}, \quad
x^{K} = \frac{\sum_{k=0}^{K-1} \lambda_k \bar x_{k+1}}{\Lambda_K}.
\]
Then we have
\begin{multline*}
\sum_{k=1}^{K-1} \lambda_k G(\bar x_{k+1}, \bar y_{k+1}, x, y)  \geq \Lambda_K G(x^K, y^K, x, y) +\sum_{k=1}^{K-1} \mathcal{E}_{k, 3}(x) \\
+ \frac{n\max_i\|A_i\|}{2 \gamma } \Big(\| y-y_1\|^2 - \| y-y_0 \|^2 \Big)
 + \frac{n\max_i\|A_i\|}{2}  \| x-x_1\|^2 -n\max_i\|A_i\| \| x-x_0\|^2,
\end{multline*}
where
$\mathcal{E}_{k, 3}(x) = \lambda_k\left(- \langle x, A^\top \bar y_{k+1} \rangle +  \langle x, A^\top(ny_{k+1} - (n-1) y_k) \rangle + h^\ast(\bar y_{k+1}) - nh^\ast(y_{k+1}) + (n-1) h^\ast(y_k)\right)$, as in  \eqref{eq: err_term3}.
\end{lemma}
\begin{proof}
We have by using the definitions of $G$ (from~\eqref{eq: g_from_l}) and $\mathcal{E}_{k, 3}(x)$ (above) that 
\begin{align}
\nonumber
& \sum_{k=1}^{K-1} \lambda_k G(\bar x_{k+1}, \bar y_{k+1}, x, y) \\
\nonumber
& =\sum_{k=1}^{K-1} \lambda_k \left( g(\bar x_{k+1}) + \langle A\bar x_{k+1}, y \rangle - g(x) \right) + \lambda_k \left( h^\ast(\bar y_{k+1}) - \langle Ax, \bar y_{k+1} \rangle - h^\ast(y) \right)\\
\nonumber
& =\sum_{k=1}^{K-1} \lambda_k \left( g(\bar x_{k+1}) + \langle A\bar x_{k+1}, y \rangle - g(x) \right) \\
& \quad + \sum_{k=1}^{K-1} [n\lambda_k h^\ast(y_{k+1}) - (n-1) \lambda_k h^\ast(y_k) - \langle Ax, n\lambda_k y_{k+1} - (n-1)\lambda_k y_k \rangle - \lambda_k h^\ast(y) + \mathcal{E}_{k, 3}(x)].
\label{eq: general_trick_lhs}
\end{align}
Next, we work with two summations on the RHS of \eqref{eq: general_trick_lhs}. As in ~\cite{song2021variance}, we have
\begin{align*}
\sum_{k=1}^{K-1} [n\lambda_k h^\ast(y_{k+1}) - (n-1) \lambda_k h^\ast(y_k)] &= n\lambda_{K-1} h^\ast(y_{K}) + \sum_{k=1}^{K-2} n\lambda_k h^\ast(y_{k+1}) - \sum_{k=1}^{K-1} (n-1)\lambda_k h^\ast(y_k) \\
&= n\lambda_{K-1} h^\ast(y_{K}) + \sum_{k=1}^{K-2} n\lambda_k h^\ast(y_{k+1}) - \sum_{k=0}^{K-2} (n-1)\lambda_{k+1} h^\ast(y_{k+1}) \\
&= n\lambda_{K-1} h^\ast(y_{K}) + \sum_{k=1}^{K-2} (n\lambda_k - (n-1)\lambda_{k+1}) h^\ast(y_{k+1}) -(n-1)\lambda_1 h^\ast(y_1).
\end{align*}
Developing this estimate further, we note that by definition of $\{ \lambda_k \}$, we have  $n\lambda_k - (n-1)\lambda_{k+1} \geq 0$ for $k\geq 1$.
Moreover, since $n\lambda_1 = \lambda_1+ \lambda_0$, we have $\Lambda_K = \sum_{k=0}^{K-1} \lambda_k =n\lambda_1 + \sum_{k=2}^{K-1} \lambda_k = n \lambda_{K-1} + \sum_{k=1}^{K-2} (n\lambda_k-(n-1) \lambda_{k+1})$.
Hence, we can use the definitions of $\Lambda_K$ and $y^K$ and convexity of $h^\ast$ to deduce that
\begin{equation}\label{eq: ke3}
    \sum_{k=1}^{K-1} [n\lambda_k h^\ast(y_{k+1}) - (n-1) \lambda_k h^\ast(y_k)] \geq \Lambda_K h^\ast(y^K) - (n-1)\lambda_1 h^\ast(y_1).
\end{equation}
By the same derivation, we have 
\begin{align*}
\sum_{k=1}^{K-1} \langle Ax, n\lambda_k y_{k+1}-(n-1)\lambda_{k} y_k \rangle = n\lambda_{K-1} \langle Ax, y_{K}\rangle + \sum_{k=1}^{K-2} \langle Ax, (n\lambda_k - (n-1)\lambda_{k+1})y_{k+1} \rangle - (n-1)\lambda_1\langle Ax, y_1 \rangle.
\end{align*}
By using the definitions of $\Lambda_K$ and $y^K$, we can further develop this estimate to obtain
\begin{align}
\sum_{k=1}^{K-1} \langle Ax, n\lambda_k y_{k+1}-(n-1)\lambda_{k} y_k \rangle &= \langle Ax, n\lambda_{K-1} y_{K} \rangle + \sum_{k=1}^{K-2} \langle Ax, (n\lambda_k - (n-1)\lambda_{k+1})y_{k+1} \rangle  - (n-1)\lambda_1\langle Ax, y_1 \rangle \notag \\
&=\Lambda_K \langle Ax,y^K \rangle - (n-1)\lambda_1 \langle Ax,y_1 \rangle.\label{eq: ke4}
\end{align}
By using~\eqref{eq: ke3} and \eqref{eq: ke4} in~\eqref{eq: general_trick_lhs}, we have 
\begin{multline}\label{eq: se3}
\sum_{k=1}^{K-1} \lambda_k G(\bar x_{k+1}, \bar y_{k+1}, x, y) \ge \sum_{k=1}^{K-1} \lambda_k \left( g(\bar x_{k+1}) + \langle A\bar x_{k+1}, y \rangle - g(x) \right)   - \sum_{k=1}^{K-1} \lambda_k h^\ast(y) +\Lambda_K h^\ast(y^K) - \Lambda_K\langle Ax, y^K \rangle\\
 - (n-1) \lambda_1 h^\ast(y_1) + (n-1) \lambda_1 \langle Ax, y_1 \rangle + \sum_{k=1}^{K} \mathcal{E}_{k, 3}(x)
\end{multline}
We have by taking the negative of both sides in the first claim of \Cref{lem: lem_init}) that
\begin{multline*}
 -h^\ast(y_1) +  \langle Ax, y_1 \rangle \geq  -h^\ast(y) - g(x) +  g(x_1) +  \langle Ax_1, y \rangle + \frac{n\max_i\|A_i\|}{2 \gamma } \Big(\| y-y_1\|^2 - \| y-y_0 \|^2 \Big)\\
 + \frac{n\max_i\|A_i\|}{2}  \| x-x_1\|^2 -n\max_i\|A_i\| \| x-x_0\|^2.
\end{multline*}
By using this inequality together with  $\bar x_1 = x_1$ and $1=\lambda_1(n-1)=\lambda_0$ in~\eqref{eq: se3}, we obtain
\begin{multline*}
\sum_{k=1}^{K-1} \lambda_k G(\bar x_{k+1}, \bar y_{k+1}, x, y) \geq \sum_{k=0}^{K-1} \lambda_k \left( g(\bar x_{k+1}) + \langle A\bar x_{k+1}, y \rangle - g(x) \right)  - \sum_{k=0}^{K-1} \lambda_k h^\ast(y) +\Lambda_K h^\ast(y^K) - \Lambda_K\langle Ax, y^K \rangle \\
 + \sum_{k=1}^{K-1} \mathcal{E}_{k, 3}(x)
 + \frac{n\max_i\|A_i\|}{2 \gamma } \Big(\| y-y_1\|^2 - \| y-y_0 \|^2 \Big)
 + \frac{n\max_i\|A_i\|}{2}  \| x-x_1\|^2 -n\max_i\|A_i\| \| x-x_0\|^2.
\end{multline*}
In particular, after recalling $\Lambda_K = \sum_{k=0}^{K-1} \lambda_k$, with the definition of $x^K$ in this lemma,
we have by convexity of $g$ that
\begin{multline}\label{eq: final_lb_gap}
\sum_{k=1}^{K-1} \lambda_k G(\bar x_{k+1}, \bar y_{k+1}, x, y)  \geq \Lambda_K \left[ g(x^K) + \langle Ax^K, y \rangle - g(x) - h^\ast(y) + h^\ast(y^K) - \langle Ax, y^K \rangle \right] +\sum_{k=1}^{K-1} \mathcal{E}_{k, 3}(x) \\
+ \frac{n\max_i\|A_i\|}{2 \gamma } \Big(\| y-y_1\|^2 - \| y-y_0 \|^2 \Big)
 + \frac{n\max_i\|A_i\|}{2}  \| x-x_1\|^2 -n\max_i\|A_i\| \| x-x_0\|^2.
\end{multline}
We identify $G (x^K,y^K,x,y)$ (see~\eqref{eq: g_from_l}) in the RHS of this inequality to complete the proof.
\end{proof}

\begin{lemma}[\cite{chambolle2011first}]\label{lem: acc_rate_seq}
Let us define a positive sequence $\alpha_k$ for $k\geq 0$ such that $\alpha_0 \leq 1$ and $\alpha_{k+1} = \frac{\alpha_k}{\sqrt{1+\alpha_k}}$. Then $\alpha_K \leq 3/K$ for all $K>0$.
\end{lemma}
\begin{proof}
Note first that $\{\alpha_k\}$ is a positive decreasing sequence.
For any $k \ge 0$, we have $\alpha_{k+1} = \frac{\alpha_k}{\sqrt{1+\alpha_k}}$ and thus for $s_k = \alpha_k^{-1}$, we have $\frac{1}{\alpha_{k+1}} = \frac{\sqrt{1+\alpha_k}}{\alpha_k}  \iff s_{k+1} = s_k \sqrt{1+\frac{1}{s_k}}=\sqrt{s_k^2 + s_k}=(s_k+1/2)\sqrt{1-\frac{1/4}{(s_k+1/2)^2}}$ and for $s_k \geq 1$ which is due to $\alpha_k \leq 1$, we use $\sqrt{1-t} \geq 1-t$ for $t\in [0, 1]$ to get $s_{k+1} \geq (s_k+1/2)\left(1-\frac{1/4}{(s_k+1/2)^2}\right) = s_k + \frac{1}{2}-\frac{1/4}{s_k+1/2} \geq s_k + \frac{1}{2}-\frac{1/4}{1+1/2}= s_k + \frac{1}{3}$ (again using $s_k\geq 1$). 
Since this bound holds for all $k=0,1,\dotsc$, we have  $s_k\geq s_0 + \frac{k}{3} \geq \frac{k}{3}$ due to $s_0\geq 1$.
Since $s_{k} = \alpha_{k}^{-1} \geq \frac{k}{3}$, the result follows.
\end{proof}

\section{Conversion of Results for ERM in~\Cref{th: rates_for_part1}}\label{app: upper_lower_bd}
We restate the problem considered in this paper as follows:
\begin{equation}\label{eq:gr3}
    \min_{x} \max_y \sum_{i=1}^n \langle A_ix, y^{(i)} \rangle - h_i^\ast(y^{(i)}) + g(x) = \min_x \sum_{i=1}^n h_i(A_i x) + g(x),
\end{equation}
where $A_i$ is a row vector of length $d$.
In \Cref{th: rates_for_part1} and other results in the paper, we measure the complexity bounds of the algorithms in terms of their dependence on $D_x = \max_{x\in \dom g} \|x\|$, $D_y = \max_{y\in\dom h^\ast} \|y\|$ and $\max_i \| A_i \|$, where $h^\ast(y) = \sum_{i=1}^n h_i^\ast(y^{(i)})$.

In ERM, we often write the formulation with a feature vector $a_i \in \mathbb{R}^d$ and,
\begin{equation}\label{eq: ml_prob}
    \min_{x} \frac{1}{n} \sum_{i=1}^n f_i(\langle a_i, x \rangle) + g(x),
\end{equation}
and we measure the complexity bounds of the algorithms in terms of their dependence on $D_x = \max_{x\in\dom g} \|x\|$, $L_f$: Lipschitz constant of $f_i$ and $ \max_i \| a_i \|_2$.
Of course, \eqref{eq: ml_prob} has the form \eqref{eq:gr3}, but the dependence of the bounds in different quantities can make it hard to compare results in different papers, because they use one or other of there formulations. 

First, note that we can map~\eqref{eq: ml_prob} into~\eqref{eq:gr3} by setting
\begin{equation}\label{eq: app2_mapping}
    h_i \leftarrow \frac{1}{n} f_i \text{~~~~and~~~~} A_i \leftarrow a_i^\top.
\end{equation}
 Let us use \cite[Corollary 17.19]{bauschke2011convex} to state for proper, convex, lsc $h_i:\mathbb{R}\to\mathbb{R}$ that
\begin{equation*}
    h_i \text{ is } \bar L_{h}\text{-Lipschitz} \iff |v| \leq \bar L_{h}, \text{ for all } v \in \dom h^\ast_i.
\end{equation*}
Therefore, we have (see~\eqref{eq:gr3})
\begin{equation}\label{eq:wm3}
D_y = \max_{y\in\dom h^\ast} \| y\| \leq \sqrt{n} \bar L_{h}.    
\end{equation}
The upper bounds we derived for~\eqref{eq:gr3} in Thm.~\ref{th: rates_for_part1},~\ref{th: rates_for_part1.2} are of the form
\begin{equation}\label{eq:hw3}
    \frac{1}{K}n\max_i\|A_i\|\left( D_x^2 + D_y^2 \right).
\end{equation}
Note that this bound is slightly different from \Cref{th: rates_for_part1}.
The reason is that \Cref{th: rates_for_part1} assumes that $h$ is $L_h$-Lipschitz and uses this constant in the bounds instead of $D_y$.
Here, we have that $f_i$ is $L_f$-Lipschitz
and we want to write the bounds in terms of $L_f$ to compare with existing bounds in~\cite{allen2017katyusha,zhang2015stochastic,tan2020accelerated}.

We now show how to  write this bound in an alternative way.
Note that the upper bound given in~\eqref{eq:hw3} normally arises from an inequality of the form
\begin{equation*}
    \frac{1}{K}\left(\frac{1}{\tau} D_x^2 + \frac{n}{\sigma} D_y^2\right),
\end{equation*}
and the step sizes are chosen so that $\tau \sigma n\max_i\|A_i\| \leq 1$.
Normally, we pick $\tau = \frac{1}{n\max_i\|A_i\|}$ and $\sigma = \frac{1}{\max_i\|A_i\|}$, resulting in the upper bound~\eqref{eq:hw3}.
However, with the knowledge of $D_x, D_y$, we can also pick $\tau = \frac{D_y}{n D_x \max_i\|A_i\|}$ and $\sigma = \frac{D_x}{D_y\max_i\|A_i\|}$ which still satisfies the step size rule, since we scaled the step sizes in the same way, and the upper bound in~\eqref{eq:hw3} can be alternatively written as
\begin{equation*}
    \frac{n\max_i\|A_i\|}{K} D_x D_y.
\end{equation*}
Using~\eqref{eq: app2_mapping} and~\eqref{eq:wm3}, our last bound becomes
\begin{equation}
    \frac{1}{K}n\max_i\|a_i\|D_x \sqrt{n}\bar L_{h}.
\end{equation}
Since $f_i$ is $L_f$-Lipschitz and $h_i = \frac{1}{n}f_i$, we have $\bar L_{h} = \frac{L_f}{n}$, so this bound becomes
\begin{equation}
    \frac{\sqrt{n}L_f\max_i\|a_i\|D_x}{K}.
\end{equation}
This bound gives the complexity 
\begin{equation*}
    \frac{ \sqrt{n}L_f\max_i\|a_i\|D_x}{\varepsilon},
\end{equation*}
for solving ERM problem in~\eqref{eq: ml_prob} to $\varepsilon$-accuracy.
We can now compare it directly with the existing results~\cite{allen2017katyusha,zhang2015stochastic,tan2020accelerated}.

\bibliographystyle{plain}
\bibliography{main}

\begin{thebibliography}{10}

\bibitem{alacaoglu2019convergence}
Ahmet Alacaoglu, Olivier Fercoq, and Volkan Cevher.
\newblock On the convergence of stochastic primal-dual hybrid gradient.
\newblock {\em arXiv preprint arXiv:1911.00799}, 2019.

\bibitem{alacaoglu2020random}
Ahmet Alacaoglu, Olivier Fercoq, and Volkan Cevher.
\newblock Random extrapolation for primal-dual coordinate descent.
\newblock In {\em International Conference on Machine Learning}, pages
  191--201. PMLR, 2020.

\bibitem{alacaoglu2021stochastic}
Ahmet Alacaoglu and Yura Malitsky.
\newblock Stochastic variance reduction for variational inequality methods.
\newblock {\em arXiv preprint arXiv:2102.08352}, 2021.

\bibitem{allen2017katyusha}
Zeyuan Allen-Zhu.
\newblock Katyusha: The first direct acceleration of stochastic gradient
  methods.
\newblock {\em The Journal of Machine Learning Research}, 18(1):8194--8244,
  2017.

\bibitem{bauschke2011convex}
Heinz~H Bauschke and Patrick~L Combettes.
\newblock {\em Convex Analysis and Monotone Operator Theory in Hilbert Spaces},
  volume 408.
\newblock Springer, 2011.

\bibitem{carmon2019variance}
Yair Carmon, Yujia Jin, Aaron Sidford, and Kevin Tian.
\newblock Variance reduction for matrix games.
\newblock {\em Advances in Neural Information Processing Systems}, 2019.

\bibitem{carmon2020coordinate}
Yair Carmon, Yujia Jin, Aaron Sidford, and Kevin Tian.
\newblock Coordinate methods for matrix games.
\newblock In {\em 2020 IEEE 61st Annual Symposium on Foundations of Computer
  Science (FOCS)}, pages 283--293. IEEE, 2020.

\bibitem{chambolle2018stochastic}
Antonin Chambolle, Matthias~J Ehrhardt, Peter Richt{\'a}rik, and Carola-Bibiane
  Schonlieb.
\newblock Stochastic primal-dual hybrid gradient algorithm with arbitrary
  sampling and imaging applications.
\newblock {\em SIAM Journal on Optimization}, 28(4):2783--2808, 2018.

\bibitem{chambolle2011first}
Antonin Chambolle and Thomas Pock.
\newblock A first-order primal-dual algorithm for convex problems with
  applications to imaging.
\newblock {\em Journal of Mathematical Imaging and Vision}, 40(1):120--145,
  2011.

\bibitem{chambolle2016ergodic}
Antonin Chambolle and Thomas Pock.
\newblock On the ergodic convergence rates of a first-order primal--dual
  algorithm.
\newblock {\em Mathematical Programming}, 159(1):253--287, 2016.

\bibitem{fercoq2019coordinate}
Olivier Fercoq and Pascal Bianchi.
\newblock A coordinate-descent primal-dual algorithm with large step size and
  possibly nonseparable functions.
\newblock {\em SIAM Journal on Optimization}, 29(1):100--134, 2019.

\bibitem{lan2020first}
Guanghui Lan.
\newblock {\em First-order and Stochastic Optimization Methods for Machine
  Learning}.
\newblock Springer Nature, 2020.

\bibitem{nemirovski2009robust}
Arkadi Nemirovski, Anatoli Juditsky, Guanghui Lan, and Alexander Shapiro.
\newblock Robust stochastic approximation approach to stochastic programming.
\newblock {\em SIAM Journal on Optimization}, 19(4):1574--1609, 2009.

\bibitem{nesterov2007dual}
Yurii Nesterov.
\newblock Dual extrapolation and its applications to solving variational
  inequalities and related problems.
\newblock {\em Mathematical Programming}, 109(2):319--344, 2007.

\bibitem{nesterov2009primal}
Yurii Nesterov.
\newblock Primal-dual subgradient methods for convex problems.
\newblock {\em Mathematical programming}, 120(1):221--259, 2009.

\bibitem{nesterov2012efficiency}
Yurii Nesterov.
\newblock Efficiency of coordinate descent methods on huge-scale optimization
  problems.
\newblock {\em SIAM Journal on Optimization}, 22(2):341--362, 2012.

\bibitem{rockafellar1970convex}
R~Tyrrell Rockafellar.
\newblock {\em Convex Analysis}.
\newblock Princeton University Press, 1970.

\bibitem{shalev2013stochastic}
Shai Shalev-Shwartz and Tong Zhang.
\newblock Stochastic dual coordinate ascent methods for regularized loss
  minimization.
\newblock {\em Journal of Machine Learning Research}, 14(2), 2013.

\bibitem{shalev2014accelerated}
Shai Shalev-Shwartz and Tong Zhang.
\newblock Accelerated proximal stochastic dual coordinate ascent for
  regularized loss minimization.
\newblock In {\em International Conference on Machine Learning}, pages 64--72.
  PMLR, 2014.

\bibitem{song2021coordinate}
Chaobing Song, Cheuk~Yin Lin, Stephen~J Wright, and Jelena Diakonikolas.
\newblock Coordinate linear variance reduction for generalized linear
  programming.
\newblock {\em arXiv preprint arXiv:2111.01842}, 2021.

\bibitem{song2021variance}
Chaobing Song, Stephen~J Wright, and Jelena Diakonikolas.
\newblock Variance reduction via primal-dual accelerated dual averaging for
  nonsmooth convex finite-sums.
\newblock In {\em International Conference on Machine Learning}, pages
  9824--9834. PMLR, 2021.

\bibitem{tan2020accelerated}
Conghui Tan, Yuqiu Qian, Shiqian Ma, and Tong Zhang.
\newblock Accelerated dual-averaging primal--dual method for composite convex
  minimization.
\newblock {\em Optimization Methods and Software}, 35(4):741--766, 2020.

\bibitem{tran2018smooth}
Quoc Tran-Dinh, Olivier Fercoq, and Volkan Cevher.
\newblock A smooth primal-dual optimization framework for nonsmooth composite
  convex minimization.
\newblock {\em SIAM Journal on Optimization}, 28(1):96--134, 2018.

\bibitem{wright2015coordinate}
Stephen~J Wright.
\newblock Coordinate descent algorithms.
\newblock {\em Mathematical Programming}, 151(1):3--34, 2015.

\bibitem{xiao2014proximal}
Lin Xiao and Tong Zhang.
\newblock A proximal stochastic gradient method with progressive variance
  reduction.
\newblock {\em SIAM Journal on Optimization}, 24(4):2057--2075, 2014.

\bibitem{zhang2015stochastic}
Yuchen Zhang and Xiao Lin.
\newblock Stochastic primal-dual coordinate method for regularized empirical
  risk minimization.
\newblock In {\em International Conference on Machine Learning}, pages
  353--361. PMLR, 2015.

\end{thebibliography}

\end{document}